\documentclass[leqno, 12pt]{amsart}

\usepackage{seqsplit}
\usepackage[english]{babel}
\usepackage[utf8]{inputenc}
\usepackage{lineno}
\usepackage{amsmath}
\usepackage{amssymb}
\usepackage{amsthm}
\usepackage{thmtools}
\usepackage{graphicx}
\usepackage{subfig}
\usepackage{caption}
\usepackage{hyperref}
\hypersetup{colorlinks=true, citecolor=darkblue, linkcolor=darkblue}
\usepackage{float}
\usepackage[table]{xcolor}
\usepackage{xcolor}
\usepackage{xargs}
\usepackage{a4wide}
\usepackage[noabbrev,capitalise]{cleveref}
\usepackage[export]{adjustbox}
\usepackage{enumerate}

\numberwithin{equation}{section}
\newtheorem{theorem}{Theorem}[section]
\newtheorem{proposition}[theorem]{Proposition}
\newtheorem{definition}[theorem]{Definition}
\newtheorem{lemma}[theorem]{Lemma}
\newtheorem{remark}[theorem]{Remark}
\newtheorem{corollary}[theorem]{Corollary}
\newtheorem{example}[theorem]{Example}

\newcommand{\N}{\mathbb N}
\newcommand{\C}{\mathbb C}
\renewcommand{\L}{\mathbf L}
\newcommand{\K}{\mathbf K}
\renewcommand{\P}{\mathcal P}
\newcommand{\Cc}{\mathcal C}
\newcommand{\Ss}{\mathfrak S}
\newcommand{\ASM}{\mathbf{ASM}}
\newcommand{\HFM}{\mathbf{HFM}}
\newcommand{\NCP}{\mathbf{NCP}}
\newcommand{\Cat}{\mathbf{Cat}}
\newcommand{\Dyck}{\mathbf{Dyck}}

\newcommand{\DPoset}{\mathcal D}
\newcommand{\TPoset}{\mathcal T}
\newcommand{\JIrr}{\mathcal J}
\newcommand{\MIrr}{\mathcal M}
\newcommand{\CT}{\operatorname {CT}}
\newcommand{\BPD}{\operatorname {BPD}}
\newcommand{\BCPD}{\operatorname {BCPD}}
\newcommand{\bcc}{\operatorname {bcc}}
\newcommand{\Upper}{\mathbf{Up}}
\newcommand{\Lower}{\mathbf{Lo}}
\newcommand{\JIdeal}{\operatorname J}
\newcommand{\jAnti}{\operatorname j}
\newcommand{\MIdeal}{\operatorname M}
\newcommand{\mAnti}{\operatorname m}
\newcommand{\Nappe}{\operatorname N}
\newcommand{\Epos}{\operatorname{E}_{\mathrm{pos}}} 
\newcommand{\Eval}{\operatorname{E}_{\mathrm{val}}} 
\newcommand{\TL}{\operatorname {TL}}

\newcommand{\Sym}{\operatorname{Sym}}
\newcommand{\QSym}{\operatorname{QSym}}

\definecolor{darkblue}{rgb}{0,0,0.7} 
\newcommand{\darkblue}{\color{darkblue}} 
\newcommand{\defn}[1]{\textsl{\darkblue #1}} 
\newcommand{\set}[2]{\left\{ #1 \;\middle|\; #2 \right\}} 
\renewcommand{\b}[1]{\boldsymbol{#1}} 
\newcommand{\dotprod}[2]{\langle #1 | #2 \rangle} 
\newcommand{\ie}{\textit{i.e.}~} 
\newcommand{\eg}{\textit{e.g.}~} 
\newcommand{\eqdef}{\mbox{\,\raisebox{0.2ex}{\scriptsize\ensuremath{\mathrm:}}\ensuremath{=}\,}} 
\newcommand{\ssm}{\smallsetminus} 
\newcommand{\precdot}{\prec\mathrel{\mkern-5mu}\mathrel{\cdot}}

\newcommand{\cross}[1][black]{\raisebox{-.15cm}{\includegraphics[scale=.9]{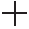}}}
\newcommand{\elbow}[1][black]{\raisebox{-.15cm}{\includegraphics[scale=.9]{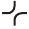}}}

\usepackage{todonotes}

\makeatletter
\def\l@section{\@tocline{1}{1pt}{0pc}{}{}}
\def\l@part{\@tocline{1}{7pt}{0pc}{}{}}
\makeatother
\let\oldtocpart=\tocpart
\renewcommand{\tocpart}[2]{\sc\large\oldtocpart{#1}{#2}}
\let\oldtocsection=\tocsection
\renewcommand{\tocsection}[2]{\bf\oldtocsection{#1}{#2}}
\let\oldtocsubsubsection=\tocsubsubsection
\renewcommand{\tocsubsubsection}[2]{\quad\oldtocsubsubsection{#1}{#2}}

\title[Heaps of dodecahedra, Catalan congruences on ASMs, and TL algebra bases]{Heaps of rhombic dodecahedra, \\ catalan congruences on alternating sign matrices, \\ and bases of the Temperley--Lieb algebra}

\thanks{FH was partially supported by the European projet FRESCO (European Union’s Horizon 2020 research and innovation program grant No.101001995) and by the French--Austrian project PAGCAP (ANR-21-CE48-0020 \& FWF I 5788).
VP was partially supported by the Spanish project PID2022-137283NB-C21 of MCIN/AEI/10.13039/501100011033 / FEDER, UE, by the Spanish--German project COMPOTE (AEI PCI2024-155081-2 \& DFG 541393733), by the Severo Ochoa and María de Maeztu Program for Centers and Units of Excellence in R\&D (CEX2020-001084-M), by the Departament de Recerca i Universitats de la Generalitat de Catalunya (2021 SGR 00697), and by the French--Austrian project PAGCAP (ANR-21-CE48-0020 \& FWF I 5788).}

\author{Florent Hivert}
\address[Florent Hivert]{LISN (UMR CNRS 9015), Université Paris Saclay, INRIA, CNRS}
\email{florent.hivert@lisn.fr}
\urladdr{\url{http://www.lisn.fr/~hivert/}}

\author{Vincent Pilaud}
\address[Vincent Pilaud]{Universitat de Barcelona \& Centre de Recerca Matemàtica, Barcelona}
\email{vincent.pilaud@ub.edu}
\urladdr{\url{https://www.ub.edu/comb/vincentpilaud/}}

\author{Ludovic Schwob}
\address[Ludovic Schwob]{LIGM, Université Gustave Eiffel, Champs-sur-Marne}
\email{ludovic.schwob@edu.univ-eiffel.fr}
\urladdr{\url{https://igm.univ-mlv.fr/~schwob/}}

\begin{document}

\vspace*{.6cm}

\maketitle

\begin{abstract}
We prove that the excedance relation on permutations defined by N.~Bergeron and L.~Gagnon actually extends to a congruence of the lattice on alternating sign matrices.
Motivated by this example, we study all lattice congruences of the lattice on alternating sign matrices whose quotient is isomorphic to the Stanley lattice on Dyck paths, which we call catalan congruences.
We prove that the maxima of the congruence classes are always covexillary permutations (and all covexillary permutations appear this way), and that the minimal permutations in each class are always precisely the $321$-avoiding permutations.
Finally, we show that any choice of representative permutations in each congruence class yield a basis of the \mbox{Temperley--Lieb} algebra with parameter~$2$, vastly generalizing the bases arising from the excedance relation. 
\end{abstract}

\enlargethispage{-.5cm}
\tableofcontents


\enlargethispage{-.1cm}
\section{Introduction}
\label{sec:introduction}

It is well known that the quotient $\K[X]/\langle\Sym^+(X)\rangle$ of the polynomial ring on an alphabet~$X = \{x_1,\dots,x_n\}$ of $n$~variables by the ideal generated by symmetric polynomials with no constant terms is of dimension~$n!$ and is isomorphic to the regular representation of the symmetric group~$\Ss_n$.
Chevalley–Shephard–Todd's theorem asserts that this statement remains true when the symmetric group is replaced by any real or complex reflection group.
In the combinatorics community, a large body of work has been devoted to understanding this representation-theoretic structure in more explicit terms.
One notable outcome is A.~Lascoux’s theory of Schubert polynomials~\cite{lascouxschutzenberger87,macdonald.1991}, which provides a remarkably elegant basis for the above quotient.
This perspective also connects to the celebrated $n!$ conjecture~\cite{Garsia1993,Haiman2001}, which extends the setting from one set of variables to two alphabets and gives rise to deep interactions between algebraic geometry, representation theory, and combinatorics.

There is another closely related case which, however, remains much less understood: the ring of quasi-symmetric polynomials.
Recall that these were introduced by I.~Gessel~\cite{Gessel.1984} as generating series encoding permutations with a prescribed descent set.
While many parallels with the theory of symmetric functions have been developed, both from combinatorial and representation-theoretic perspectives~\cite{KrobThibon.1997, Hivert2000}, the invariant-theoretic viewpoint has long remained elusive.
Nevertheless, there have been persistent indications that a structure analogous to the symmetric and coinvariant settings should exist.

On one hand, the first author proved that the ring of quasi-symmetric polynomials is invariant under a modified action of the symmetric group~\cite{Hivert2000}.
In contrast to the classical situation, this action is not faithful and factors through the Temperley--Lieb algebra~$\TL(2)$, whose dimension is the $n$-th Catalan number~$C_n$.
On the other hand, J.~C.~Aval and N.~Bergeron~\cite{AvalBergeron2003} noticed that the quotient $\K[X]/\langle\QSym^+(X)\rangle$ is also of dimension~$C_n$.
However, the connection between these two occurrences of Catalan numbers is far from straightforward.
The reason is that the Temperley–Lieb action is not multiplicative, and therefore does not descend to the quotient.
As a result, the space~$\K[X]/\langle\QSym^+(X)\rangle$ does not naturally inherit an action of~$\TL(2)$, making the relationship between the two Catalan structures subtle to uncover.

In a recent breakthrough~\cite{BergeronGagnon}, N.~Bergeron and L.~Gagnon constructed a particular set of~$C_n$ permutations whose projection in $\TL(2)$ gives a basis, and such that the graded ideal of the vanishing ideal is equal to~$\langle\QSym^+(X)\rangle$.
More precisely, they defined the \defn{excedance equivalence}~$\equiv$ on permutations of~$[n]$, where two permutations are excedance equivalent if they share the same sets of positions and of values of their excedances (see \cref{def:excedanceRelation,fig:excedanceRelation}).
We summarize their main results into the following statement.

\begin{theorem}[\cite{BergeronGagnon}]
\label{thm:main1}
For the excedance relation on permutations,
\begin{enumerate}[(i)]
\item the excedance classes are intervals of the Bruhat order, whose minimal elements are precisely the $321$-avoiding permutations, and whose maximal elements are some covexillary permutations already considered in~\cite{Zinno,GobetWilliams},
\item the poset quotient of the (strong) Bruhat order on permutations by the excedance relation is isomorphic to the Stanley lattice on Dyck paths,
\item any choice of representatives of the excedance classes yields a basis of the Temperley--Lieb algebra~$\TL(2)$,
\item the top degree homogeneous components of the vanishing ideal of the set of permutations which are maximal in their excedance classes and the positive degree quasisymmetric polynomials generate the same ideal.
\end{enumerate}
\end{theorem}

Our work starts from the observation that this excedance relation on the Bruhat order actually extends to a congruence of the lattice of alternating sign matrices~\cite{MillsRobbinsRumsey1, MillsRobbinsRumsey2, RobbinsRumsey} (which is the MacNeille completion of the Bruhat order~\cite{LascouxSchutzenberger_TreillisBasesCoxeter}).
The congruences of a lattice are the equivalence relations which respect the lattice structure, and they come with rich combinatorial and algebraic properties.
Our main results (partially) extend the results of~\cite{BergeronGagnon} to all lattice congruences~$\equiv$ on alternating sign matrices whose quotient is isomorphic to the Stanley lattice, which we call \defn{catalan congruences}.

\begin{theorem}
\label{thm:main2}
For any catalan congruence~$\equiv$ of the lattice of alternating sign matrices,
\begin{enumerate}[(i)]
\item the maximal alternating sign matrices of the congruence classes of~$\equiv$ are covexillary permutation matrices,
\item the minimal permutation matrices in the congruence classes of~$\equiv$ are precisely the $321$-avoiding permutation matrices,
\item any choice of representative permutation matrices in each congruence class of~$\equiv$ yields a basis of the Temperley--Lieb algebra.
\end{enumerate}
\end{theorem}

\enlargethispage{-.1cm}
A few comments on this statement. 
In~(i), the covexillary permutation matrices that appear as maxima of congruence classes depend on the congruence.
However, any covexillary permutation matrix appears as the maxima of a class for at least one catalan congruence.
In~(ii), the minimal alternating sign matrix of each class is not necessarily a permutation (we actually discuss along the paper the specific congruences with this property).
But the permutations in each congruence class form an interval of the Bruhat order, whose minimum is $321$-avoiding and whose maximum is covexillary, and (by counting) all $321$ permutations appear as minimal permutations of congruence classes.
Finally, we note that the problem to extend \cref{thm:main1}\,(iv) to any catalan congruence is under investigation.

Our approach to \cref{thm:main2} is a combination of lattice and geometric perspectives on alternating sign matrices.
It is well known that the lattice of alternating sign matrices is distributive, hence isomorphic to the inclusion lattice of lower sets of its join irreducible subposet by the fundamental theorem for distributive lattices.
As already observed and exploited in~\cite{ElkiesKuperbergLarsenPropp, Propp01, Str11}, the join irreducible poset of the lattice of alternating sign matrices can be seen geometrically as a poset on the integer points inside the $(n-2)$th dilate of a $3$-dimensional tetrahedron resting on an edge (see \cref{fig:posets}).
Pushing slightly further this geometric perspective in \cref{subsec:ASMs}, we interpret a lower set in this tetrahedral poset as a stack of rhombic tetrahedra (see \cref{fig:ASM5}). 
The upper hull of this stack is the colored polygonal surface separating the lower set from its complementary upper set (see \cref{fig:ASM2}).
Projecting this surface to the ground (see \cref{fig:ASM3}) naturally leads to interpret geometrically the classical combinatorial models in bijection with alternating sign matrices (\eg height functions, aztec diamond~\cite{ElkiesKuperbergLarsenPropp}, six-vertex configurations~\cite{Kuperberg}, osculating path configurations~\cite{Behrend}, bumpless pipe dreams~\cite{Weigandt}, fully packed loops~\cite{BatchelorBloteNienhuisYung, deGier}), and to cook up some seemingly new models (\eg cliff configurations).
We note that this perspective also naturally leads to higher dimensional analogues of alternating matrices which will deserve further investigation.

The quotients of a distributive lattice~$\L$ are distributive lattices, whose join irreducible posets are induced subposets of the join irreducible poset of~$\L$.
Understanding quotients of the lattice of alternating sign matrices isomorphic to the Stanley lattice on Dyck paths thus boils down to finding subposets of the tetrahedron poset isomorphic to the triangular poset.
Simple geometric arguments show that these triangular subposets of the tetrahedron are encoded by certain depth triangles (see \cref{subsec:depthTriangles}), which are in natural bijections with doubly gapless Gelfand--Tsetlin patterns with bottom row~$12 \dots (n-1)$ (see \cref{subsec:catalanTriangles}), and with bicolored pipe dreams with $n-2$ pipes (see \cref{subsec:bicoloredPipeDreams}).
Interestingly, these depth triangles are themselves endowed with a natural lattice structure (see \cref{fig:catalanLattice}), which we exploit to show \cref{thm:main2} in \cref{sec:maxima,,sec:minimalPermutations,,sec:TL}.

Finally, we introduce in \cref{sec:Pn/Sn} a general operation of symmetrization of posets which specializes to the join irreducible posets of the lattices considered in this paper.
We show in particular that the poset resulting of this symmetrization can be described by inequalities indexed by upper sets of the original poset.


\section{Two distributive lattices}
\label{sec:distributiveLattices}


\subsection{Meet and join representations in distributive lattices}
\label{subsec:meetJoinRepresentationsDistributiveLattices}

We first recall some classical facts about distributive lattices.
A \defn{lower} (resp.~\defn{upper}) \defn{set} of a poset~$\P$ is a subset~$X$ of~$\P$ such that~$x \in X$ and~$x \ge y$ (resp.~$x \le y$) implies~$y \in X$.
A \defn{lattice}~$\L$ is a poset where any two elements~$x,y \in \L$ admits a \defn{join}~$x \vee y$ (least upper bound) and a \defn{meet}~$x \wedge y$ (greatest lower bound).
It is \defn{distributive} if~$x \vee (y \wedge z) = (x \vee y) \wedge (x \vee z)$ and~$x \wedge (y \vee z) = (x \wedge y) \vee (x \wedge z)$ for all~$x,y,z \in \L$ (in fact, these two conditions are equivalent).
An element~$x \in \L$ is \defn{join} (resp.~\defn{meet}) \defn{irreducible} if it covers (resp.~is covered by) a single element~$x_\star$ (resp.~$x^\star$).
We denote by~$\JIrr(\L)$ (resp.~$\MIrr(\L)$) the subposet of~$\L$ induced by its join (resp.~meet) irreducible elements.
The following classical result is known as the fundamental theorem for distributive lattices.

\begin{theorem}[Birkhoff]
\label{thm:ftfdl}
~
\begin{itemize}
\item The set~$\Lower(\P)$ (resp.~$\Upper(\P)$) of lower (resp.~upper) sets of any poset~$\P$ ordered by inclusion forms a distributive lattice.
\item Conversely, any distributive lattice~$\L$ is isomorphic to the inclusion poset of lower (resp.~upper) sets of its join (resp.~meet) irreducible poset~$\JIrr(\L)$ (resp.~$\MIrr(\L)$).
\end{itemize}
\end{theorem}

For~$x \in \L$, we denote by~$\JIdeal(x) \eqdef \set{j \in \JIrr(\L)}{j \le x}$ the corresponding lower set of~$\JIrr(\L)$, and by~$\jAnti(x) \eqdef \max(\JIdeal(x))$ the corresponding antichain of~$\JIrr(\L)$.
We define dually \linebreak $\MIdeal(x) \eqdef \set{m \in \MIrr(\L)}{x \le m}$ and $\mAnti(x) \eqdef \min(\MIdeal(x))$.

Note that there is a canonical isomorphism between the posets $\JIrr(\L)$ and $\MIrr(\L)$ defined by~$\kappa(j) = \bigvee\set{x \in \L}{x \wedge j = j_\star}$.
It has the property that~$\MIdeal(x) = \kappa(\JIrr(\L) \ssm \JIdeal(x))$.
In other words, an element of~$\L$ can be thought of as a \defn{cut} separating a lower set from an upper set in the poset~$\JIrr(\L) \simeq \MIrr(\L)$, 

For instance, the chain with $n+1$ elements is (isomorphic to) the lattice of lower (resp.~upper) sets of a chain~$1-2-\dots-n$, and the map~$\kappa$ is given by~$\kappa(i) = i-1$.

 
\subsection{Dyck paths}
\label{subsec:DyckPaths}

A \defn{Dyck path} of semilength~$n$ is a path with up steps~$(1,1)$ and down steps~$(1,-1)$, starting at~$(0,0)$, ending at~$(2n,0)$, and never passing strictly below the horizontal axis.
They define the \defn{Dyck path lattice}~$\Dyck_n$ (also known as the \defn{Stanley lattice}) with the relation~$P \le Q$ if~$P$ stays below~$Q$.
See \cref{fig:Dyck3ASM3}\,(left).
The join (resp.~meet) irreducible Dyck paths are those with a single peak (resp.~valley), hence the poset~$\JIrr(\Dyck_n)$ (resp.~$\MIrr(\Dyck_n)$) is isomorphic to the triangular poset~$\DPoset_n$ of triples $(y_1,y_2,y_3)$ with $y_1+y_2+y_3 =n-2$, ordered by $(y_1,y_2,y_3) \le (z_1,z_2,z_3)$ if and only if $y_1+y_2\le z_1+z_2$ and $y_2+y_3 \le z_2+z_3$.
See \cref{fig:posets}\,(left).
(Note that this poset could be encoded using two coordinates. However, we prefer to work with barycentric coordinates throughout the paper, as they make the presentation more symmetric.)

\begin{figure}
	\centering
	\includegraphics[scale=.6]{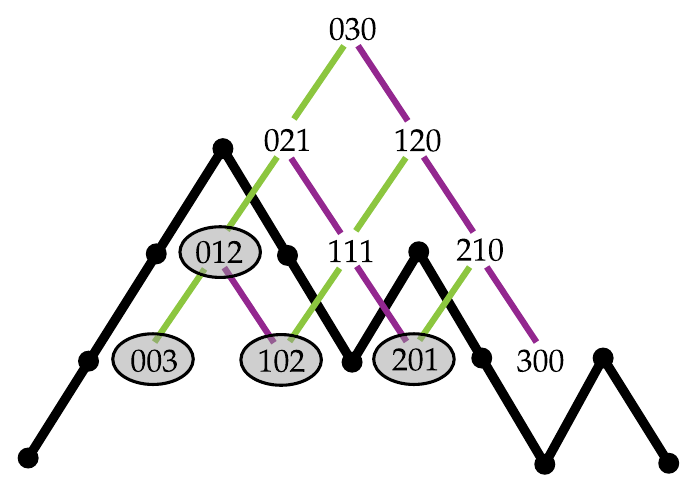}
	\caption{A lower set of $\DPoset_5$  and the corresponding Dyck path of semilength~$5$.}
	\label{fig:dyckpath}
\end{figure}

\begin{figure}
	\centerline{
	\includegraphics[scale=.7]{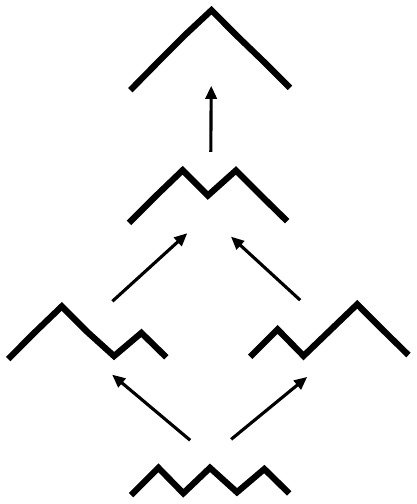}
	\qquad\qquad
	\includegraphics[scale=.7]{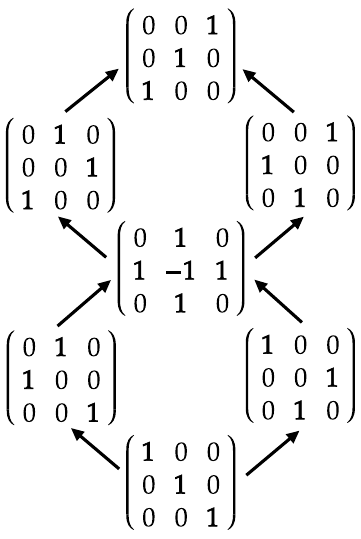}
	}
	\caption{The Stanley lattice~$\Dyck_3$ (left) and the ASM lattice~$\ASM_3$ (right).}
	\label{fig:Dyck3ASM3}
\end{figure}

\begin{figure}
	\centering
	\includegraphics[scale=.7]{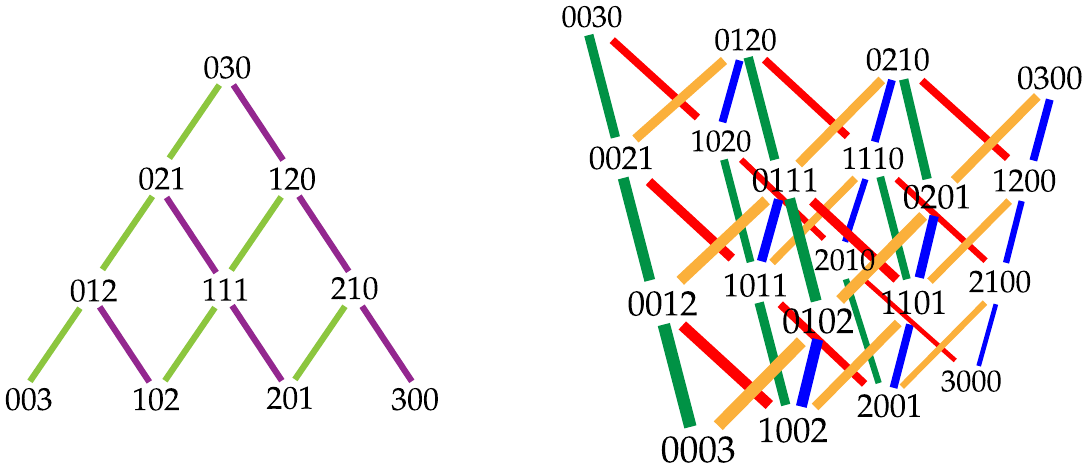}
	\caption{The posets $\DPoset_5$ (left) and $\TPoset_5$ (right).}
	\label{fig:posets}
\end{figure}


\subsection{Alternating sign matrices}
\label{subsec:ASMs}

We closely follow the presentation of~\cite{Propp01}.
For~${n \ge 1}$, one can define three families of matrices, illustrated in \cref{fig:ASM1}:
\begin{itemize}
\item An \defn{alternating sign matrix} (ASM) of order~$n$ is an $(n \times n)$-matrix~$A = (A_{i,j})_{i,j \in [n]}$ with entries~$-1$, $0$ or~$1$, with row and column sums equal to~$1$, and such that nonzero entries alternate between~$1$ and~$-1$ in each row and each column.
\item A \defn{corner sum matrix} (CSM) of order~$n$ is an $(n+1) \times (n+1)$-matrix whose first row and column consist of~$0$ and last row and column consist of the numbers from~$0$ to~$n$ in increasing order, and each entry is either equal to, or one more than the preceding entry in its row and in its column.
\item A \defn{height function matrix} (HFM) of order~$n$  is an $(n+1) \times (n+1)$-matrix whose first (resp.~last) row and column consist of the numbers from~$0$ to~$n$ in increasing (resp.~decreasing) order, and any two consecutive entries in a row or a column differ~by~$1$.
\end{itemize}
As illustrated in \cref{fig:ASM1}, there are simple bijections between these families of matrices:
\begin{itemize}
\item the CSM~$C$ of an ASM~$A$ is defined by~$C_{i,j} = \sum_{i' \le i, j' \le j} A_{i',j'}$ for~$0 \le i,j \le n$,
\item the HFM~$H$ of a CSM~$C$ is defined by~$H_{i,j} = i+j-2C_{i,j}$ for~$0 \le i,j \le n$.
\end{itemize}
(Note that we label the rows and columns of ASMs from $1$ to~$n$ and the rows and columns of CSMs and HFMs from $0$ to~$n$.)

\begin{figure}
	\centerline{
	$\begin{pmatrix} \cdot & \cdot & + & \cdot & \cdot \\[.1cm] \cdot & + & - & + & \cdot \\[.1cm] \cdot & \cdot & + & - & + \\[.1cm] + & - & \cdot & + & \cdot \\[.1cm] \cdot & + & \cdot & \cdot & \cdot \end{pmatrix}$\quad
	$\begin{pmatrix} 0 & 0 & 0 & 0 & 0 & 0 \\ 0 & 0 & 0 & 1 & 1 & 1 \\ 0 & 0 & 1 & 1 & 2 & 2 \\ 0 & 0 & 1 & 2 & 2 & 3 \\ 0 & 1 & 1 & 2 & 3 & 4 \\ 0 & 1 & 2 & 3 & 4 & 5 \end{pmatrix}$\quad
	$\begin{pmatrix} 0 & 1 & 2 & 3 & 4 & 5 \\ 1 & 2 & 3 & 2 & 3 & 4 \\ 2 & 3 & 2 & 3 & 2 & 3 \\ 3 & 4 & 3 & 2 & 3 & 2 \\ 4 & 3 & 4 & 3 & 2 & 1 \\ 5 & 4 & 3 & 2 & 1 & 0 \end{pmatrix}$\quad
	}
	\caption{An ASM, a CSM, and a HFM, all corresponding to each other.}
	\label{fig:ASM1}
\end{figure}

The HFMs of order~$n$ ordered by entrywise comparison define a distributive lattice~$\HFM_n$ where the join (resp.~meet) of two HFMs~$H$ and~$H'$ is the HFM with~$(i,j)$-entry $\max(H_{i,j}, H'_{i,j})$ (resp.~$\min(H_{i,j}, H'_{i,j})$).
Note that two HFMs~$H, H' \in \HFM_n$ form a cover relation~$H \lessdot H'$ if and only if there is~$0 < i,j < n$ such that~$H'_{i,j} = H_{i,j} + 2$ and~$H'_{k,\ell} = H_{k,\ell}$ for all~$0 < k, \ell < n$ with~$(i,j) \ne (k,\ell)$.
We then say that~$(i,j)$ is an \defn{ascent} of~$H$ and a \defn{descent} of~$H'$.

We now introduce a geometric interpretation of the join irreducible poset following~\cite{Propp01, Str11}.
The join irreducible poset~$\JIrr(\HFM_n)$ can be seen geometrically as a poset~$\TPoset_n$ on the integer points inside the $(n-2)$th dilate of a $3$-dimensional tetrahedron resting on an edge~\cite{Propp01}.
See \cref{fig:posets}.
Namely, using barycentric coordinates, this poset is defined on the quadruples~$\b{x} := (x_1, x_2, x_3, x_4) \in \N^4$ with~$x_i \ge 0$ for~$i \in [4]$ and~$x_1 + x_2 + x_3 + x_4 = n-2$ ordered by $\b{x} \le \b{y}$ if and only if $x_1\ge y_1$, $x_2\le y_2$, $x_3\le y_3$ and $x_4\ge y_4$.
Its cover relations are the pairs~$\b{x} \lessdot \b{y}$ such that~$\b{y} - \b{x}$ is one of the four vectors~$b \eqdef (-1,1,0,0)$, $r \eqdef (-1,0,1,0)$, $o \eqdef (0,1,0,-1)$ and~$g \eqdef (0,0,1,-1)$.
We draw these vectors in $4$ distinct colors, blue, red, orange, green.
We can thus see an HFM as a lower set in this 4-colored tetrahedral poset, as illustrated in \cref{fig:ASM2}\,(left).

Note that via the above-mentioned bijection between ASMs and HFMs, we can transport the lattice~$\HFM_n$ on HFMs of order~$n$ to a lattice~$\ASM_n$ on ASMs of order~$n$.
See \cref{fig:Dyck3ASM3}\,(right).
Identifying a permutation~$\sigma$ of~$[n]$ with its matrix~$P_\sigma \eqdef [\delta_{\sigma(i),j}]_{i,j \in [n]}$, this lattice~$\ASM_n$ is the MacNeille completion of the Bruhat order on permutations of~$[n]$~\cite{LascouxSchutzenberger_TreillisBasesCoxeter}.
We call \defn{ascents} and \defn{descents} of an ASM the ascents and descents of its corresponding~HFM.
The join (resp.~meet) irreducible ASMs are known to be the matrices of the \defn{bigrassmannian} (resp.~\defn{anti-bigrassmannian}) permutations, \ie the permutations~$\sigma$ such that $\sigma$ and its inverse~$\sigma^{-1}$ have only one descent (resp.~ascent)~\cite{LascouxSchutzenberger_TreillisBasesCoxeter}.
Namely, an integer point~$\b{x} = (x_1, x_2, x_3, x_4)$ in the tetrahedron corresponds to the bigrassmannian and anti-bigrassmannian permutations
\[
j(\b{x}) = \begin{pmatrix} I_{x_1} & 0 & 0 & 0 \\ 0 & 0 & I_{x_2+1} & 0 \\ 0 & I_{x_3+1} & 0 & 0 \\ 0 & 0 & 0 & I_{x_4} \end{pmatrix}
\qquad\text{and}\qquad
m(\b{x}) = \begin{pmatrix} 0 & 0 & 0 & \bar I_{x_2} \\ 0 & \bar I_{x_1+1} & 0 & 0 \\ 0 & 0 & \bar I_{x_4+1} & 0 \\  \bar I_{x_3} & 0 & 0 & 0 \end{pmatrix}
\]
where~$I_k$ (resp.~$\bar I_k$) is the $k \times k$ identity (resp.~anti-identity) matrix with $1$ on the diagonal (resp.~anti-diagonal) and~$0$ everywhere else.

It is also interesting to represent an ASM as the surface~$S$ separating the lower set~$X$ of~$\TPoset_n$ and its complement.
For this, we include additional edges so that each integer point of the tetrahedron gets $12$ incident arrows (one in the direction~$\b{e}_i-\b{e}_j$ for each~$i \ne j \in [4]$).
The horizontal edges, in direction~$\pm(1,0,0,-1)$ and~$\pm(0,1,-1,0)$, are colored in gray.
The additional endpoints are considered in~$X$ (resp.~in its complement) if they are incident to an arrow which is also incident to one of the two bottom (resp.~top) faces of the tetrahedron.
We then add a rhombus orthogonal to each arrow from a point in~$X$ to a point in its complement (the color of the rhombus matches the color of the orthogonal arrow).
The corresponding surface is illustrated in \cref{fig:ASM2}\,(middle and right)\footnote{See also \url{https://www.ub.edu/comb/vincentpilaud/documents/publications/ASM.html}.}.

\begin{figure}[h]
	\centerline{
	\includegraphics[scale=.14,valign=c]{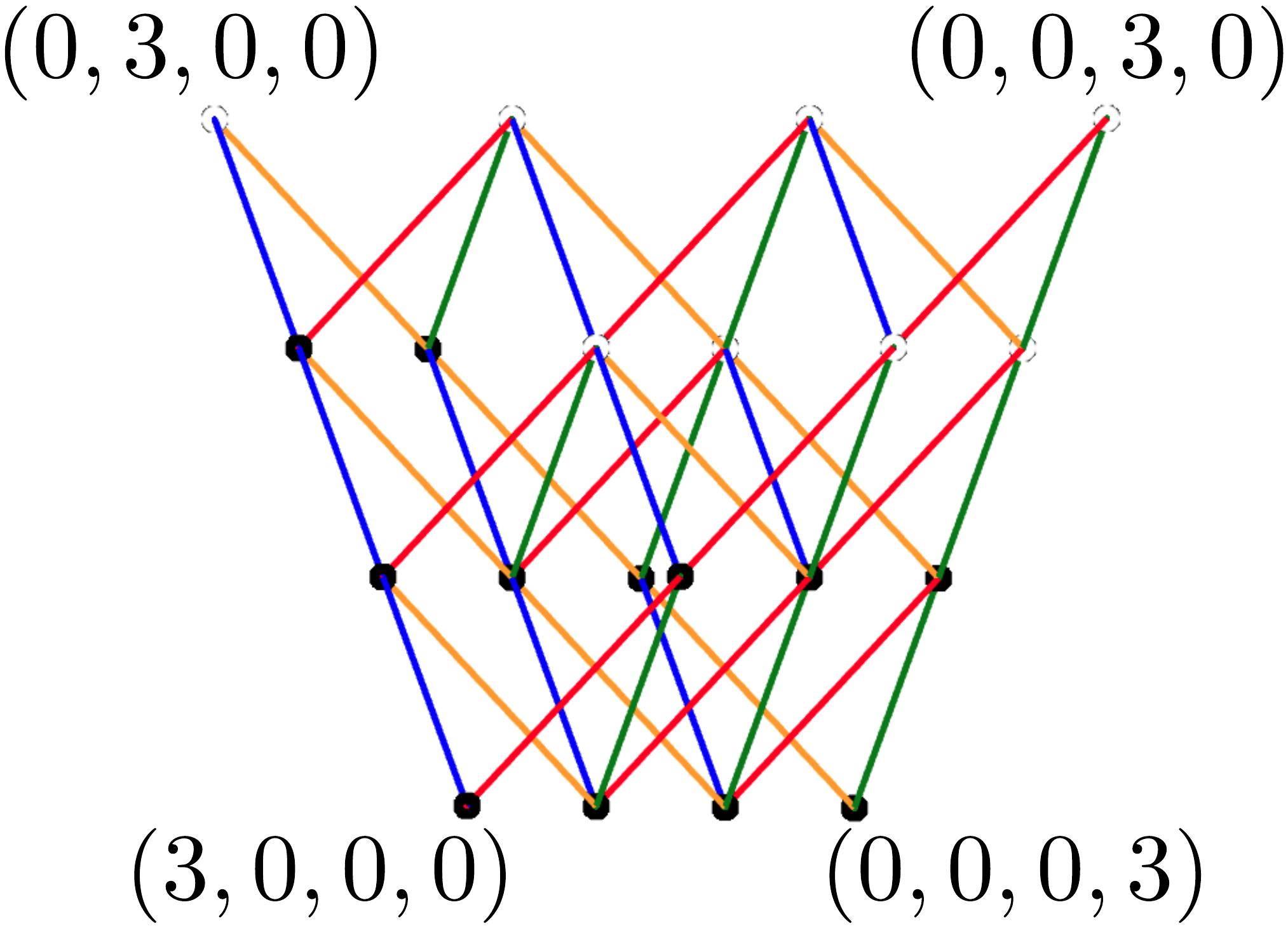}
	\includegraphics[scale=.2,valign=c]{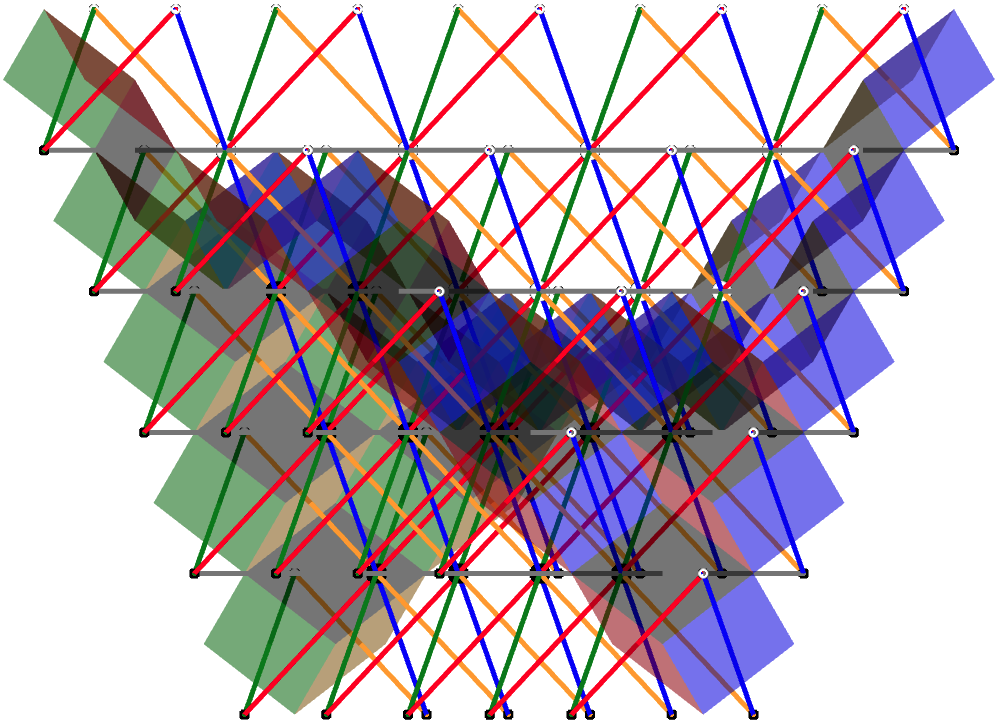}
	\includegraphics[scale=.2,valign=c]{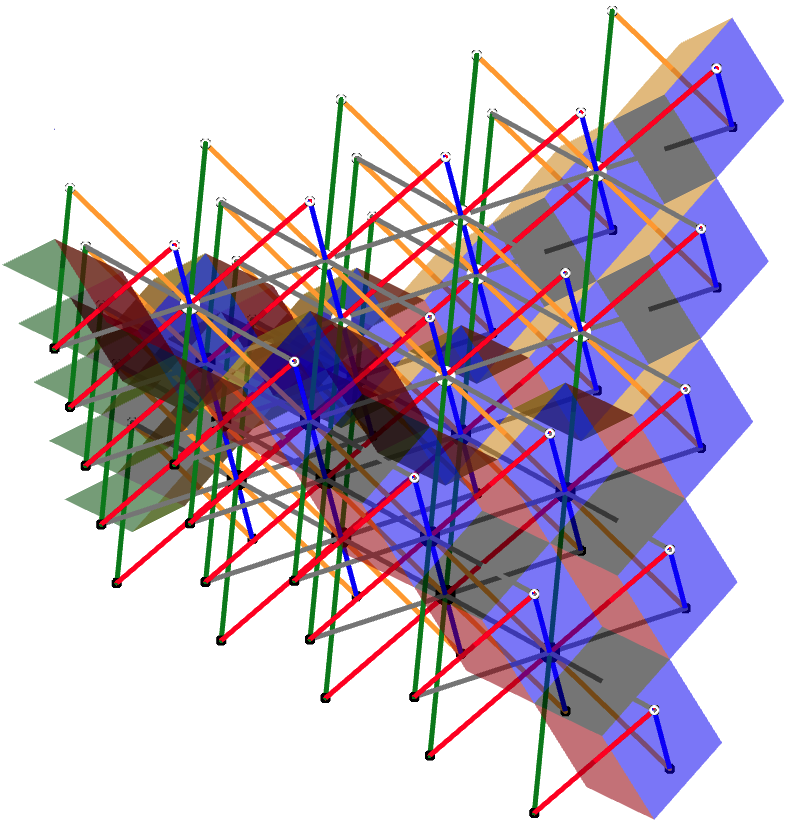}
	}
	\caption{The ASM of \cref{fig:ASM1} seen as a lower set~$X$ of the tetrahedron poset~$\TPoset_5$ (left) or equivalently as the surface separating~$X$ from its complement (middle and right).}
	\label{fig:ASM2}
\end{figure}

As illustrated in \cref{fig:ASM3}, looking at this surface~$S$ from above (meaning in the direction~$(-1,1,1,-1)$, or more precisely under the projection~${(x_1,x_2,x_3,x_4) \mapsto (x_1+x_2, x_1+x_3)}$) makes transparent some well-known bijections between classical equivalent combinatorial models for ASMs (and naturally motivates combinatorial models that were not considered before to the best of our knowledge):

\enlargethispage{.5cm}
\begin{itemize}
\item Seen from above, the colored rhombi of~$S$ forms a \defn{rainbow tiling}. See \cref{fig:ASM3}\,(a). The tiled zone is formed by $n$ rows and $n$ columns of $n+1$ tiles each. Hence, we have~$2n(n+1)$ tiles in total. We call \defn{integer points} the vertices in the middle of a row and a column (see \cref{fig:ASM3}\,(c)), and \defn{half integer points} the other vertices (see \cref{fig:ASM3}\,(d)). A rainbow tiling is a $4$-coloring of the tiles such that
\begin{itemize}
\item the south/east/north/west tiles are colored red/blue/orange/green, 
\item the rows (resp.~columns) are colored with cold colors blue/green (resp.~with warm colors red/orange), 
\item each integer point sees either~$2$ or~$4$ colors.
\end{itemize}
Note that, up to recoloring, this rainbow tiling is equivalent to a 4-coloring of the \defn{aztec diamond}~\cite{ElkiesKuperbergLarsenPropp}.

\item The \defn{ASM} records the saddle points of~$S$, seen from above. Namely, the $+1$ and $-1$ of the ASM are the two types of saddle points of~$S$ (depending on the orientation of the principal directions of the saddle point). They can be visualized as the rainbow integer points (\ie surrounded by the four colors red/blue/orange/green) of the rainbow tiling of \cref{fig:ASM3}\,(a). See \cref{fig:ASM3}\,(b).

\item The \defn{six-vertex configuration}~\cite{Kuperberg} is given by the horizontal diagonals of the rhombi of~$S$, oriented counterclockwise around the surface~$S$ (in other words, in the direction of the march of the dahu\footnote{https://en.wikipedia.org/wiki/Dahu}), and seen from above. See \cref{fig:ASM3}\,(c). Equivalently, it can be obtained by placing south/west/north/east pointing arrows in the red/blue/orange/green tiles of the rainbow tiling of \cref{fig:ASM3}\,(a). This gives an orientation of the integer grid such that
\begin{itemize}
\item the extremal horizontal (resp.~vertical) arrows are incoming (resp.~outgoing), 
\item each integer point has two incoming and two outgoing arrows.
\end{itemize}
It has $6$ types of vertices \smash{\includegraphics[scale=.11,valign=c]{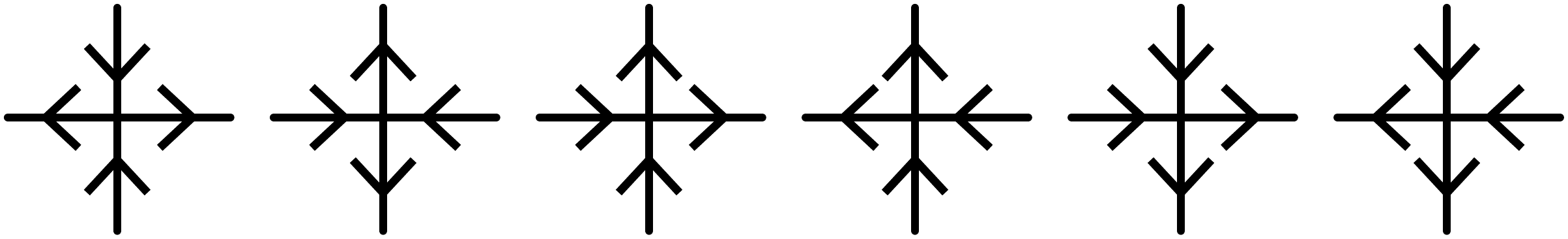}}, hence the name.
Note that, as the nonzero entries of the ASM correspond to saddle points of~$S$, they correspond to the vertices of the six-vertex configuration with opposite incoming arrows and opposite outgoing arrows. See the connection between \cref{fig:ASM3}\,(b)\,\&\,(c).

\item The \defn{six-square configuration} is given by the non-horizontal diagonals of the rhombi of~$S$, oriented upwards, and seen from above. See \cref{fig:ASM3}\,(d). Equivalently, it can be obtained by placing east/south/west/north pointing arrows in the red/blue/oran\-ge/green tiles of the rainbow tiling of \cref{fig:ASM3}\,(a). Equivalently, it is the dual directed graph of the six-vertex configuration of \cref{fig:ASM3}\,(c) (meaning that each arrow is rotated by~$90^\circ$). This gives an orientation of the half-integer grid such that
\begin{itemize}
\item the arrows on the bottom/right/top/left side are oriented east/south/west/north, 
\item each mesh, traveled clockwise, has two forward and two backward arrows.
\end{itemize}
It has $6$ types of squares \includegraphics[scale=.11,valign=c]{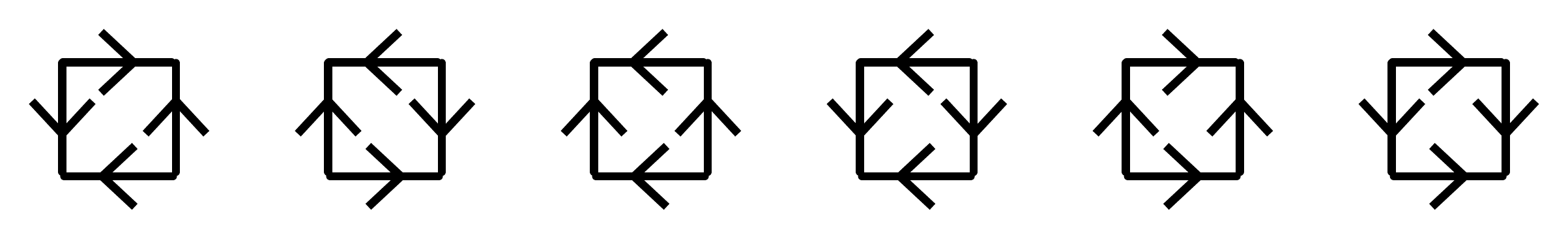}, hence the name.

\item The \defn{peak--pit matrix} is given by the local optima of the surface~$S$, seen from above. \cref{fig:ASM3}\,(e). They can be visualized as the rainbow half integer points of the rainbow tiling of \cref{fig:ASM3}\,(a). Equivalently, they are the points with indegree or outdegree~$4$ in the six-square configuration of \cref{fig:ASM3}\,(d). We are not aware that this model has been studied before.

\item The \defn{osculating path configuration}~\cite{Behrend} or \defn{bumpless pipe dream}~\cite{Weigandt} is given by the horizontal diagonals of the red and blue rhombi of~$S$, seen from above. See \cref{fig:ASM3}\,(f). Equivalently, they are the east and south pointing arrows in the six-vertex configuration of \cref{fig:ASM3}\,(c). They correspond to distinct adjacent entries in the CSM (see \cref{fig:ASM3}\,(g)) and to decreasing adjacent entries in the HFM (see \cref{fig:ASM3}\,(h)).

\item The \defn{HFM} records the height of the surface above each integer point seen from above. See \cref{fig:ASM3}\,(h).

\item The \defn{fully packed loop}~\cite{BatchelorBloteNienhuisYung, deGier} is given by the horizontal diagonals of the red and orange (resp.~green and blue) rhombi of~$S$ located at odd (resp.~even) heights, seen from above. See \cref{fig:ASM3}\,(e).

\item The \defn{cliffs configuration} is given by the cliffs (gray rhombi) of~$S$, seen from above. See \cref{fig:ASM3}\,(f). Equivalently, it is obtained from the HFM by joining the two equal heights in each $(2 \times 2)$-submatrix containing $3$ distinct heights. Cliffs configurations are exactly the configurations of paths where each point inside a square grid belongs to an even number of diagonal steps (0, 2 or 4), and points on the border belong to one diagonal step and are connected to points of the same height. We are not aware that this model has been studied before.
\end{itemize}

\begin{figure}
	\centerline{
	\begin{tabular}{ccccc}
	(a) & (b) & (c) & (d) & (e) \\
	\includegraphics[scale=.11,valign=c]{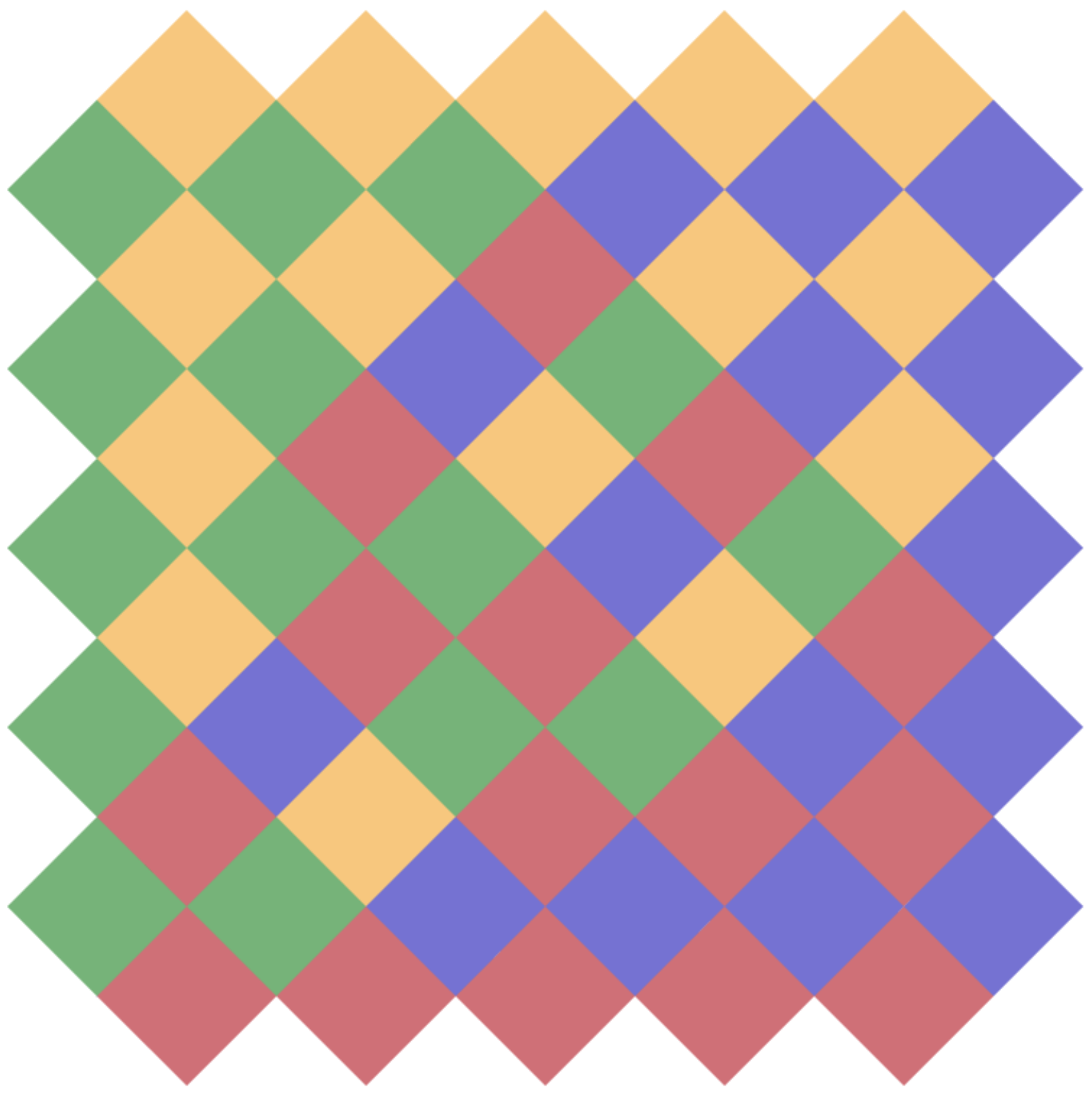} &
	\includegraphics[scale=.11,valign=c]{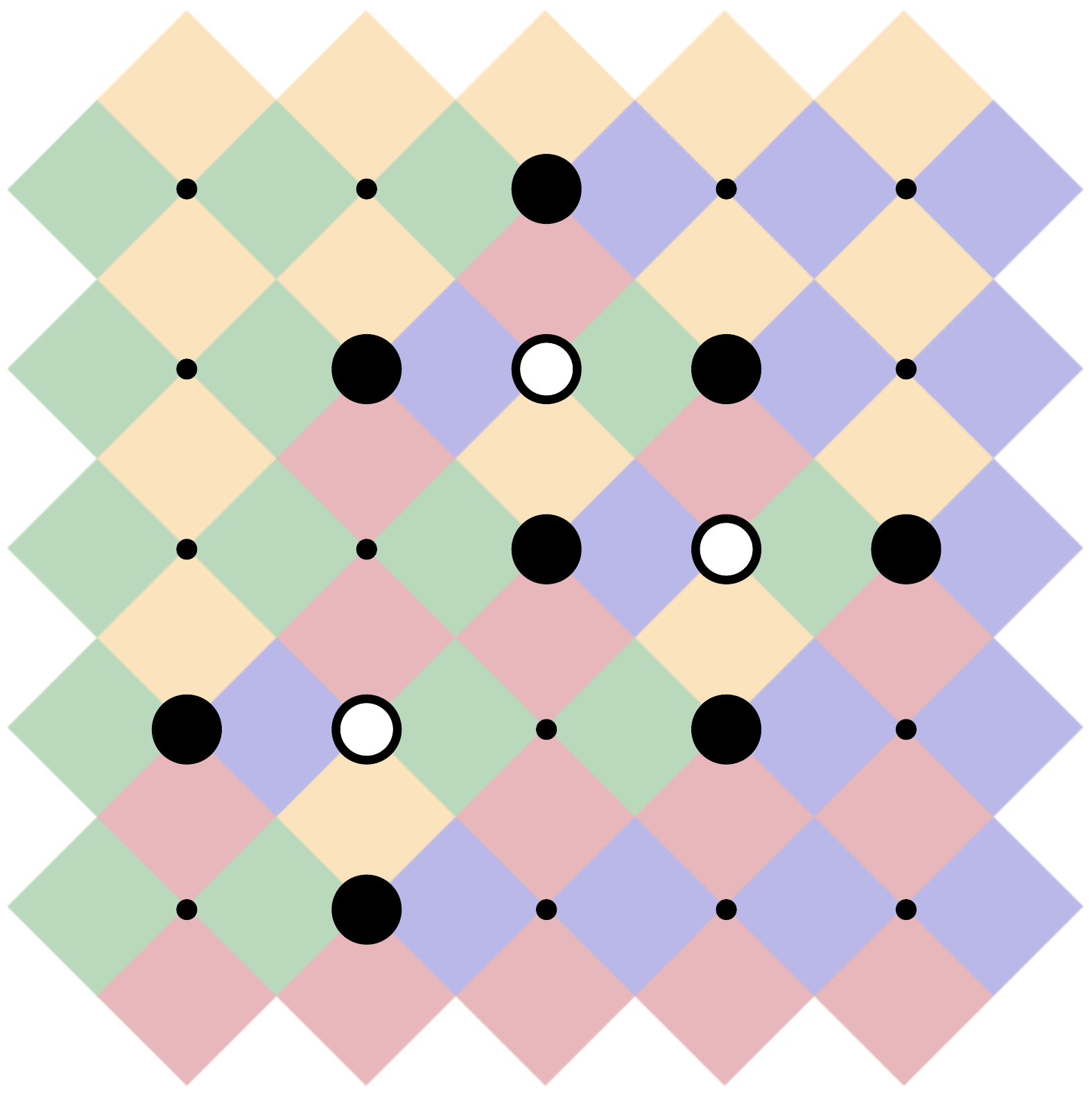} &
	\includegraphics[scale=.11,valign=c]{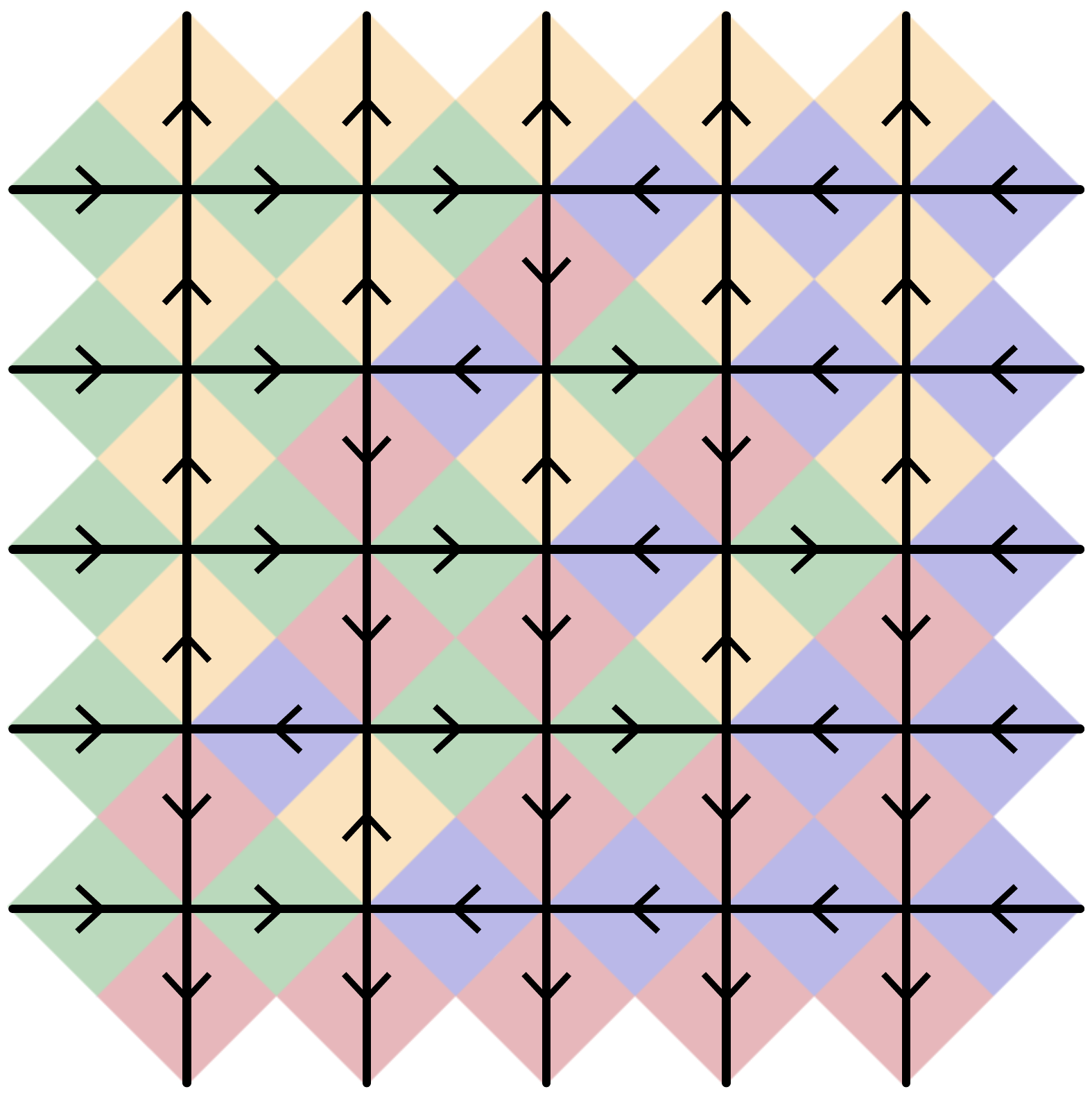} &
	\includegraphics[scale=.11,valign=c]{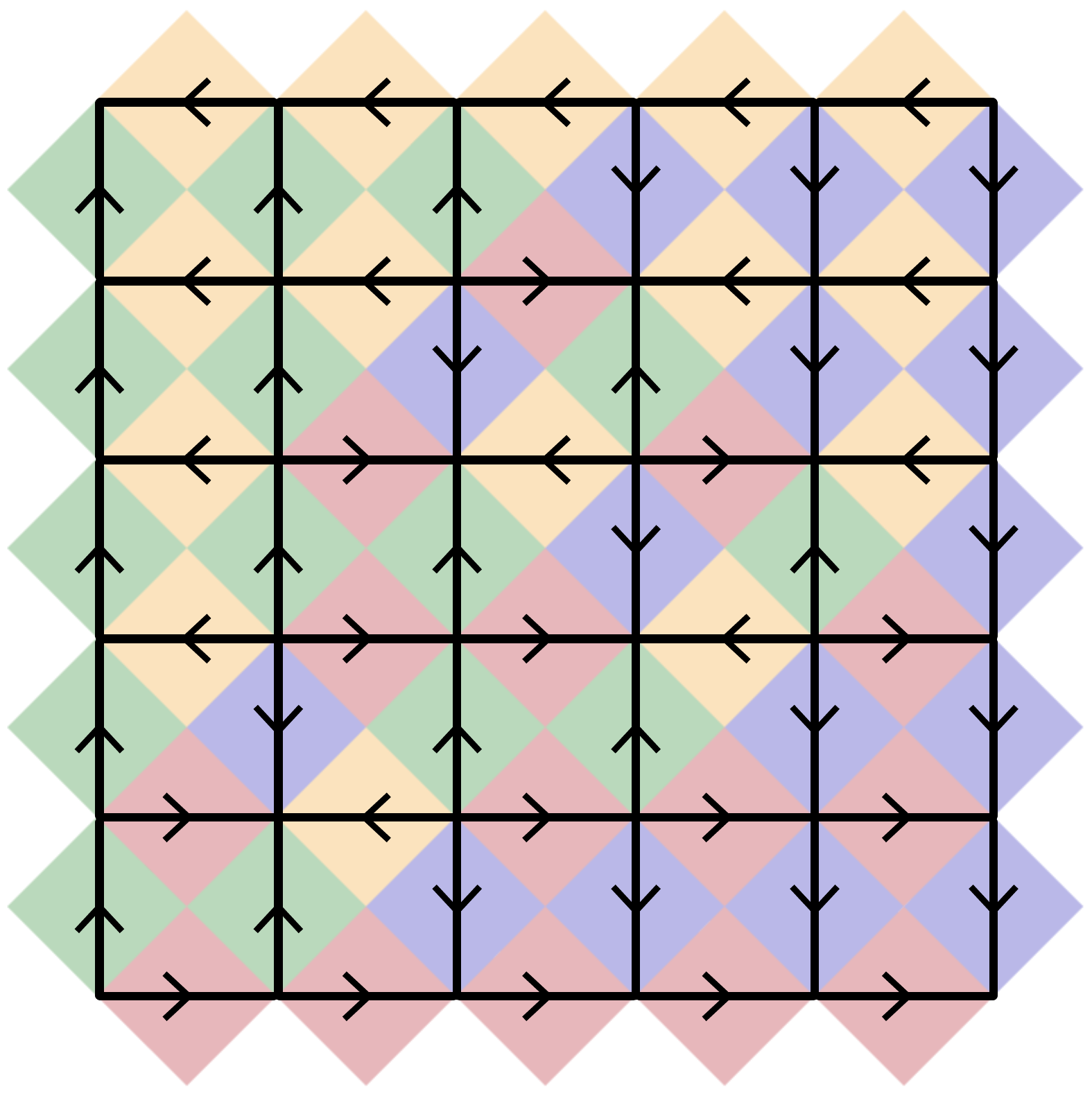} &
	\includegraphics[scale=.11,valign=c]{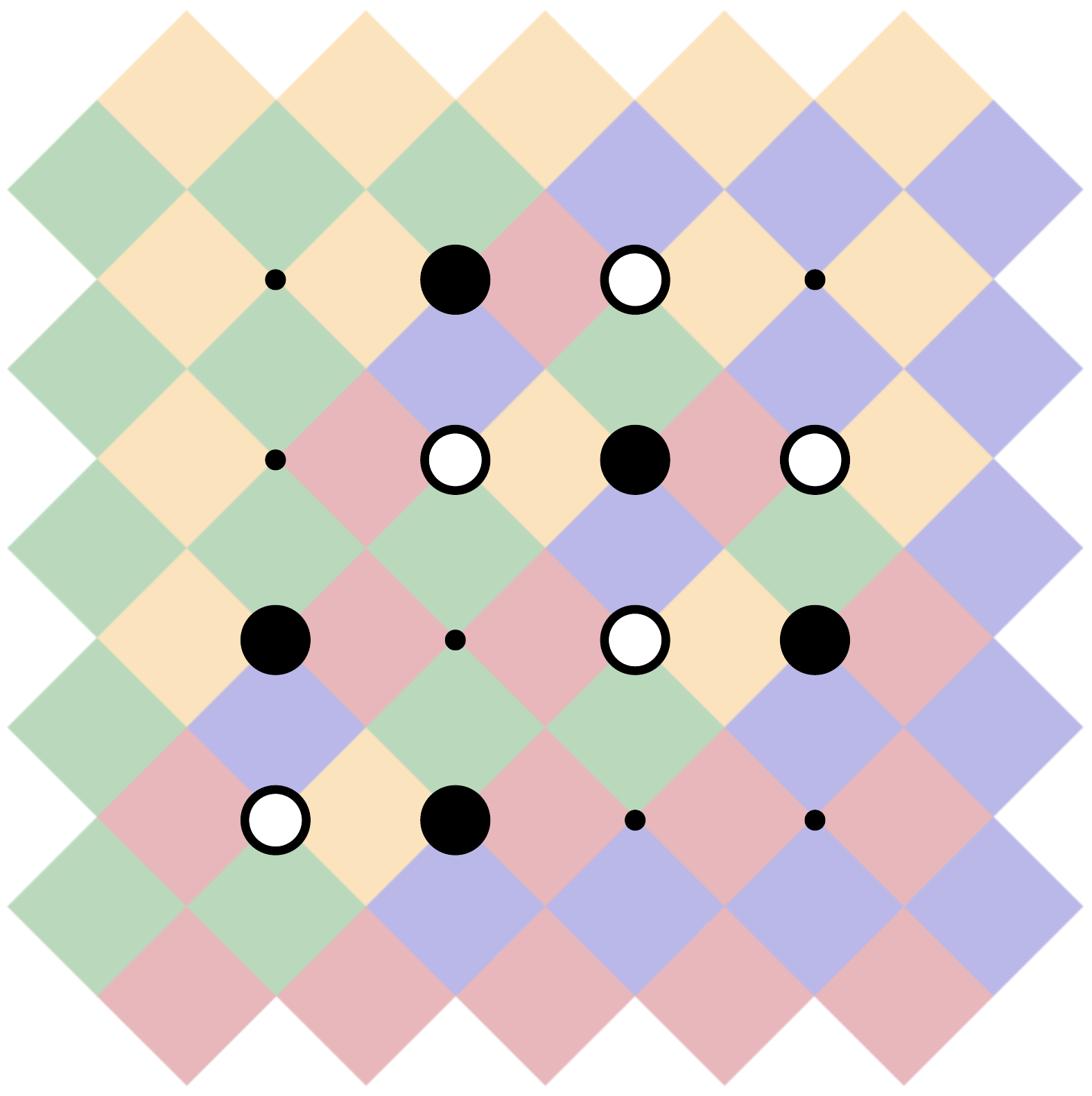} \\[1.5cm]
	\includegraphics[scale=.11,valign=c]{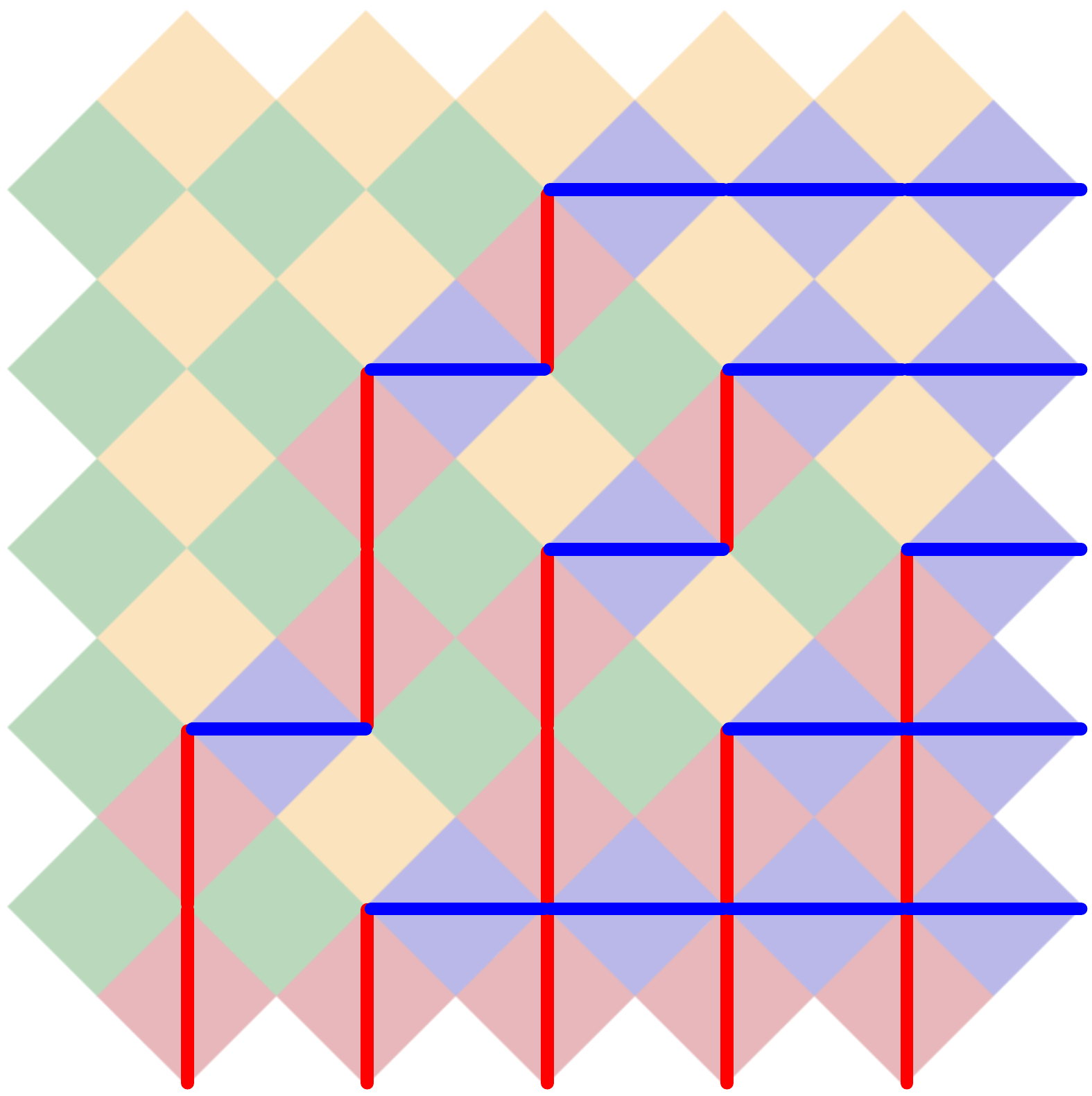} &
	\includegraphics[scale=.11,valign=c]{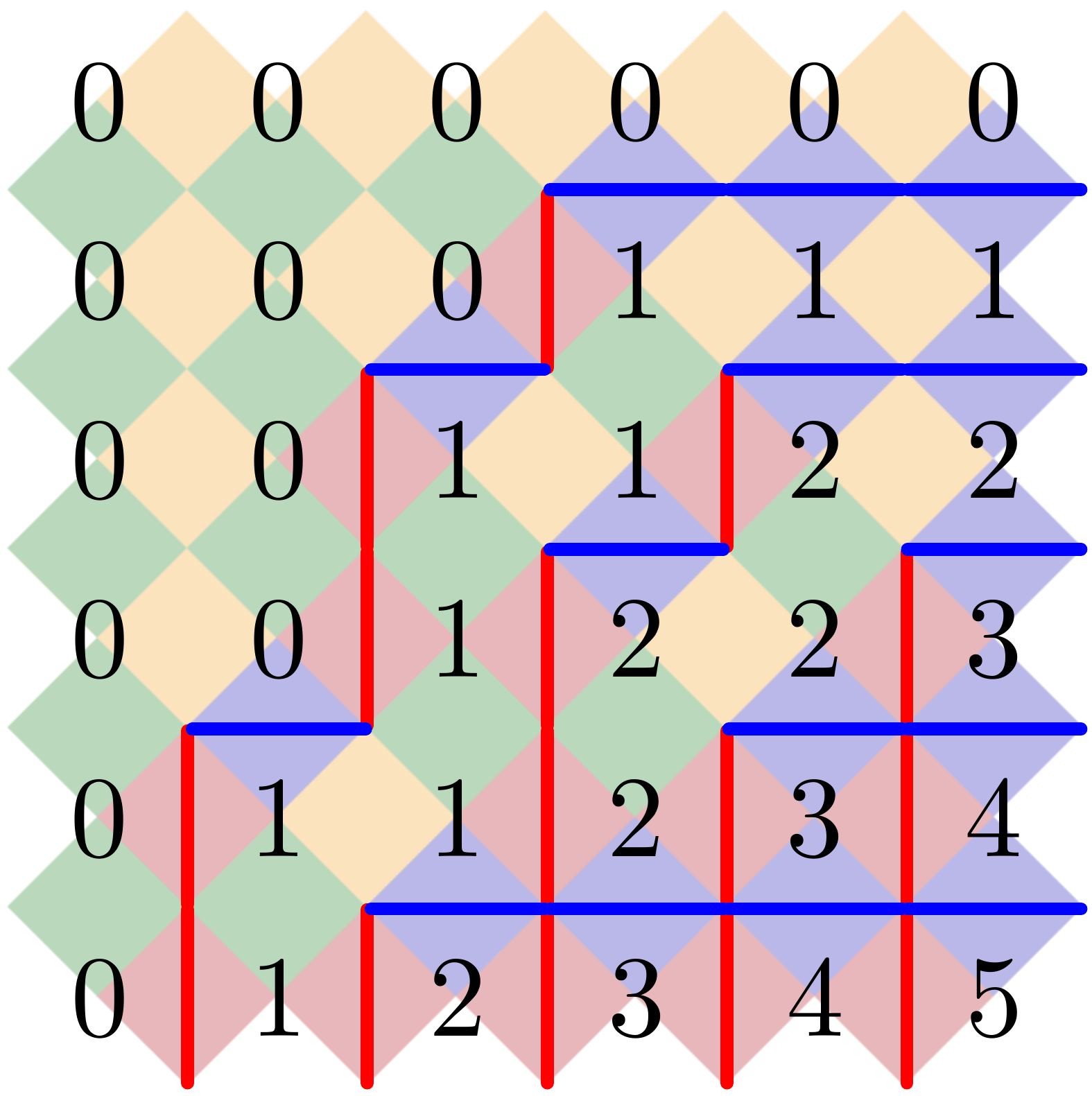} &
	\includegraphics[scale=.11,valign=c]{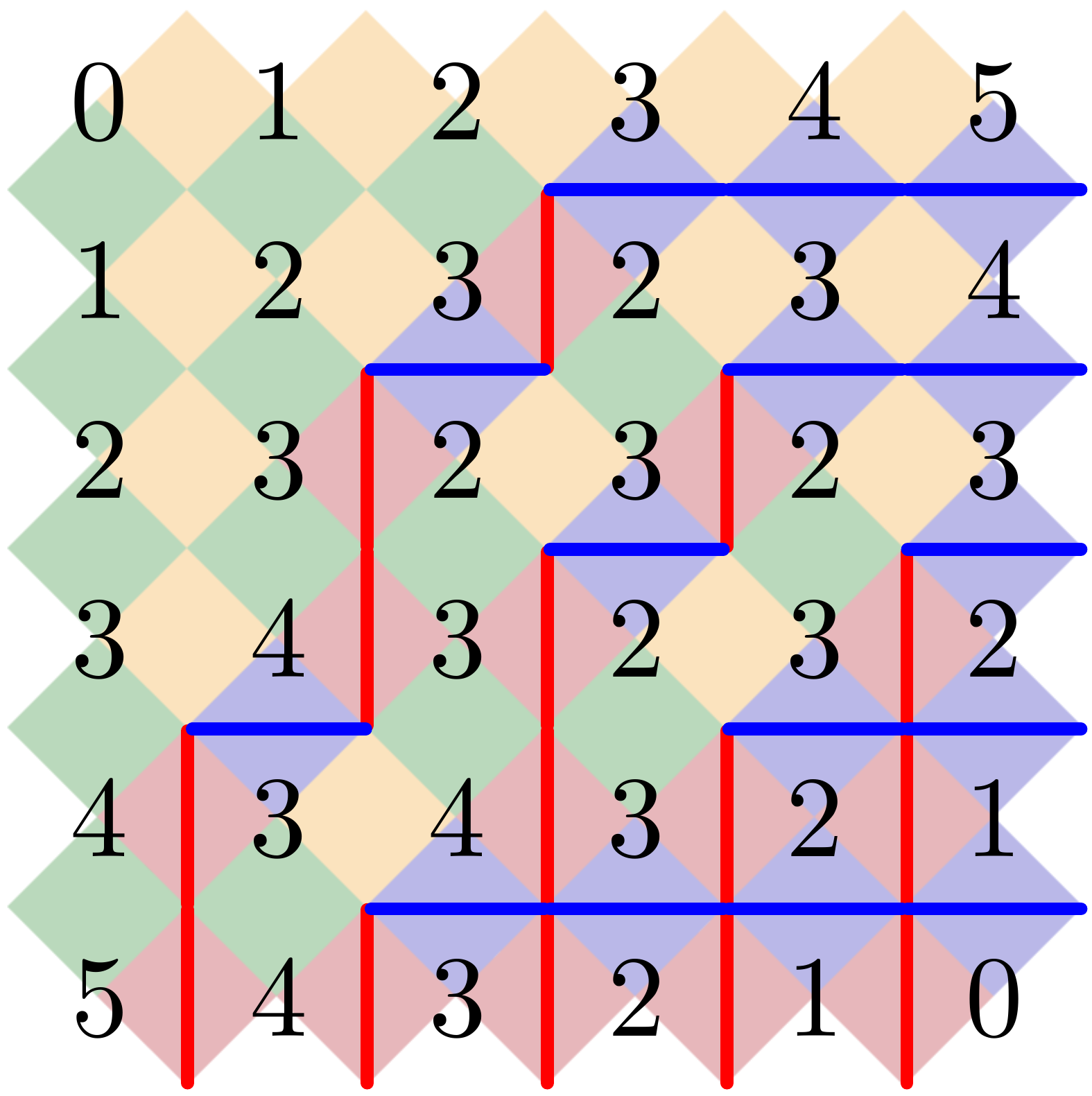} &
	\includegraphics[scale=.11,valign=c]{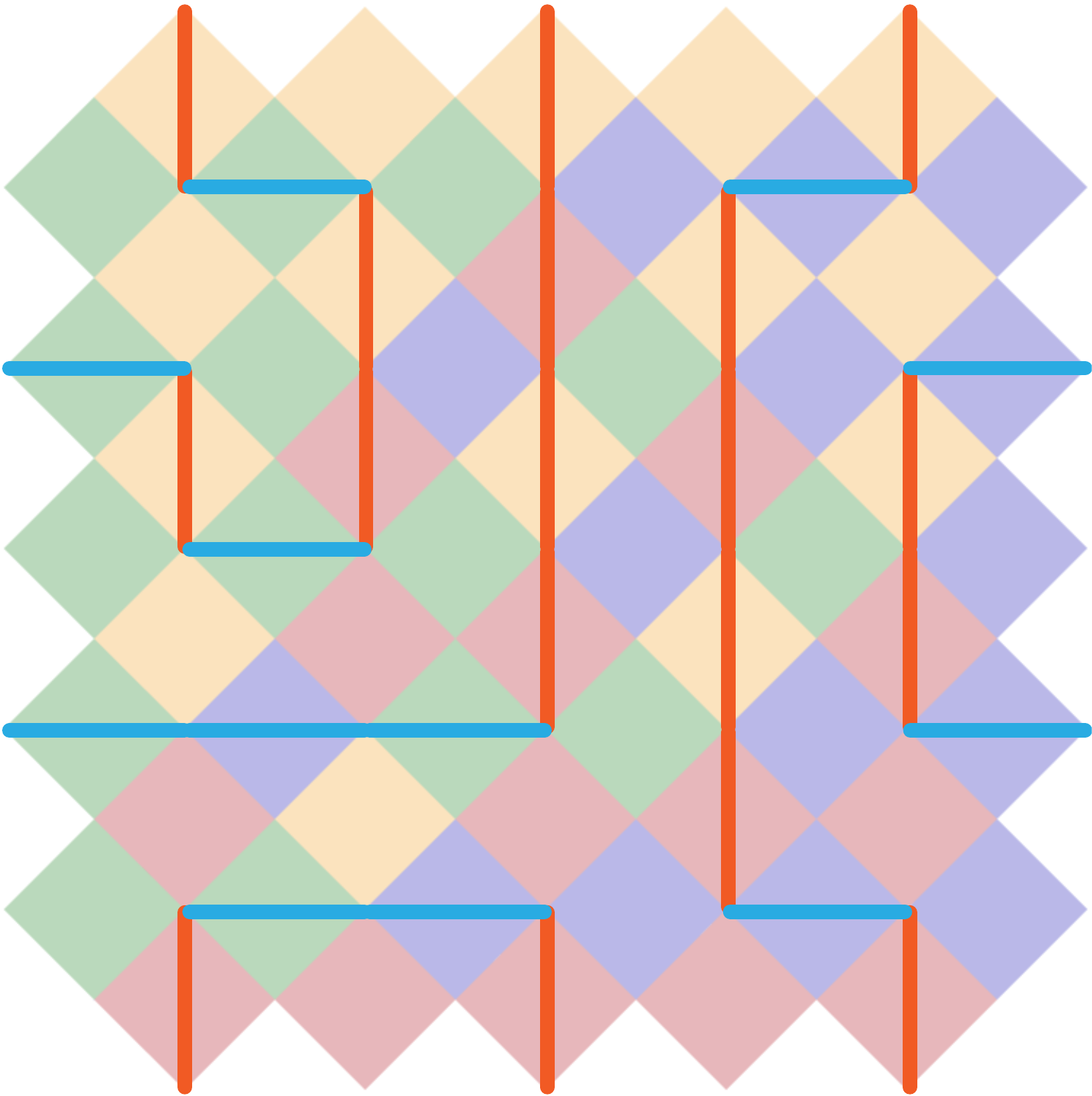} &
	\includegraphics[scale=.11,valign=c]{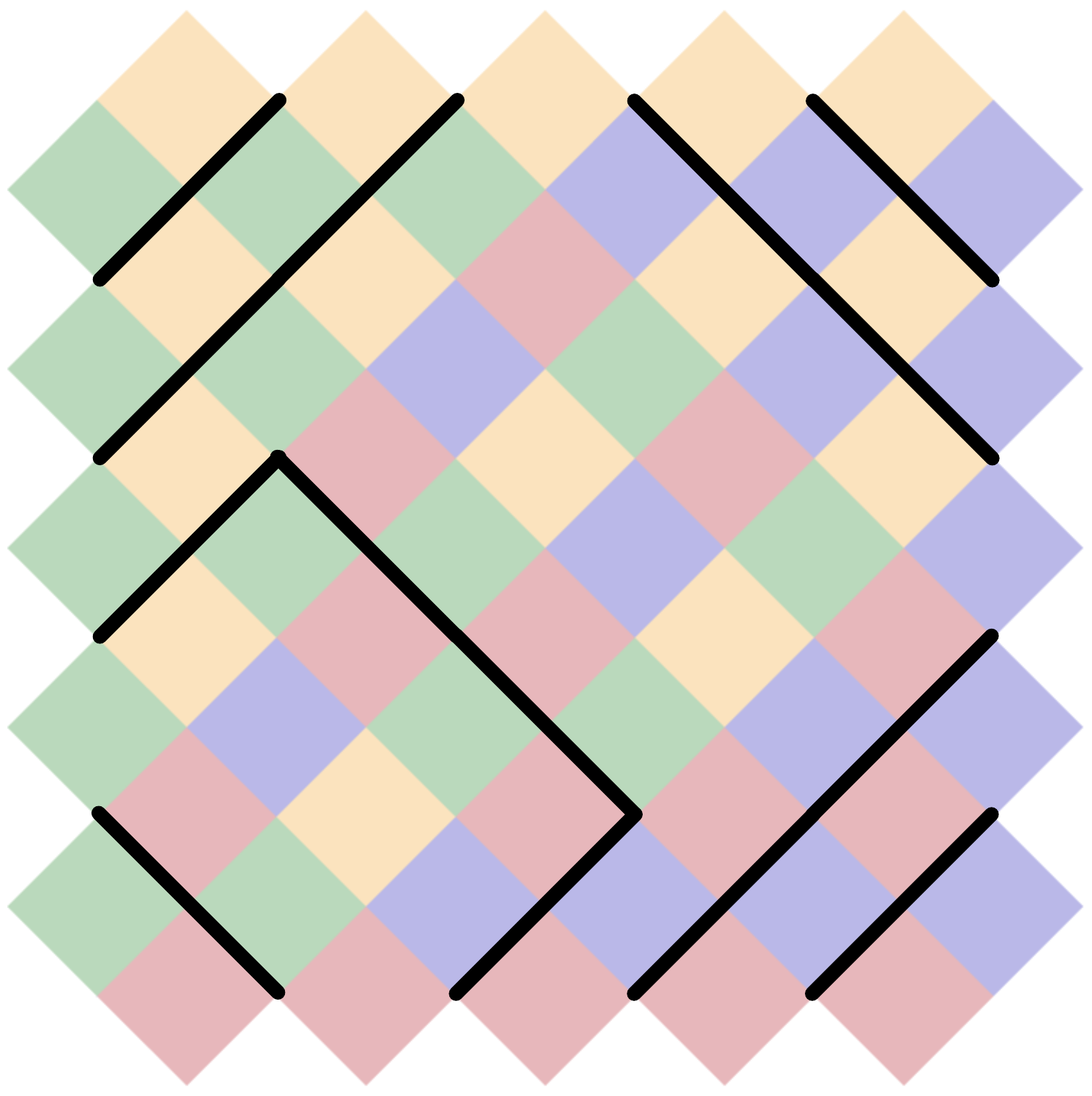} \\[1.4cm]
	(f) & (g) & (h) & (i) & (j)
	\end{tabular}
	}
	\caption{Reading bijections on the surface of the ASM of \cref{fig:ASM1}.}
	\label{fig:ASM3}
\end{figure}

\cref{fig:ASM4} illustrates these correspondence on bigger examples of ASM (b), peak-pit matrix~(e), fully packed loop (i), and cliff configuration (j).

\begin{figure}[h]
	\centerline{
	\begin{tabular}{cc}
	alternating sign matrix & peak-pit matrix \\[.2cm]
	\includegraphics[height=8cm,valign=c]{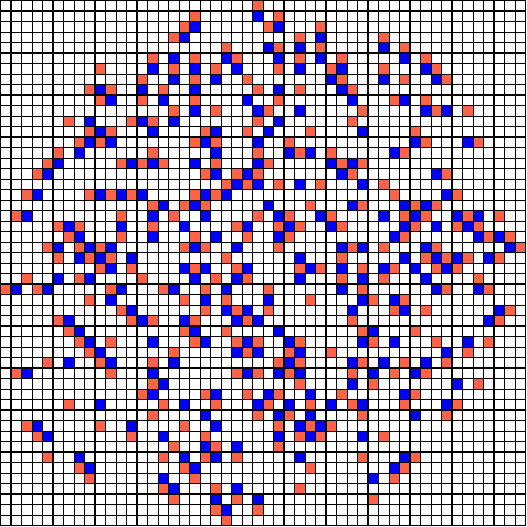} &
	\includegraphics[height=8cm,valign=c]{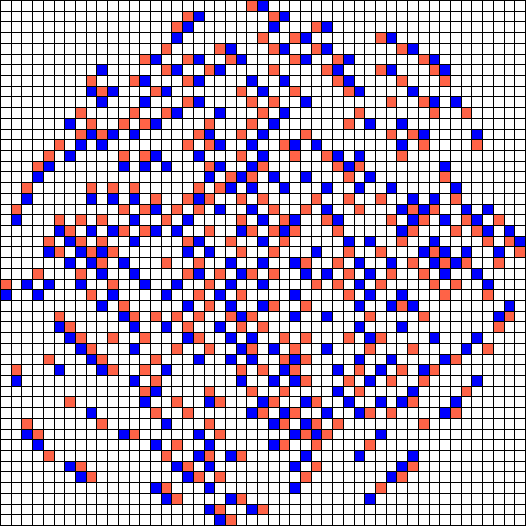} \\[4cm]
	\includegraphics[height=8cm,valign=c]{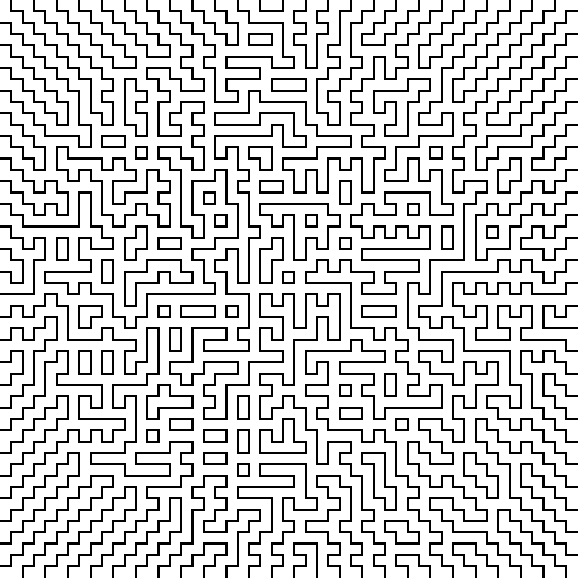} &
	\includegraphics[height=8cm,valign=c]{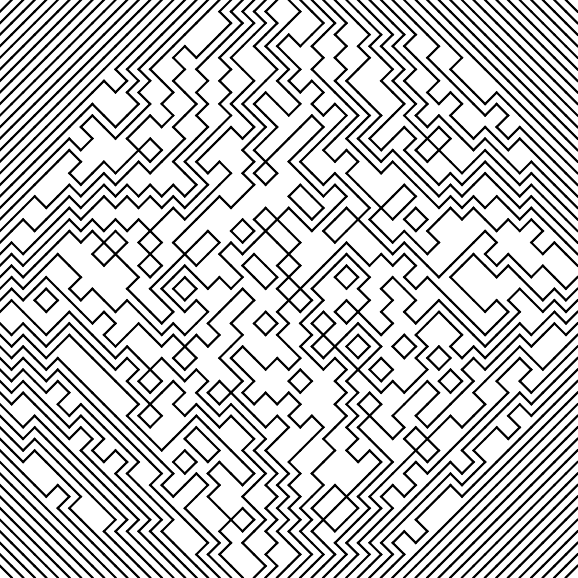} \\[4cm]
	fully packed loop & cliff configuration
	\end{tabular}
	}
	\caption{Bigger examples of the correspondence between ASM, peak-pit matrix, fully packed loop, and cliff configuration.}
	\label{fig:ASM4}
\end{figure}

Finally, it is also interesting to represent an ASM as a stack of rhombic dodecahedra.
Remember that the \defn{rhombic dodecahedron} is the polytope obtained 
\begin{itemize}
\item either as the convex hull of the $8$ vertices~$\pm\b{e}_1 \pm\b{e}_2 \pm\b{e}_3$ of the standard cube together with the~$6$ points~$\pm2\b{e}_i$ for~$i \in [3]$,
\item or as the intersection of the $12$ halfspaces defined by~$\dotprod{\b{r}}{\b{x}} \le 1$ for the roots~$\b{r}$ of type~$A_3$.
\end{itemize}
See \cref{fig:ASM5}\,(left).
It is precisely the Voronoi cell of the integer point lattice generated by the root system of type~$A_3$.
Hence, we can see a lower set~$X$ of the tetrahedron poset as a stack of rhombic dodecahedra.
The advantage of this model is that it can be 3d-printed and manipulated (see \cref{fig:ASM5}), after adding a suitable support (containing the rhombic dodecahedra corresponding to the points added to~$X$ above).
The surface~$S$ discussed above is the upper boundary of this stacking (including the additional rhombic dodecahedra of the support).

\begin{figure}
	\centerline{
	\includegraphics[scale=.3,valign=c]{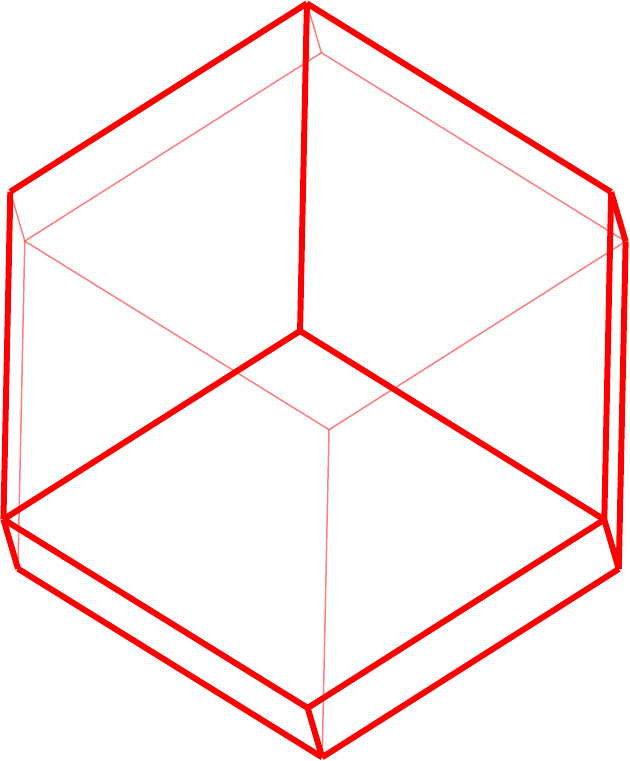} \quad
	\includegraphics[height=5cm,valign=c]{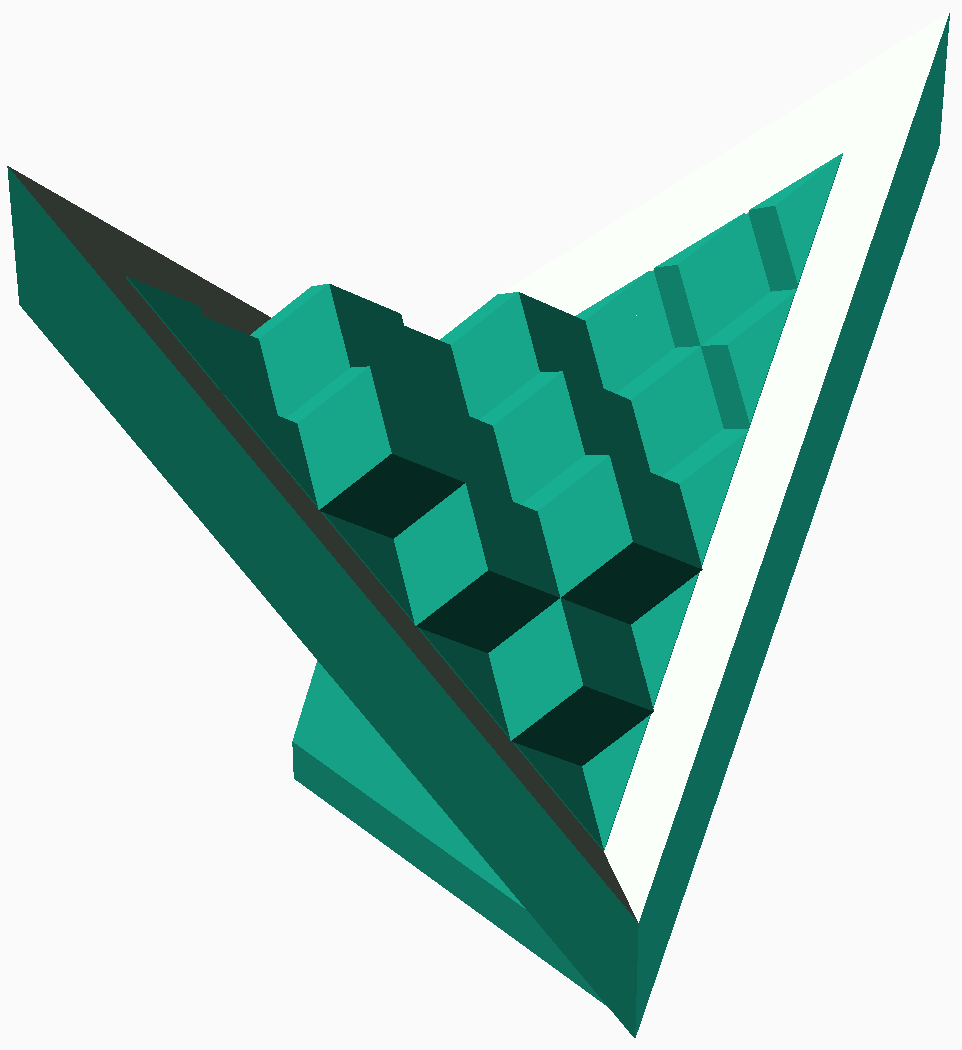} \quad
	\includegraphics[height=5cm,valign=c]{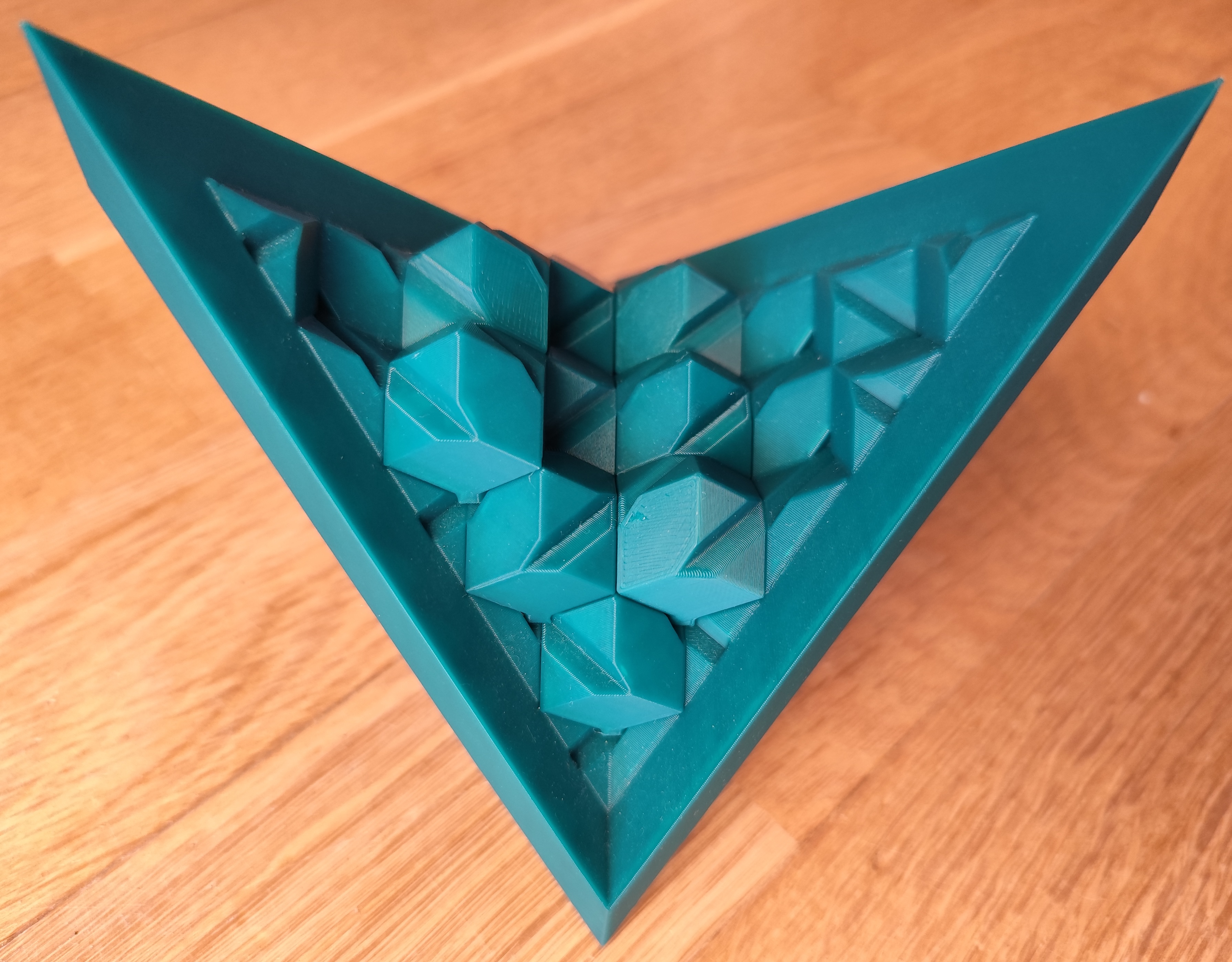}
	}
	\caption{The rhombic dodecahedron (left), and the ASM of \cref{fig:ASM1} seen as a pile of rhombic dodecahedra, as 3d objects in OpenSCAD (middle) and in real life (right).}
	\label{fig:ASM5}
\end{figure}



\subsection{Higher dimensional analogues}
\label{subsec:higherDimensionalAnalogues}

It is of course tempting to consider higher dimensional analogues of \cref{subsec:DyckPaths,subsec:ASMs}.
Fix three integers~$d, k, n$ and consider the poset~$P(d,k,n)$ whose
\begin{itemize}
    \item elements are the integer points in the \((n-2)\)th dilate of the \(d\)-dimensional simplex embedded in the standard way in \(\mathbb{R}^{d+1}\),
    \item cover relations are pairs of points \((\b{x},\b{y})\) such that \( \b{x} - \b{y} = \b{e}_i - \b{e}_j \) with \( i \leq k < j \).
\end{itemize}
Let \( Q(d,k,n) \) denote the distributive lattice of lower sets of \( P(d,k,n) \).
For example, \( Q(2,2,n) \) is the Stanley lattice, and \( Q(3,2,n) \) is the ASM lattice.

The Voronoi cell of the integer point lattice generated by the root system of type~$A_d$ is the graphical zonotope~\( Z(d) \) of a \((d+1)\)-cycle.
This is a polytope with \( 2 \times (2^d - 1) \) vertices (corresponding to all the acyclic orientations of the \((d+1)\)-cycle) and \( 2 \times \binom{d+1}{2} \) facets (corresponding to the biconnected subsets of the \((d+1)\)-cycle).
For \( d = 3 \), it is the rhombic dodecahedron with 14 vertices and 12 facets illustrated in \cref{fig:ASM5}.
We can thus see an element of~$D(d,k,n)$ as a stacking of copies of~$Z(d)$.

Now, one can try to understand the other models for ASMs in this framework.
More precisely, given an element of \( Q(d,k,n) \), \ie~a stacking of copies of~\( Z(d) \), we can look at it from the direction~$\b{w} = (\b{e}_1 + \ldots + \b{e}_k) - (\b{e}_{k+1} + \ldots + \b{e}_{d+1})$ and try to encode it by the following models:
\begin{itemize}
\item remembering the height above each point, seen from the direction~$\b{\omega}$,
\item remembering the colors (\ie directions) of the upper faces of the stacking, seen from the direction~$\b{\omega}$,
\item remembering the saddle points of the stacking, more precisely their types ($1$ or~$-1$) and positions under the projection in direction~$\b{w}$,
\item remembering the peaks and pits of the stacking, seen from the direction~$\b{w}$,
\item remembering the cliffs of the stacking, seen from the direction~$\b{\omega}$.
\end{itemize}
When the dimension is too small (\(d = 1\) and sometimes for \(d = 2\)), some of these models are not rich enough to recover the stacking (in particular the saddle point model). But as soon as the dimension is at least \(d = 3\), all these models allow one to encode the stackings. This yields interesting generalizations of Dyck paths and ASMs that deserve further study.


\section{The excedance congruence of~$\ASM_n$}
\label{sec:excedanceQuotient}

This section defines a relevant lattice congruence of~$\ASM_n$, extending the excedance relation of the Bruhat order on permutations defined by N.~Bergeron and L.~Gagnon in~\cite{BergeronGagnon}.


\subsection{Lattice morphisms and congruences}
\label{subsec:latticeCongruences}

A \defn{lattice morphism} is a map~$\phi: \L \to \K$ between two lattices which respects meets and joins, \ie such that $\phi(x \vee y) = \phi(x) \vee \phi(y)$ and~$\phi(x \wedge y) = \phi(x) \wedge \phi(y)$ for any~$x,y \in \L$.
A \defn{congruence} of a lattice~$\L$ is an equivalence relation~$\equiv$ on~$\L$ which respects meets and joins, \ie $x \equiv x'$ and~$y \equiv y'$ implies $x \vee y \equiv x' \vee y'$ and~$x \wedge y \equiv x' \wedge y'$ for any~$x,y \in \L$.
The \defn{quotient}~$\L/{\equiv}$ is the lattice on the equivalence classes of~$\L$ where
\begin{itemize}
\item $X \le Y$ if and only if there exist~$x \in X$ and~$y \in Y$ such that~$x \le y$,
\item $X \vee Y$ (resp.~$X \wedge Y$) is the equivalence class of~$x \vee y$ (resp.~$x \wedge y$) for any representatives~$x \in X$ and~$y \in Y$.
\end{itemize}
Note that the equivalence classes of a congruence are always intervals of~$\L$, and that the quotient~$\L/{\equiv}$ is isomorphic (as poset) to the subposet of~$\L$ induced by the minimal (resp.~maximal) elements of the congruence classes.
Observe also that
\begin{itemize}
\item the projection map~$\L \to \L/{\equiv}$ of a congruence~$\equiv$ on~$\L$ is a surjective lattice morphism,
\item conversely, the fibers of a surjective lattice morphism~$\L \to \K$ define a congruence~$\equiv$ on~$\L$ whose quotient~$\L/{\equiv}$ is isomorphic to~$\K$.
\end{itemize}


\subsection{The excedance relation on permutations}
\label{subsec:excedanceRelation}

We now recall the amazing properties of the excedance relation on permutations defined in~\cite{BergeronGagnon}.

A \defn{noncrossing partition} (NCP) of~$[n]$ is a collection~$\lambda \subseteq \set{(i,j)}{1 \le i \le j \le n}$ such that there is no~$(i,j), (k,\ell) \in \lambda$ with~$i \le k < j \le \ell$.
Equivalently, it is a collection of arcs over the points~$1, \dots, n$ on the horizontal line, such that any two arcs do not cross in their interior, nor share their left endpoints nor their right endpoints.
We denote by~$\lambda^- \eqdef \set{i}{(i,j) \in \lambda}$ the set of left endpoints and~$\lambda^+ \eqdef \set{j}{(i,j) \in \lambda}$ the set of of right endpoints of~$\lambda$.
We denote by~$\NCP_n$ the transitive closure of the relation~$\precdot$ on NCPs of~$[n]$ defined by~$\lambda \precdot \mu$ if and only if~$\lambda = \mu \ssm \{(i,i+1)\}$ for some~$i \in [n-1]$ or~$\lambda = \mu \ssm \{(i,k)\} \cup \{(i,j), (j,k)\}$ for some~$1 \le i < j < k \le n$.
Note that the standard bijection between Dyck paths of semilength~$n$ and NCPs of~$[n]$ (see \cref{fig:bijectionDyckPathNonCrossingPartition}) sends the Stanley lattice~$\Dyck_n$ to the lattice~$\NCP_n$ (see \cref{fig:noncrossingpartitions}).

\begin{figure}[ht]
	\centerline{\includegraphics[scale=1]{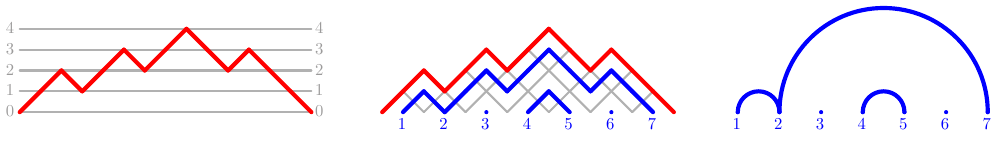}}
	\caption{The standard bijection from Dyck paths (left) to NCPs (right).}
	\label{fig:bijectionDyckPathNonCrossingPartition}
\end{figure}

\begin{figure}[ht]
	\centering
	\includegraphics[scale=.8]{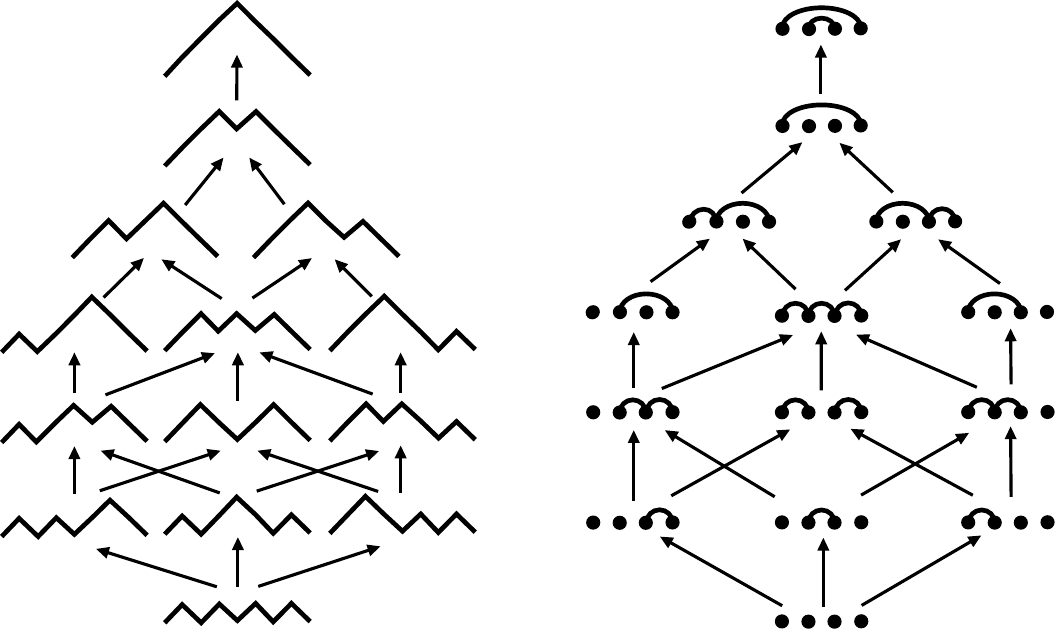}
	\caption{The Stanley lattice~$\Dyck_4$ on Dyck paths of semilength~$4$ (left) is isomorphic to the lattice~$\NCP_n$ on noncrossing partitions of~$[4]$ (right).}
	\label{fig:noncrossingpartitions}
\end{figure}

\pagebreak
\begin{definition}[\cite{BergeronGagnon}]
\label{def:excedanceRelation}
\begin{figure}
	\centering
	\includegraphics[scale=.7]{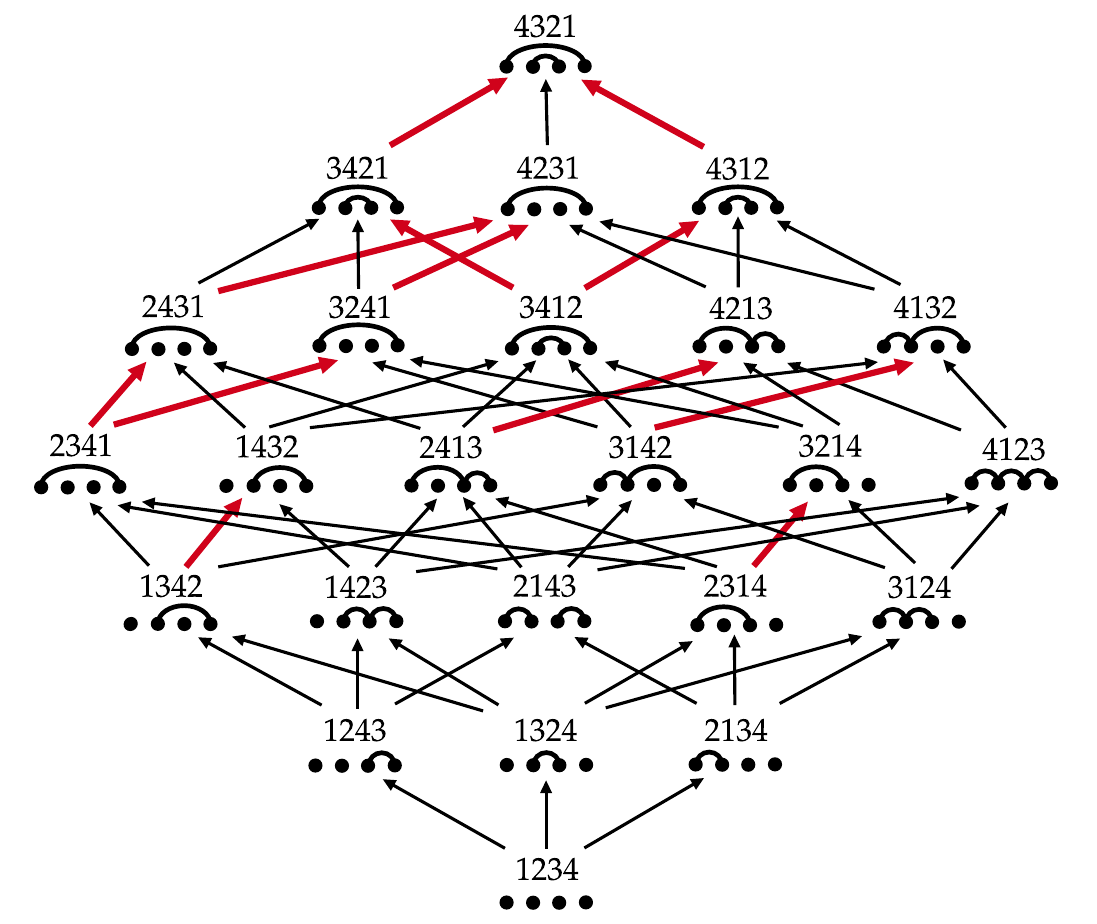}
	\caption{The excedance relation on the Bruhat order on~$\Ss_4$. Cover relations between two excedance equivalent permutations are red and bolded.}
	\label{fig:excedanceRelation}
\end{figure}
A \defn{weak excedance} of a permutation~$\sigma \in \Ss_n$ is a pair~$(i, \sigma_i)$ such that~$i \le \sigma_i$.
Define the \defn{excedance positions}~$\Epos(\sigma) \eqdef \set{i}{i \le \sigma_i}$ and the \defn{excedance values}~$\Eval(\sigma) \eqdef \set{\sigma_i}{i \le \sigma_i}$.
The \defn{excedance relation} on~$\Ss_n$ is defined by~$\sigma \equiv \tau$ if and only if~$\Epos(\sigma) = \Epos(\tau)$ and~$\Eval(\sigma) = \Eval(\tau)$.
We denote by~$\Ss_n/{\equiv}$ the poset quotient on equivalence classes of~$\equiv$, defined by~$X \le Y$ if and only if there are representatives~$x \in X$ and~$y \in Y$ such that~$x \le y$ in Bruhat order.
See \cref{fig:excedanceRelation} for an illustration when~$n = 4$.
\end{definition}

\begin{definition}[\cite{BergeronGagnon}]
\enlargethispage{.1cm}
For~$\lambda \in \NCP_n$, define two permutations~$\sigma_\lambda$ and~$\tau_\lambda$ of~$[n]$~by
\[
\sigma_\lambda(j) \eqdef 
\begin{cases}
\text{$r$-th element of~$\lambda^+$} & \text{if $i$ is the $r$-th element of~$\lambda^-$} \\
\text{$r$-th element of~$[n] \ssm \lambda^+$} & \text{if $i$ is the $r$-th element of~$[n] \ssm \lambda^-$} \\
\end{cases}
\]
and
\[
\tau_\lambda(j) \eqdef 
\begin{cases}
i & \text{if } j \notin \lambda^- \text{ and } (i,j) \in \lambda \\
\text{maximum of the connected component of~$j$ in~$\lambda$} & \text{if } j \notin \lambda^-
\end{cases}
\]
and the set of permutations
\[
C_\lambda \eqdef \set{\sigma \in \Ss_n}{\Epos(\sigma) = [n] \ssm \lambda^- \text{ and } \Eval(\sigma) = [n] \ssm \lambda^+ }.
\]
\end{definition}

Note that the permutations~$\tau_\lambda$ where already considered by M.~Zinno~\cite{Zinno} and known to provide a basis of the Temperley--Lieb algebra (see \cref{sec:TL} for details).
N.~Bergeron and L.~Gagnon extended this as follows~\cite{BergeronGagnon}.

\begin{theorem}[\cite{BergeronGagnon}]
\label{thm:BG}
For any~$n \in \N$,
\begin{enumerate}
\item For any~$\lambda \in \NCP_n$, the permutation~$\sigma_\lambda$ (resp.~$\tau_\sigma$) avoids the pattern~$321$ (resp.~$3412$).
\item For any~$\lambda \in \NCP_n$, the set~$C_\lambda$ is the interval~$[\sigma_\lambda, \tau_\lambda]$ of the Bruhat order on~$\Ss_n$.
\item The map~$\lambda \mapsto C_\lambda$ is a bijection from the NCPs of~$[n]$ to the excedance classes of~$\Ss_n$.
\item The quotient~$\Ss_n/{\equiv}$ is isomorphic to the Stanley lattice: $\lambda \le \mu$ if and only if~${C_\lambda \le C_\mu}$.
\item Any choice of representatives of the excedance classes yields a basis of the Temperley--Lieb algebra~$\TL_n(2)$ (see \cref{sec:TL} for details).
\end{enumerate}
\end{theorem}


\subsection{The excedance congruence on~$\ASM$}
\label{subsec:excedanceCongruence}

We now extend the excedance relation on permutations to a lattice congruence on~$\ASM_n$.
For this, we consider the diagonal and the superdiagonal of the corresponding HFMs.

\begin{definition}
The \defn{excedance congruence} of~$\HFM_n$ is the equivalence relation~$\equiv$ defined by~$H \equiv H'$ if and only if~$H_{i,j} = H'_{i,j}$ for all~$0 \le i,j \le n$ with~$j-i \in \{0,1\}$.
The \defn{excedance congruence} of~$\ASM_n$ is the corresponding equivalence relation~$\equiv$ on~$\ASM_n$.
\end{definition}

\begin{lemma}
The excedance congruence is a lattice congruence on~$\HFM_n$ (or~$\ASM_n$).
\end{lemma}

\begin{proof}
Let~$H \equiv H'$ and~$K \equiv K'$ in~$\HFM_n$.
For any~$0 \le i, j \le n$ with~$j-i \in \{0,1\}$, we have~$H_{i,j} = H'_{i,j}$ and~$K_{i,j} = K'_{i,j}$, hence~$(H \vee K)_{i,j} = \max(H_{i,j}, K_{i,j}) = \max(H'_{i,j}, K'_{i,j}) = (H' \vee K')_{i,j}$.
We conclude that~$H \vee K \equiv H' \vee K'$.
The proof is symmetric for the meet.
\end{proof}

We now check that the excedance relation on alternating sign matrices restricts to the excedance relation on permutations.
Denote by~$P_\sigma$ the permutation matrix of~$\sigma \in \Ss_n$, that is, where~$(P_\sigma)_{i,j} = \delta_{\sigma_i,j}$ for all~$i,j \in [n]$.

\begin{lemma}
\label{lem:excedanceRelationVSexcedanceCongruence}
For any permutations~$\sigma, \tau \in \Ss_n$, we have~$\sigma \equiv \tau$ if and only if~$P_\sigma \equiv P_\tau$.
\end{lemma}

\begin{proof}
Note that for an ASM~$A$ with HFM~$H$, we have~$H_{i,i}-H_{i,i+1} = 2\sum_{i' < i} A_{i',i} - 1$.
For~$A = P_\sigma$, we thus obtain that~$H_{i,i}-H_{i,i+1} = 1$ if and only if~$i \in \Eval(\sigma)$.
Similarly, $H_{i,i+1}-H_{i+1,i+1} = 1$ if and only if~$i \in \Epos(\sigma)$.
Finally, recall that~$H_{i,i}-H_{i,i+1} \in \{-1,1\}$ and~$H_{i,i+1}-H_{i+1,i+1} \in \{-1,1\}$.
Consequently, the values~$H_{i,j}$ for~$i,j \in [n]$ with~$j-i \in \{0,1\}$ and the pair~$\big( \Epos(\sigma), \Eval(\sigma) \big)$ determine each other.
\end{proof}

We will now describe the classes of the excedance congruence.
For this, define the \defn{height sequence} of a Dyck path as the sequence of heights of its integer points (see \cref{fig:ASM2NCP}).

\begin{lemma}
\label{lem:heightSequence}
For any HFM~$H$, the sequence~$H_{1,1}, H_{1,2}, H_{2,2}, H_{2,3}, \dots, H_{n-1,n-1}, H_{n-1,n}, H_{n,n}$ is the height sequence of a Dyck path.
\end{lemma}

\begin{proof}
It immediately follows from the fact that~$H_{1,1} = 0 = H_{n,n}$ and that~$H_{i,j} \ge 0$ and~$H_{i,j} - H_{i,j+1} \in \{-1,1\}$ and~$H_{i,j} - H_{i+1,j} \in \{-1,1\}$ for any~$i,j \in [n]$.
\end{proof}

Combining \cref{lem:heightSequence} with the bijections between ASMs and HFMs and between NCPs, Dyck paths, and height sequences, we obtain a map~$A \mapsto \b{\lambda}(A)$ from ASMs to NCPs.
See \cref{fig:ASM2NCP} for an illustration.

\begin{definition}
\begin{figure}
	\centerline{
	$\begin{pmatrix} \cdot & + & \cdot & \cdot & \cdot & \cdot & \cdot \\ \cdot & \cdot & \cdot & \cdot & \cdot & + & \cdot \\ \cdot & \cdot & + & \cdot & \cdot & - & + \\ \cdot & \cdot & \cdot & \cdot & + & \cdot & \cdot \\ + & \cdot & - & + & \cdot & \cdot & \cdot \\ \cdot & \cdot & \cdot & \cdot & \cdot & + & \cdot \\ \cdot & \cdot & + & \cdot & \cdot & \cdot & \cdot \\ \end{pmatrix}$\quad
	$\begin{pmatrix} \color{red}\b{0} & \color{red}\b{1} & 2 & 3 & 4 & 5 & 6 & 7 \\ 1 & \color{red}\b{2} & \color{red}\b{1} & 2 & 3 & 4 & 5 & 6 \\ 2 & 3 & \color{red}\b{2} & \color{red}\b{3} & 4 & 5 & 4 & 5 \\ 3 & 4 & 3 & \color{red}\b{2} & \color{red}\b{3} & 4 & 5 & 4 \\ 4 & 5 & 4 & 3 & \color{red}\b{4} & \color{red}\b{3} & 4 & 3 \\ 5 & 4 & 3 & 4 & 3 & \color{red}\b{2} & \color{red}\b{3} & 2 \\ 6 & 5 & 4 & 5 & 4 & 3 & \color{red}\b{2} & \color{red}\b{1} \\ 7 & 6 & 5 & 4 & 3 & 2 & 1 & \color{red}\b{0} \\ \end{pmatrix}$\quad	
	\raisebox{-2cm}{\includegraphics[scale=.95]{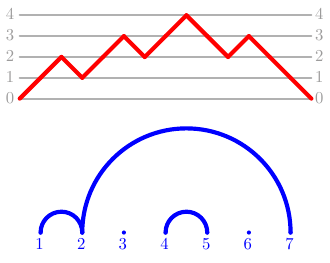}}
	}
	\caption{An ASM (left), its HFM (middle), its Dyck path (top right), and its NCP (bottom right).}
	\label{fig:ASM2NCP}
\end{figure}
For~$\lambda \in \NCP_n$, we define the set of ASMs
\[
D_\lambda \eqdef \set{A \in \ASM_n}{\b{\lambda}(A) = \lambda}.
\]
\end{definition}

\begin{lemma}
\label{lem:restrictionClasses}
For any~$\lambda \in \NCP_n$, we have~$C_\lambda = \set{\sigma \in \Ss_n}{P_\sigma \in D_\lambda}$.
\end{lemma}

\begin{proof}
Same arguments as in \cref{lem:excedanceRelationVSexcedanceCongruence}.
\end{proof}

\begin{lemma}
The map~$\lambda \mapsto D_\lambda$ is a bijection from the NCPs of~$[n]$ to the excedance classes of~$\ASM_n$.
\end{lemma}

\begin{proof}
For~$A,A' \in \ASM_n$ with corresponding~$H,H' \in \HFM_n$, we have $\b{\lambda}(A) = \b{\lambda}(A')$ if and only if~$H$ and~$H'$ have the same diagonal and superdiagonal, that is~$H \equiv H'$, that is~$A \equiv A'$.
\end{proof}

\begin{lemma}
\label{lem:maxBG}
For any NCP~$\lambda$, the maximal ASM of the class~$D_\lambda$ is the permutation matrix~$P_{\tau_\lambda}$.
\end{lemma}

\begin{proof}
We skip the proof that the maxima of classes are~always permutation matrices as it will be proved in more generality in \cref{sec:maxima}.
The result then follows from \cref{lem:restrictionClasses,thm:BG}\,(2).
\end{proof}

In contrast, note that the minima of excedance classes on~$\ASM_n$ are not permutation matrices in general, although it follows from~\cref{lem:restrictionClasses,thm:BG}\,(2) that each excedance class~$D_\lambda$ contains a minimal permutation matrix~$P_{\sigma_\lambda}$.

\begin{corollary}
The lattice quotient~$\ASM_n/{\equiv}$ is isomorphic to the Stanley lattice: $\lambda \le \mu$ if and only if~$D_\lambda \le D_\mu$.
\end{corollary}

\begin{proof}
This follows from \cref{lem:maxBG,thm:BG}\,(4).
\end{proof}

The goal of this paper is to show that all properties of \cref{thm:BG} actually hold for any lattice quotient of~$\ASM_n$ isomorphic to~$\Dyck_n$.


\section{Catalan congruences of~$\ASM_n$}
\label{sec:quotientsASMs}

This paper focuses on the following generalization of \cref{subsec:excedanceCongruence}.

\begin{definition}
A \defn{catalan congruence} is a congruence of the lattice~$\ASM_n$ on alternating sign matrices whose quotient is isomorphic to the Stanley lattice~$\Dyck_n$ on Dyck paths.
\end{definition}

The two catalan congruences when~$n = 3$ are represented in \cref{fig:ASMQuotients}.
In this section, we describe all catalan congruences and encode them bijectively using several combinatorial structures, namely walls, depth triangles, catalan triangles and bicolored pipe dreams.

\begin{figure}[h]
	\centering
	\includegraphics[scale=.7]{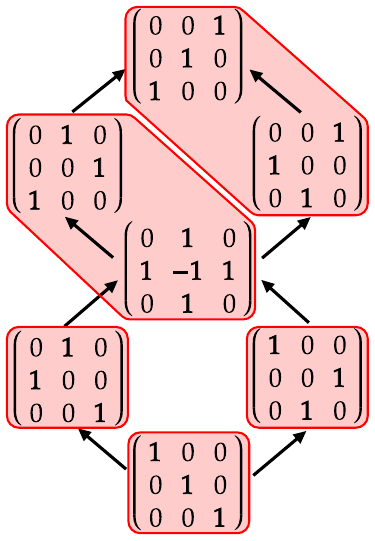}
	\qquad\qquad
	\includegraphics[scale=.7]{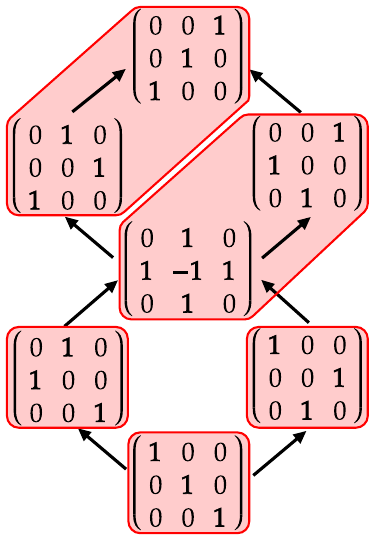}
	\caption{The two lattice congruences~$\ASM_3$ whose quotient are isomorphic to~$\Dyck_3$.}
	\label{fig:ASMQuotients}
\end{figure}


\subsection{Quotients of distributive lattices}
\label{subsec:quotientsDistributiveLattices}

Quotients of distributive lattices are distributive (by the definition of~$\vee$ and~$\wedge$ on~$\L/{\equiv}$), and behave properly with respect to Birkhoff's fundamental theorem of distributive lattices (\cref{thm:ftfdl}).
Recall that we denote by~$j_\star$ the single element covered by a join irreducible~$j$.

\begin{theorem}
\label{thm:congruencesDistributiveLattices1}
For a congruence~$\equiv$ on a distributive lattice~$\L$, denote by~$\JIrr(\equiv)$ the subposet of~$\JIrr(\L)$ induced by the join irreducibles~$j$ such that~$j \not\equiv j_\star$.
Then
\begin{itemize}
\item an element $x \in \L$ is minimal in its congruence class if and only if the corresponding antichain~$\jAnti(x)$ is contained in~$\JIrr(\equiv)$,
\item the quotient~$\L/{\equiv}$ is isomorphic to the distributive lattice of lower sets of~$\JIrr(\equiv)$.
\end{itemize}
Dual statements hold for meet irreducibles.
\end{theorem}

This yields in particular the following statement.

\begin{theorem}[\cite{Th50}]
\label{thm:congruencesDistributiveLattices2}
Given two finite distributive lattices~$\L$ and~$\K$, the lattice congruences of~$\L$ whose quotients are isomorphic to~$\K$ are in bijection with the subposets of $\JIrr(\L)$ isomorphic to $\JIrr(\K)$.
\end{theorem}


\subsection{Walls}
\label{subsec:walls}

We now specialize \cref{thm:congruencesDistributiveLattices2} to understand all catalan congruences.
We have seen in \cref{sec:distributiveLattices} that $\JIrr(\ASM_n)$ is isomorphic to the tetrahedron poset~$\TPoset_n$, and $\JIrr(\Dyck_n)$ is isomorphic to the triangular poset~$\DPoset_n$ (see \cref{fig:posets}).
Hence, by \cref{thm:congruencesDistributiveLattices2}, quotients of $\ASM_n$ isomorphic to $\Dyck_n$ are in bijection with subposets of~$\TPoset_n$ isomorphic to~$\DPoset_n$.
We now show that this isomorphism always factorizes through the horizontal projection $(x_1, x_2, x_3, x_4) \mapsto (x_1, x_2+x_3, x_4)$ of~$\TPoset_n$.

\begin{proposition}
\label{prop:walls}
Let $\P$ be a subposet of the tetrahedron poset~$\TPoset_n$ isomorphic to the triangular poset~$\DPoset_n$.
For any~$(y_1, y_2, y_3)\in \DPoset_n$, the subposet~$\P$ contains exactly one element $(x_1, x_2, x_3, x_4)$ such that $x_1 = y_1$ and $x_4 = y_3$ (and thus~$x_2+x_3 = y_2$).
\end{proposition}

\begin{proof}
The posets~$\TPoset_n$ and $\DPoset_n$ are both ranked of height~$n-2$ and therefore any morphism must preserve the rank.
Namely, $\TPoset_n$ has rank function $(x_1, x_2, x_3, x_4) \mapsto x_2+x_3$ and $\DPoset_n$ has rank function $(y_1, y_2, y_3)\mapsto y_2$.
Hence, any subposet~$\P\subset \TPoset_n$ isomorphic to $\DPoset_n$ must contain $n-1-i$ elements of rank $i$ for all~$0 \le i \le n-2$.
Therefore, $\P$ cannot contain two distinct elements $\mathbf x$ and $\mathbf w$ where $x_1=w_1$ and $x_4=w_4$, since there would be $x_2+x_3+1$ elements of rank $0$ lesser than $\mathbf x$ or $\mathbf w$ in $\TPoset_n$, but two distinct elements of rank $x_2+x_3$ in $\DPoset_n$ should have at least $x_2+x_3+2$ elements of rank $0$ lesser than them.
\end{proof}

In other words, \cref{prop:walls} states that the subposets of~$\TPoset_n$ isomorphic to~$\DPoset_n$ can be seen as vertical (possibly undulating) sections of~$\TPoset_n$, which motivates the following name.

\begin{definition}
A \defn{wall} of order~$n$ is a subposet of~$\TPoset_n$ isomorphic to~$\DPoset_n$.
\end{definition}

\begin{example}
\label{exm:excedanceQuotientWall}
The \defn{excedance quotient} of~\cref{subsec:excedanceCongruence} is given by the \defn{undulating wall}
\[
\set{(x_1,x_2,x_3,x_4) \in \TPoset_n}{x_3-x_2 \in \{0,-1\}}.
\]
See \cref{fig:excedanceQuotient}\,(left).
\end{example}

We will also need the following observation.
Visually, it states that a wall cannot enter the interior of the black butterfly illustrated in \cref{fig:butterflyHourglass}\,(left) at each of its points.

\begin{lemma}
\label{lem:wall}
If $(x_1,x_2,x_3,x_4)$ and $(x'_1,x'_2,x'_3,x'_4)$ are in a wall, then $x'_2-x_2$, $x_3-x'_3$, $x'_1+x'_2-x_1-x_2$ and $x_1+x_3-x'_1-x'_3$ cannot be all strictly positive or all strictly negative.
\end{lemma}

\begin{proof}
Consider~$\b{x} \eqdef (x_1,x_2,x_3,x_4)$ and $\b{x}' \eqdef (x'_1,x'_2,x'_3,x'_4)$ in a wall, and the corresponding projections~$\b{y} \eqdef (x_1, x_2+x_3, x_4)$ and~$\b{y}' \eqdef (x'_1, x'_2+x'_3, x'_4)$.
There is a (non-directed) path of cover relations of~$\DPoset_n$ joining~$\b{y}$ to~$\b{y'}$. 
Lifting this path to the wall, we obtain a path from~$\b{x}$ to~$\b{x}'$ which avoids the interior of the butterfly at~$\b{x}$.
\end{proof}


\subsection{Depth triangles}
\label{subsec:depthTriangles}

For any wall $\P\subset \TPoset_n$, \cref{prop:walls} ensures that the horizontal projection $(x_1,x_2,x_3,x_4) \mapsto (x_1,x_2+x_3,x_4)$ is a bijection from $\P$ to $\DPoset_n$.
We can therefore encode the wall~$\P$ by recording the depth~$x_3-x_2$ of the point~$(x_1,x_2,x_3,x_4)$ in the direction of the projection.
See \cref{fig:catalanTriangle}.
We now characterize the triangular arrays obtained this way.

\pagebreak
\begin{definition}
\label{def:depthTriangle}
A \defn{depth triangle} of order~$n$ is a map~$\delta : \DPoset_n \to \mathbb Z$ such that
\begin{itemize}
\item the bottom row is~$0 0 \cdots 0$, that is, $\delta(i, 0, n-2-i) = 0$ for all~$0 \le i \le n-2$,
\item $|\delta(y_1, y_2, y_3) - \delta(y_1, y_2-1, y_3+1)| = 1 = |\delta(y_1, y_2, y_3) - \delta(y_1+1, y_2-1, y_3)|$.
\end{itemize}
Note that a depth triangle of order~$n$ has~$n-1$ rows (or diagonals).
See \cref{fig:depthTriangles}.
\begin{figure}[H]
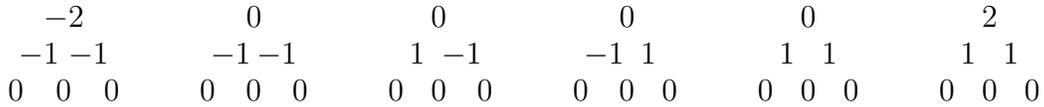

	\centering
	\setlength{\tabcolsep}{1.1pt} 
	\begin{tabular}{ccccc}
		& & $\!\!-2\!\!$ & & \\
		& $\!\!-1\!\!$ & & $\!\!-1\!\!$ & \\
		$0$ & & $0$ & & $0$
	\end{tabular}
	\qquad
	\begin{tabular}{ccccc}
		& & $0$ & & \\
		& $\!\!-1\!\!$ & & $\!\!-1\!\!$ & \\
		$0$ & & $0$ & & $0$
	\end{tabular}
	\qquad
	\begin{tabular}{ccccc}
		& & $0$ & & \\
		& $1$ & & $\!\!-1\!\!$ & \\
		$0$ & & $0$ & & $0$
	\end{tabular}
	\qquad
	\begin{tabular}{ccccc}
		& & $0$ & & \\
		& $\!\!-1\!\!$ & & $1$ & \\
		$0$ & & $0$ & & $0$
	\end{tabular}
	\qquad
	\begin{tabular}{ccccc}
		& & $0$ & & \\
		& $1$ & & $1$ & \\
		$0$ & & $0$ & & $0$
	\end{tabular}
	\qquad
	\begin{tabular}{ccccc}	
		& & $2$ & & \\
		& $1$ & & $1$ & \\
		$0$ & & $0$ & & $0$
	\end{tabular}
	\caption{The $6$ depth triangles of order $4$.}
	\label{fig:depthTriangles}
\end{figure}
\end{definition}

\begin{example}
\label{exm:excedanceQuotientDepth}
Following~\cref{exm:excedanceQuotientWall}, the depth triangle of the excedance quotient of \cref{subsec:excedanceCongruence} is given by~$\delta(y_1, y_2, y_3) = 2 \lfloor y_2/2 \rfloor - y_2 = - (y_2 \mod 2)$ (in other words, odd rows are~$0$ while even rows are~$-1$).
See \cref{fig:excedanceQuotient}\,(middle left).
\end{example}

\begin{theorem}
\label{thm:depthTriangles}
The catalan congruences of~$\ASM_n$ are in bijection with the depth triangles of order~$n$.
\end{theorem}

\begin{proof}
We have shown in~\cref{prop:walls} that subposets $\P \subset \TPoset_n$ isomorphic to $\DPoset_n$ can be encoded by triangles of numbers.
Namely, if~$(y_1,y_2,y_3) \in \DPoset_n$ and $(x_1,x_2,x_3,x_4)$ is the only element of $\P$ such that $x_1=y_1$ and $x_4=y_3$, then we define the depth~$\delta(y_1, y_2, y_3) =x_3-x_2$.
See \cref{fig:catalanTriangle} for an illustration.
These triangles are clearly exactly the depth triangles.
\qedhere
\begin{figure}
	\centerline{
		\raisebox{-2cm}{\includegraphics[scale=.7]{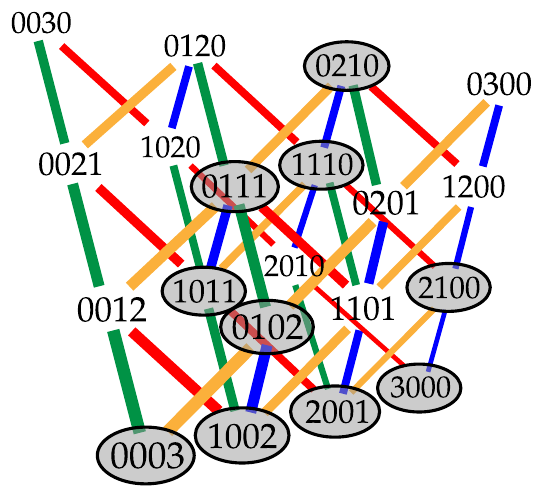}}
		\quad
		\setlength{\tabcolsep}{1.1pt}
		\begin{tabular}{ccccccc}
			& & & $\!\!-1\!\!$ & & & \\
			& & $0$ & & $0$ & & \\
			& $\!\!-1\!\!$ & & $1$ & & $\!\!-1\!\!$ & \\
			$0$ & & $0$ & & $0$ & & $0$
		\end{tabular}$$
		\qquad
		\begin{tabular}{ccccccc}
			& & & $2$ & & & \\
			& & $2$ & & $3$ & & \\
			& $1$ & & $3$ & & $3$ & \\
			$1$ & & $2$ & & $3$ & & $4$
		\end{tabular}$$
		\qquad
		\raisebox{-.6cm}{\includegraphics[scale=.5]{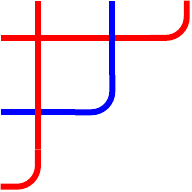}}
	}
	\caption{A subposet of $\TPoset_5$ isomorphic to $\DPoset_5$ (left) and its corresponding depth triangle (middle left), catalan triangle (middle right), and bicolored pipe dream (right).}
	\label{fig:catalanTriangle}
\end{figure}
\end{proof}

\begin{proposition}
\label{prop:depthTriangleInequalities}
For any depth triangle~$\delta$, we have
\begin{itemize}
\item $-y_2 \le \delta(y_1, y_2, y_3) \le y_2$ and $\delta(y_1, y_2, y_3) \equiv_{\mathrm{mod} \; 2} y_2$ for all~$(y_1, y_2, y_3) \in \DPoset_n$, and
\item $|\delta(z_1, z_2, z_3) - \delta(y_1, y_2, y_3)| \le |y_1 - z_1| + |y_3 - z_3|$ for all $(y_1, y_2, y_3), (z_1, z_2, z_3) \in \DPoset_n$.
\end{itemize}
\end{proposition}

\begin{proof}
Immediately follows from the two conditions of \cref{def:depthTriangle}.
\end{proof}


\subsection{Catalan triangles}
\label{subsec:catalanTriangles}

By an appropriate shift of the entries, we now connect the depths triangles of \cref{subsec:depthTriangles} to certain Gelfand--Tsetlin patterns.
A \defn{Gelfand--Tsetlin pattern} of order~$n$ is a map~$\gamma : \DPoset_n \to \N$ subject to the interlacing conditions
\[
\gamma(y_1, y_2-1, y_3+1) \le \gamma(y_1, y_2, y_3) \le \gamma(y_1+1, y_2-1, y_3).
\]
Again, note that a Gelfand--Tsetlin pattern of order~$n$ has~$n-1$ rows (or diagonals).
We say that~$\gamma$ is \defn{doubly gapless} if moreover
\[
\gamma(y_1, y_2, y_3) - \gamma(y_1, y_2-1, y_3+1) \le 1
\quad\text{and}\quad
\gamma(y_1+1, y_2-1, y_3) - \gamma(y_1, y_2, y_3) \le 1.
\]
We note that our definition of Gelfand--Tsetlin patterns only differs from the classical definition by our use of barycentric coordinates to parametrize~$\DPoset_n$.
We also observe that our two conditions for doubly gapless patterns are usually called gapless and left-gapless.
Gapless triangles were introduced by A.~Ayyer, R.~Cori and D.~Gouyou-Beauchamps \cite{ACGB11}, who described a bijection between gapless Gog triangles and gapless Magog triangles.


\begin{definition}
\label{def:catalanTriangle}
A \defn{catalan triangle} of order~$n$ is a doubly gapless Gelfand--Tsetlin pattern with bottom row $1 2 \cdots (n-1)$.
See \cref{fig:catalanTriangles}.
\begin{figure}[H]
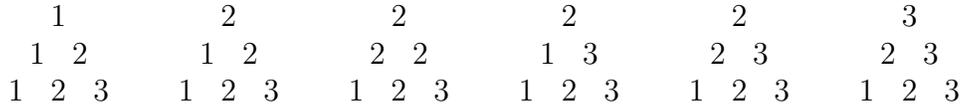

	\centering
	\setlength{\tabcolsep}{1.1pt} 
	\begin{tabular}{ccccc}
		& & $1$ & & \\
		& $1$ & & $2$ & \\
		$1$ & & $2$ & & $3$
	\end{tabular}\qquad
	\begin{tabular}{ccccc}
		& & $2$ & & \\
		& $1$ & & $2$ & \\
		$1$ & & $2$ & & $3$
	\end{tabular}\qquad
	\begin{tabular}{ccccc}
		& & $2$ & & \\
		& $2$ & & $2$ & \\
		$1$ & & $2$ & & $3$
	\end{tabular}\qquad
	\begin{tabular}{ccccc}
		& & $2$ & & \\
		& $1$ & & $3$ & \\
		$1$ & & $2$ & & $3$
	\end{tabular}\qquad
	\begin{tabular}{ccccc}
		& & $2$ & & \\
		& $2$ & & $3$ & \\
		$1$ & & $2$ & & $3$
	\end{tabular}\qquad
	\begin{tabular}{ccccc}
		& & $3$ & & \\
		& $2$ & & $3$ & \\
		$1$ & & $2$ & & $3$
	\end{tabular}
	\caption{The $6$ catalan triangles of order $4$.}
	\label{fig:catalanTriangles}
\end{figure}
\end{definition}

\begin{example}
\label{exm:excedanceQuotientCatalan}
Following~\cref{exm:excedanceQuotientWall,exm:excedanceQuotientDepth}, the catalan triangle of the excedance quotient of \cref{subsec:excedanceCongruence} is given by~$\gamma(y_1, y_2, y_3) = y_1 + \lfloor y_2/2 \rfloor + 1$.
See \cref{fig:excedanceQuotient}\,(middle right).
\end{example}

\begin{proposition}
\label{prop:depthVSCatalanTriangles}
The depth triangles of order~$n$ are in bijection with the catalan triangles of order~$n$.
\end{proposition}

\begin{proof}
The bijection is given by the formula~$\gamma(y_1, y_2, y_3) = \delta(y_1, y_2, y_3)/2 + y_1 + y_2/2 + 1$.
See \cref{fig:bijections}.
\end{proof}

\begin{proposition}
\label{prop:catalanTriangleInequalities}
For any catalan triangle~$\gamma$, we have
\begin{itemize}
\item $y_1+1 \le \gamma(y_1, y_2, y_3) \le y_1+y_2+1 = n-1-y_3$ for all~$(y_1, y_2, y_3) \in \DPoset_n$, and
\item $0 \le \gamma(z_1, z_2, z_3) - \gamma(y_1, y_2, y_3) \le z_1 - y_1 + y_3 - z_3$ for all~$(y_1, y_2, y_3), (z_1, z_2, z_3) \in \DPoset_n$ with~$y_1 \le z_1$ and~$y_3 \ge z_3$.
\end{itemize}
\end{proposition}

\begin{proof}
Immediately follows from the two conditions of \cref{def:depthTriangle}, or from the combination of~\cref{prop:depthTriangleInequalities,prop:depthVSCatalanTriangles}.
\end{proof}


\subsection{Bicolored pipe dreams}
\label{subsec:bicoloredPipeDreams}

\enlargethispage{.1cm}
We now connect depth triangles and catalan triangles to certain pipe dreams.
A \defn{pipe dream}~$P$ is a filling of a triangular shape with crossings~\cross{} and contacts~\elbow{} so that all pipes entering on the left side exit on the top side~\cite{BergeronBilley, KnutsonMiller-GroebnerGeometry}.
The \defn{contact graph} of~$P$ is the graph with a vertex for each pipe of~$P$ and an edge between two pipes if they share a contact.
We say that~$P$ is \defn{connected} (resp.~bipartite) if its contact graph is.

\begin{definition}
A \defn{bicolored pipe dream} is a pipe dream whose pipes are colored red or blue so that each contact is bichromatic.
In other words, it is a pipe dream together with a proper bicoloration of its contact graph.
Note that there is no color restriction on the crossings: they are allowed to be both monocolor or bicolor.
See \cref{fig:bicoloredPipeDreams}.
\begin{figure}[H]
	\centering
	\includegraphics[scale=.5]{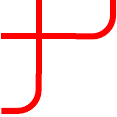} \qquad
	\includegraphics[scale=.5]{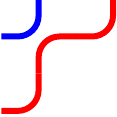} \qquad
	\includegraphics[scale=.5]{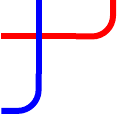} \qquad
	\includegraphics[scale=.5]{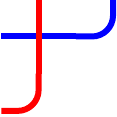} \qquad
	\includegraphics[scale=.5]{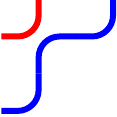} \qquad
	\includegraphics[scale=.5]{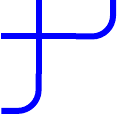}
	\caption{The $6$ bicolored pipe dreams with $2$ pipes.}
	\label{fig:bicoloredPipeDreams}
\end{figure}
\end{definition}

\begin{example}
\label{exm:excedanceQuotientPipeDream}
Following~\cref{exm:excedanceQuotientWall,exm:excedanceQuotientDepth,exm:excedanceQuotientCatalan}, the bicolored pipe dream of the excedance quotient of \cref{subsec:excedanceCongruence} is the pipe dream with no crossing (hence waving pipes alternating colors), whose longest pipe is blue.
See \cref{fig:excedanceQuotient}\,(right).
\begin{figure}
	\centerline{
	\raisebox{-1.5cm}{\includegraphics[scale=.1]{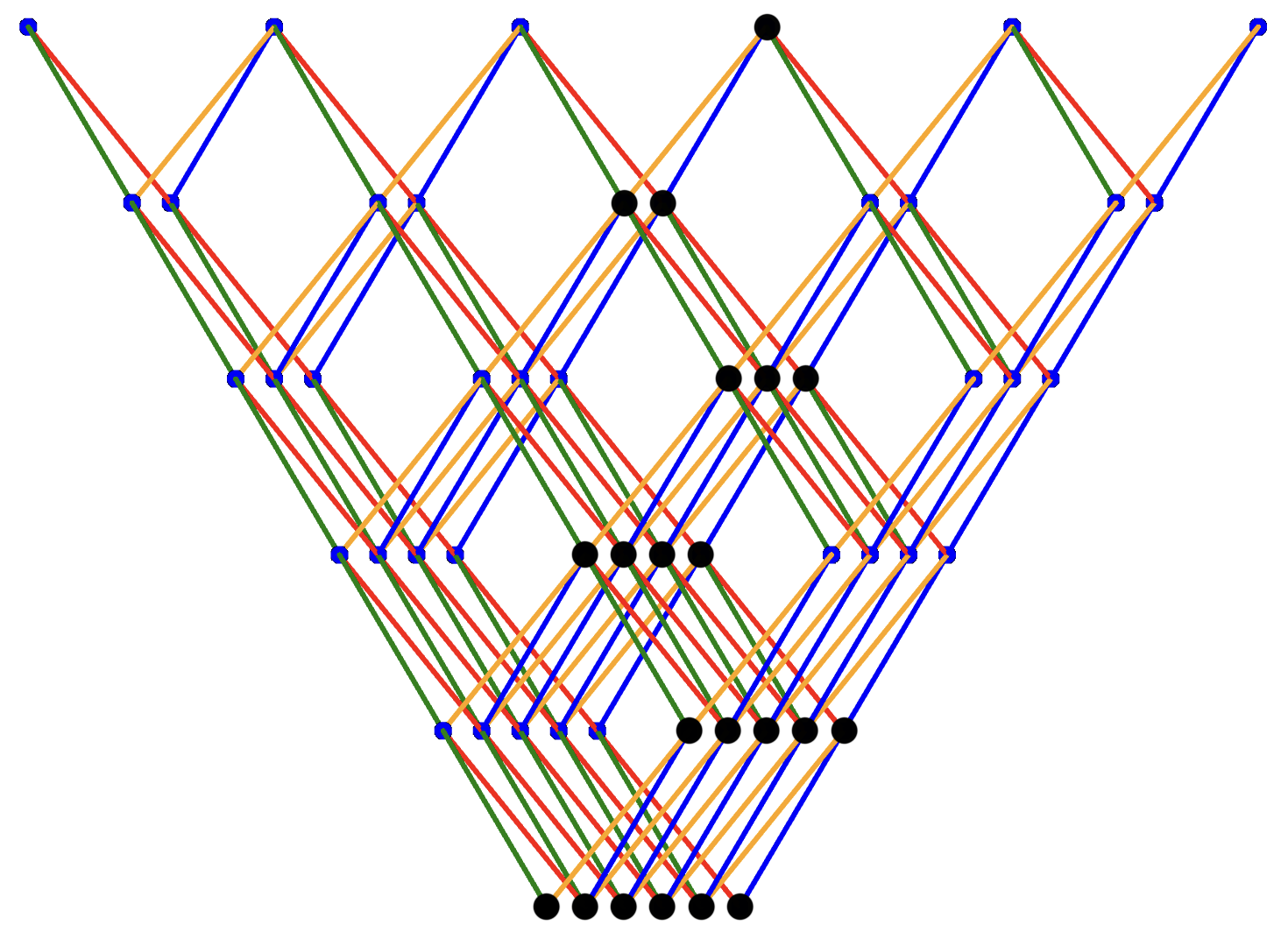}}
	\setlength{\tabcolsep}{1.1pt}
	\begin{tabular}{ccccccccccc}
		& & & & & $\!\!-1\!\!$ & & & & & \\
		& & & & $0$ & & $0$ & & & & \\
		& & & $\!\!-1\!\!$ & & $\!\!-1\!\!$ & & $\!\!-1\!\!$ & & & \\
		& & $0$ & & $0$ & & $0$ & & $0$ & & \\
		& $\!\!-1\!\!$ & & $\!\!-1\!\!$ & & $\!\!-1\!\!$ & & $\!\!-1\!\!$ & & $\!\!-1\!\!$ & \\
		$0$ & & $0$ & & $0$ & & $0$ & & $0$ & & $0$
	\end{tabular}
	\qquad
	\begin{tabular}{ccccccccccc}
		& & & & & $3$ & & & & & \\
		& & & & $3$ & & $4$ & & & & \\
		& & & $2$ & & $3$ & & $4$ & & & \\
		& & $2$ & & $3$ & & $4$ & & $5$ & & \\
		& $1$ & & $2$ & & $3$ & & $4$ & & $5$ & \\
		$1$ & & $2$ & & $3$ & & $4$ & & $5$ & & $6$
	\end{tabular}
	\qquad
	\raisebox{-1.3cm}{\includegraphics[scale=.5]{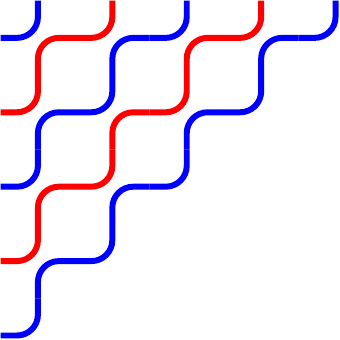}}
	}
	\caption{The wall (left), depth triangle (middle left), catalan triangle (middle right), and bicolored pipe dream (right) corresponding to the excedance quotient of~\cite{BergeronGagnon}.}
	\label{fig:excedanceQuotient}
\end{figure}
\end{example}

\begin{proposition}
\label{prop:depthVSBicoloredPipeDreams}
The depth triangles (or equivalently the catalan triangles) of order~$n$ are in bijection with the bicolored pipe dreams with $n-2$ pipes.
\end{proposition}

\begin{proof}
In a depth triangle, we draw a red (resp.~blue) line between two adjacent entries such that~$\delta(y_1, y_2, y_3) > \delta(y_1, y_2-1, y_3+1)$ or~$\delta(y_1, y_2, y_3) < \delta(y_1+1, y_2-1, y_3)$ (resp.~$\delta(y_1, y_2, y_3) < \delta(y_1, y_2-1, y_3+1)$ or~$\delta(y_1, y_2, y_3) > \delta(y_1+1, y_2-1, y_3)$).
In a catalan triangle, we draw a red (resp.~blue) line between two adjacent entries such that~$\delta(y_1, y_2, y_3) = \delta(y_1, y_2-1, y_3+1)$ or~$\delta(y_1, y_2, y_3) < \delta(y_1+1, y_2-1, y_3)$ (resp.~$\delta(y_1, y_2, y_3) < \delta(y_1, y_2-1, y_3+1)$ or~$\delta(y_1, y_2, y_3) = \delta(y_1+1, y_2-1, y_3)$).
See \cref{fig:bijections}.
%
%
%
%
\begin{figure}
	\centering
	\includegraphics[scale=.35]{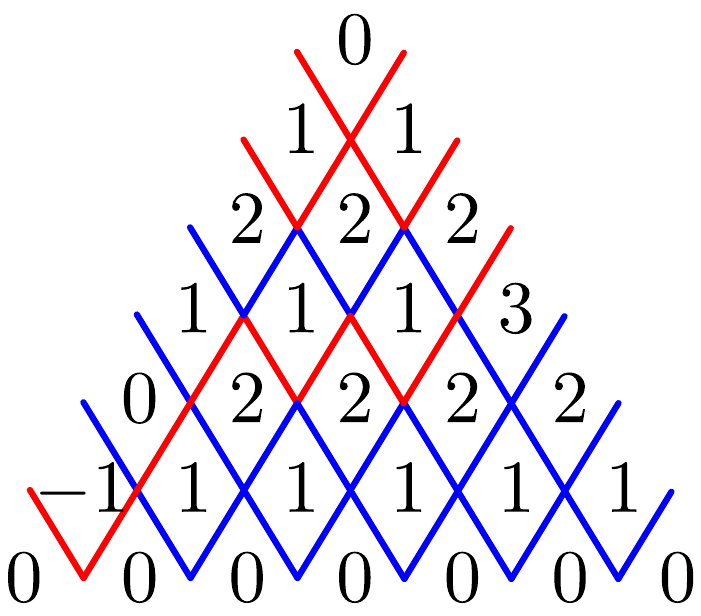} \qquad
	\includegraphics[scale=.35]{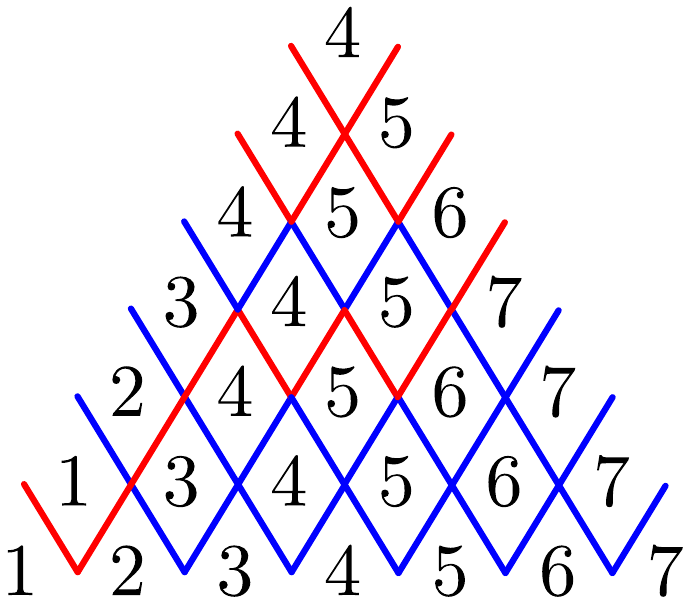} \qquad
	\includegraphics[scale=.5]{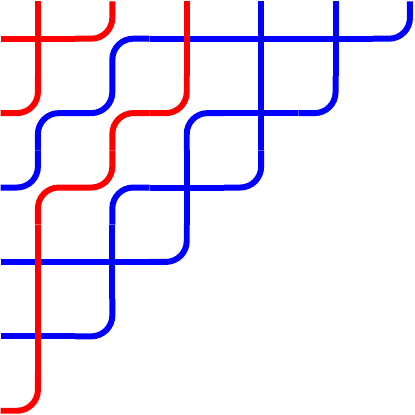}
	\caption{Bijection between depth triangles (left), catalan triangles (middle) and bicolored pipe dreams (right).}
	\label{fig:bijections}
\end{figure}
\end{proof}

By definition, a bicolored pipe dream is bipartite, but a bipartite pipe dream admits several bicolorations.
Namely, one can switch the two colors on any connected component of the contact graph and preserve a proper bicoloration.
This immediately yields the following statement connecting the numbers of \cref{table:CTBPD}.

\begin{table}[ht]
	\centering
	\begin{tabular}{c|ccccccccc}
		$n$ & $0$ & $1$ & $2$ & $3$ & $4$ & $5$ & $6$ & $7$ & $\cdots$ \\ \hline
		$\CT_n$ & $1$ & $2$ & $6$ & $28$ & $202$ & $2252$ & $38756$ & $1028964$ & $\cdots$ \\
		$\BPD_n$ & $1$ & $1$ & $2$ & $8$ & $57$ & $681$ & $12942$ & $379326$ & $\cdots$ \\
		$\BCPD_n$ & $0$ & $1$ & $1$ & $4$ & $31$ & $420$ & $8936$ & $287702$ & $\cdots$
	\end{tabular}
	\caption{Number of catalan triangles, bipartite pipe dreams and bipartite connected pipe dreams.}
	\label{table:CTBPD}
\end{table}

\begin{proposition}
Denote by
\[
\CT \eqdef \sum_{n\ge 0} \CT_n \frac{x^n}{n!},
\quad
\BPD \eqdef \sum_{n\ge 0} \BPD_n \frac{x^n}{n!}
\quad\text{and}\quad
\BCPD \eqdef \sum_{n\ge 0} \BCPD_n \frac{x^n}{n!}
\]
the exponential generating functions of the numbers of catalan triangles of order~$n$, and of the numbers of bipartite (resp.~bipartite connected) pipe dreams with $n-2$ pipes.
Then
\[
\CT = \exp \bigl( 2\BCPD \bigr) = \BPD^2.
\]
\end{proposition}

\begin{proof}
Since a bipartite pipe dream is a shuffle of connected bipartite pipe dreams, we have~$\exp \bigl( \BCPD \bigr) = \BPD$.
Since each connected component of a bipartite pipe dream can be colored two ways, we have~$\CT = \exp \bigl( 2\BCPD \bigr)$.
The result follows.
\end{proof}

We also obtain the following formula, similar to the 2-enumeration of ASMs.

\begin{proposition}
\label{prop:BooleanTriangles}
We have
\(
\sum_{P} 2^{\bcc(P)}=2^{\binom{n}{2}},
\)
where the sum ranges over the bicolored pipe dreams~$P$ with $n$ pipes, and~$\bcc(P)$ denotes the number of bicolored crossings of~$P$.
\end{proposition}

\begin{proof}
Given a triangle of $\binom{n}{2}$ Boolean variables, we construct a bicolored pipe dream as follows (see \cref{fig:BooleanTriangles}).
We start from the diagonal and take the Boolean values along the diagonal as the colors of the elbows.
Then we move upward in the northwest direction, applying the following rule at each intersection:
\begin{itemize}
\item If the two strands arriving from the south and the east have the same color, then we use the Boolean value at this intersection to choose between placing a crossing or a contact, and we color the strands leaving to the north and the west accordingly (same color as the arriving strands if we choose a crossing, distinct color from the arriving strands if we choose a contact).
\item If the two strands arriving from the south and the east have different colors, then we ignore the Boolean value at this intersection as we must place a crossing, and each strand keeps its color.
\end{itemize}
This procedure clearly constructs a bicolored pipe dream, and every such bicolored pipe dream~$P$ arises with multiplicity $2^{\bcc(P)}$, since each bicolored crossing of~$P$ corresponds to an ignored Boolean value during this procedure.
\begin{figure}
	\centering
	\setlength{\tabcolsep}{4pt}
	\begin{tabular}[b]{ccccccccccc}
		T & F & ? & T & T & T \\[.13cm]
		F & F & F & T & T \\[.13cm]
		F & F & T & T \\[.13cm]
		? & T & T \\[.13cm]
		? & T \\[.13cm]
		F
	\end{tabular}
	\qquad
	\includegraphics[scale=.5]{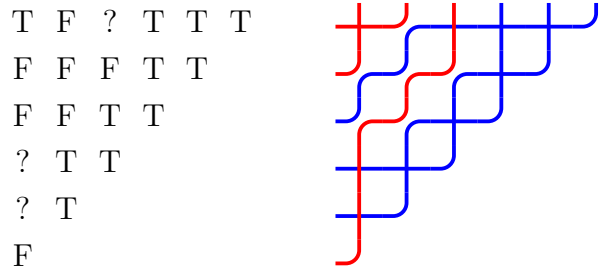}
	\caption{The map between Boolean triangles (left) and bicolored pipe dreams (right) described in the proof of \cref{prop:BooleanTriangles}. The question marks can be replaced by any Boolean values, since they are ignored during the procedure as the two strands arriving from the south and the east have different colors.}
	\label{fig:BooleanTriangles}
\end{figure}
\end{proof}


\subsection{The lattice of catalan triangles}
\label{subsec:latticeCatalanTriangles}

Finally, we note that the set of catalan triangles carries itself a natural distributive lattice structure, whose join irreducible poset is illustrated in \cref{fig:catalanLattice}.

\begin{proposition}
\label{prop:latticeCatalanTriangles}
The set of catalan triangles of order $n$ ordered by coordinatewise comparison forms a distributive lattice~$\Cat_n$, whose join-irreducibles poset is isomorphic to the poset of quadruples $(x_1,x_2,x_3,x_4)\in \N^4$ with $x_1+x_2+x_3+x_4=n-2$, ordered by $\mathbf x\le \mathbf y$ if and only if $x_2+x_3+x_4\le y_2+y_3+y_4$, $x_2+x_4\le y_2+y_4$, $x_3+x_4\le y_3+y_4$ and $x_4\le y_4$.
\end{proposition}

\begin{proof}
The set of catalan triangles of order $n$ is closed under coordinatewise maximum and minimum, making it a distributive lattice.
Each join irreducible catalan triangle corresponds to one entry of the depth triangle with non-minimal depth, which corresponds to elements of $\TPoset_n$ not in the minimal wall.
More precisely, to each quadruple $(x_1,x_2,x_3,x_4)$ with $x_1+x_2+x_3+x_4 = n-2$ corresponds the smallest catalan triangle with~$\gamma(x_3, x_1+x_4+1, x_2) = x_3+x_4+2$.
The order on join-irreducible catalan triangles is obtained by reorienting edges in $\TPoset_n$ horizontally.
\end{proof}

\begin{remark}
\label{rem:latticeDepthTriangles}
By \cref{prop:depthVSCatalanTriangles,prop:depthVSBicoloredPipeDreams}, \cref{prop:latticeCatalanTriangles} also translates to the depth triangles of \cref{subsec:depthTriangles} and the bicolored pipe dreams of \cref{subsec:bicoloredPipeDreams}.
\end{remark}

\begin{example}
\label{exm:minMaxCatalanTriangle}
The minimum (resp.~maximum) depth triangles, catalan triangles, and bicolored pipe dreams are given by:
\begin{itemize}
\item $\delta_{\min}(y_1,y_2,y_3) = -y_2$ (resp.~$\delta_{\max}(y_1,y_2,y_3) = y_2$),
\item $\gamma_{\min}(y_1,y_2,y_3) = y_1+1$ (resp.~$\gamma_{\max}(y_1,y_2,y_3) = y_1+y_2+1 = n-1-y_3$),
\item the completely red (resp.~blue) pipe dreams with only crossings.
\end{itemize}
See \cref{fig:minimalQuotient,fig:maximalQuotient}.
\end{example}

\begin{figure}
	\centerline{
	\raisebox{-1.5cm}{\includegraphics[scale=.1]{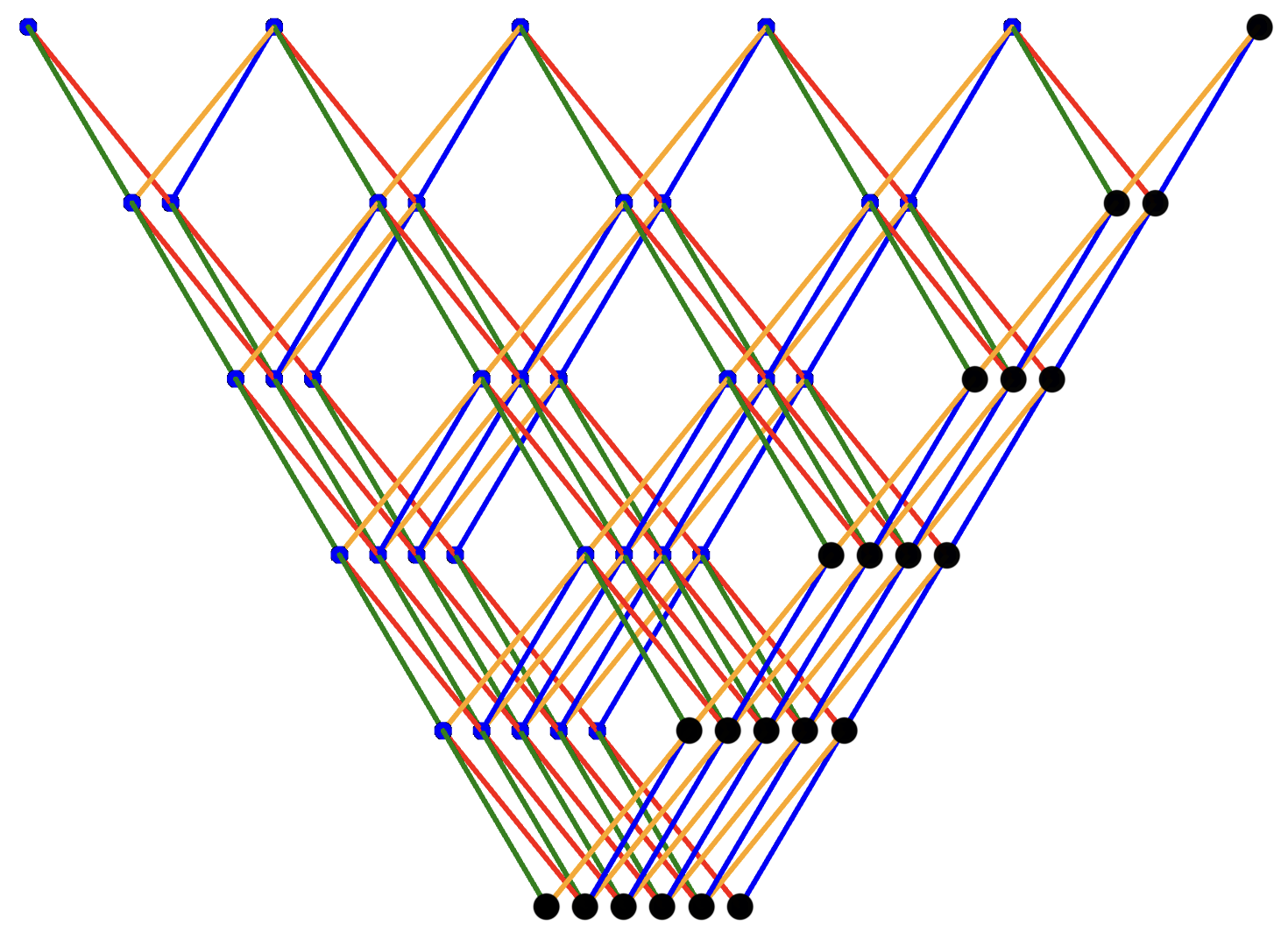}}
	\setlength{\tabcolsep}{1.1pt}
	\begin{tabular}{ccccccccccc}
		& & & & & $\!\!-5\!\!$ & & & & & \\
		& & & & $\!\!-4\!\!$ & & $\!\!-4\!\!$ & & & & \\
		& & & $\!\!-3\!\!$ & & $\!\!-3\!\!$ & & $\!\!-3\!\!$ & & & \\
		& & $\!\!-2\!\!$ & & $\!\!-2\!\!$ & & $\!\!-2\!\!$ & & $\!\!-2\!\!$ & & \\
		& $\!\!-1\!\!$ & & $\!\!-1\!\!$ & & $\!\!-1\!\!$ & & $\!\!-1\!\!$ & & $\!\!-1\!\!$ & \\
		$0$ & & $0$ & & $0$ & & $0$ & & $0$ & & $0$
	\end{tabular}
	\qquad
	\begin{tabular}{ccccccccccc}
		& & & & & $1$ & & & & & \\
		& & & & $1$ & & $2$ & & & & \\
		& & & $1$ & & $2$ & & $3$ & & & \\
		& & $1$ & & $2$ & & $3$ & & $4$ & & \\
		& $1$ & & $2$ & & $3$ & & $4$ & & $5$ & \\
		$1$ & & $2$ & & $3$ & & $4$ & & $5$ & & $6$
	\end{tabular}
	\qquad
	\raisebox{-1.3cm}{\includegraphics[scale=.5]{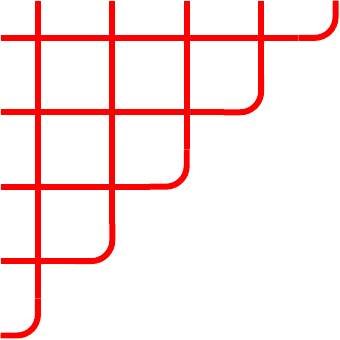}}
	}
	\caption{The minimal wall (left), depth triangle (middle left), catalan triangle (middle right), and bicolored pipe dream (right).}
	\label{fig:minimalQuotient}
\end{figure}

\begin{figure}
	\centerline{
	\raisebox{-1.5cm}{\includegraphics[scale=.1]{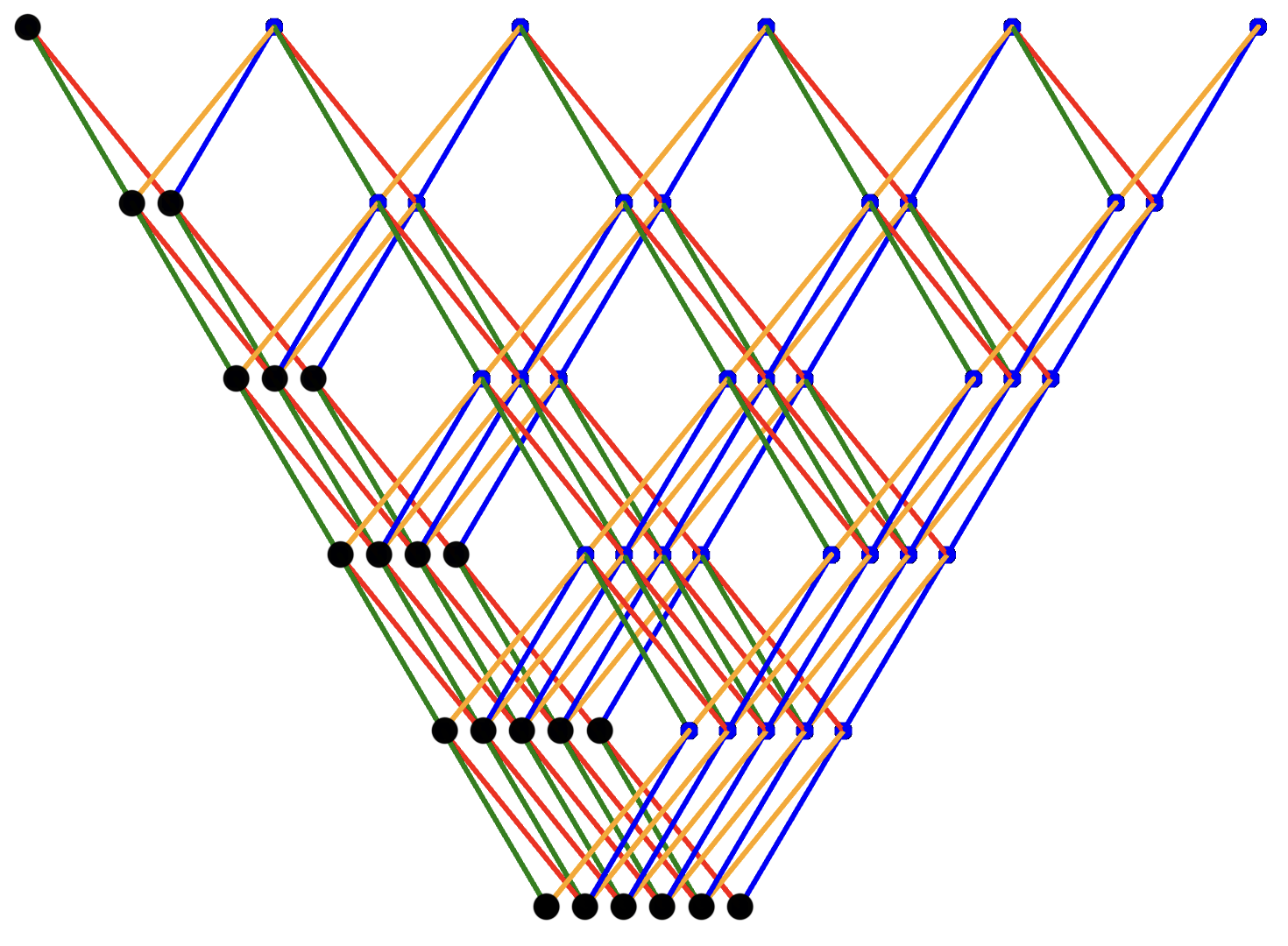}}
	\setlength{\tabcolsep}{1.1pt}
	\begin{tabular}{ccccccccccc}
		& & & & & $5$ & & & & & \\
		& & & & $4$ & & $4$ & & & & \\
		& & & $3$ & & $3$ & & $3$ & & & \\
		& & $2$ & & $2$ & & $2$ & & $2$ & & \\
		& $1$ & & $1$ & & $1$ & & $1$ & & $1$ & \\
		$0$ & & $0$ & & $0$ & & $0$ & & $0$ & & $0$
	\end{tabular}
	\qquad
	\begin{tabular}{ccccccccccc}
		& & & & & $6$ & & & & & \\
		& & & & $5$ & & $6$ & & & & \\
		& & & $4$ & & $5$ & & $6$ & & & \\
		& & $3$ & & $4$ & & $5$ & & $6$ & & \\
		& $2$ & & $3$ & & $4$ & & $5$ & & $6$ & \\
		$1$ & & $2$ & & $3$ & & $4$ & & $5$ & & $6$
	\end{tabular}
	\qquad
	\raisebox{-1.3cm}{\includegraphics[scale=.5]{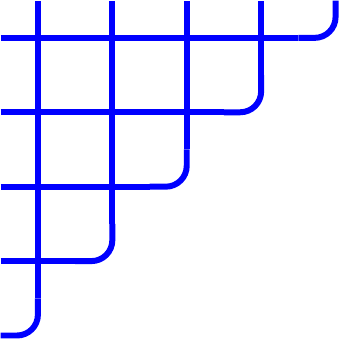}}
	}
	\caption{The maximal wall (left), depth triangle (middle left), catalan triangle (middle right), and bicolored pipe dream (right).}
	\label{fig:maximalQuotient}
\end{figure}

\begin{figure}
	\centerline{\includegraphics[scale=.8]{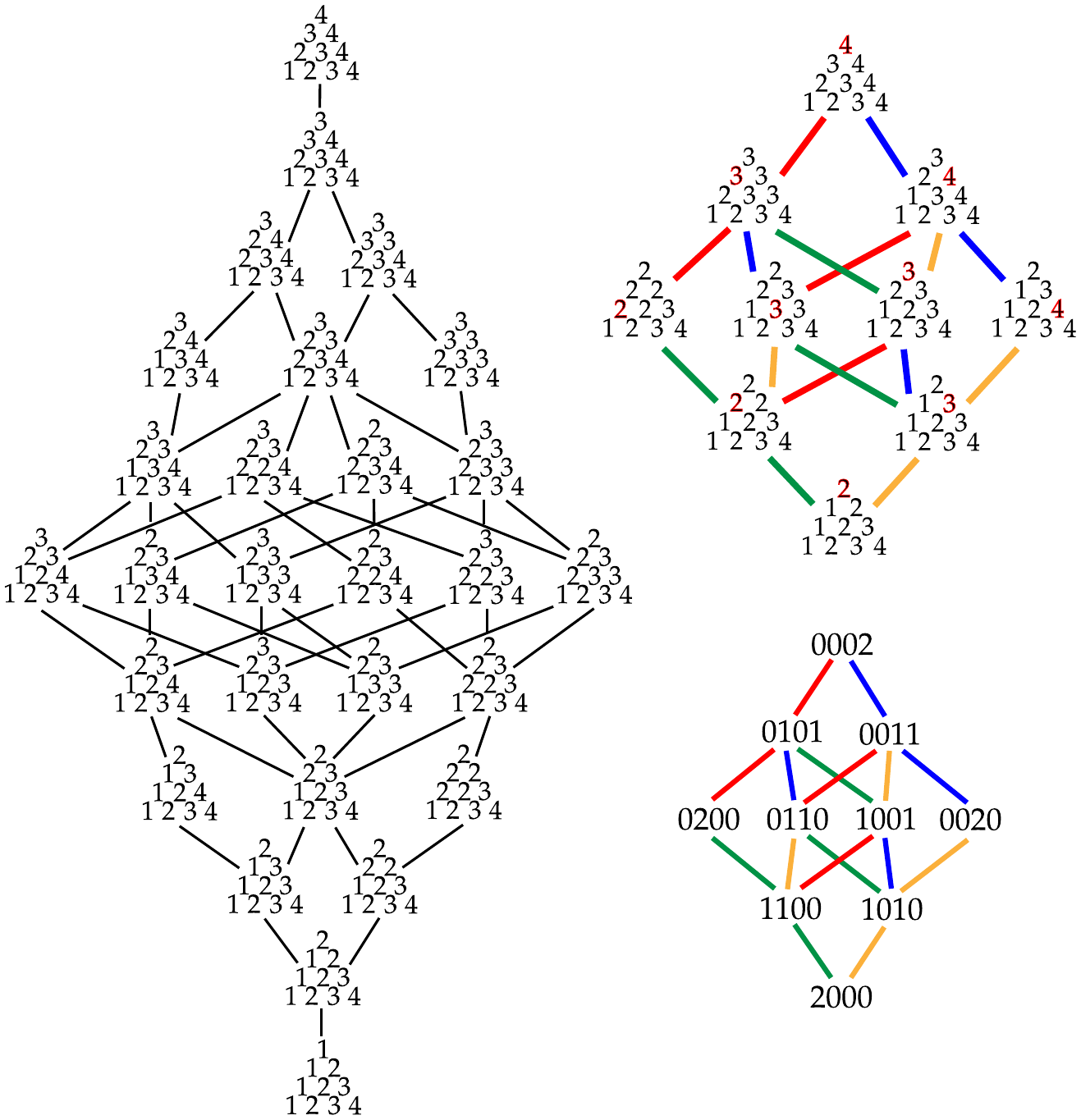}}
	\caption{The lattice of catalan triangles of order~$5$ (left), its join-irreducible poset (top right), and the corresponding barycentric coordinates (bottom right).}
	\label{fig:catalanLattice}
\end{figure}

\subsection{Symmetric and self-dual catalan triangles}
The lattice $\Cat_n$ has an automorphism given by $(X_{i,j})_{1\le j \le i \le n}\mapsto (X_{i,i-j+1}-i+2j-1)_{1\le j \le i \le n}$, corresponding to reflection of bicolored pipe dreams. We say a catalan triangle is \defn{symmetric} if is fixed by this automorphism. Symmetric catalan triangles form a sublattice of $\Cat_n$, whose poset of join-irreducible elements is the vertical half of $\JIrr(\Cat_n)$.

The lattice $\Cat_n$ has an antiautomorphism given by $(X_{i,j})_{1\le j \le i \le n}\mapsto (n+1-X_{i,i-j+1})_{1\le j \le i \le n}$, corresponding to the reflection and exchange of colors of bicolored pipe dreams. We say a catalan triangle is \defn{self-dual} if is fixed by this antiautomorphism. 

\cref{table:symmetricCT} gathers the number of all catalan triangules, of symmetric ones, and of self-dual ones.

\begin{table}[ht]
	\centering
	\begin{tabular}{c|cccccccccc|c}
		$n$ & $1$ & $2$ & $3$ & $4$ & $5$ & $6$ & $7$ & $8$ & $9$ & $\cdots$ & OEIS\\ \hline
		$|\Cat_n|$& $1$ & $2$ & $6$ & $28$ & $202$ & $2252$ & $38756$ & $1028964$ & $42127054$ & $\cdots$ & \href{https://oeis.org/A391969}{A391969} \\
		symmetric & $1$ & $2$ & $4$ & $12$ & $36$ & $168$ & $768$ & $5556$ & $38904$ &  $\cdots$ & \href{https://oeis.org/A390493}{A390493} \\
		self-dual & $1$ & $0$ & $2$ & $0$ & $10$ & $0$ & $120$ & $0$ &$3434$ & $\cdots$ & \href{https://oeis.org/A391973}{A391973}
	\end{tabular}
	\caption{Number of catalan triangles, symmetric catalan triangles and self-dual catalan triangles.}
	\label{table:symmetricCT}
\end{table}


\section{Maxima of congruence classes}
\label{sec:maxima}

This section is devoted to the maxima of the classes of catalan congruences of~$\ASM_n$.


\subsection{Depth antichains}

We first introduce the following objects, motivated by \cref{prop:catalanTriangleInequalities} and illustrated in \cref{fig:depthAntichains}.

\begin{definition}
\label{def:depthAntichain}
A \defn{depth antichain} is a map~$\varepsilon : A \to \N$ where $A$ is a (possibly empty) antichain of a triangular poset $\DPoset_n$, such that
\begin{itemize}
\item $-y_2 \le \varepsilon(y_1, y_2, y_3) \le y_2$ and~$\varepsilon(y_1, y_2, y_3) \equiv_{\mathrm{mod} \; 2} y_2$ for all~$(y_1, y_2, y_3) \in A$, and
\item $|\varepsilon(y_1, y_2, y_3) - \varepsilon(z_1, z_2, z_3)| \le |y_1 - z_1| + |y_3 - z_3|$ for all $(y_1, y_2, y_3), (z_1, z_2, z_3) \in A$.
\end{itemize}
See \cref{fig:depthAntichains}.
\begin{figure}[H]
	\setlength{\tabcolsep}{1.1pt} 
	\centerline{
	\begin{tabular}{ccccc}
		& & $\,\cdot\,$ & & \\
		& $\,\cdot\,$ & & $\,\cdot\,$ & \\
		$\,\cdot\,$ & & $\,\cdot\,$ & & $\,\cdot\,$
	\end{tabular}
	\qquad
	\begin{tabular}{ccccc}
		& & $\,\cdot\,$ & & \\
		& $\,\cdot\,$ & & $\,\cdot\,$ & \\
		$0$ & & $\,\cdot\,$ & & $\,\cdot\,$
	\end{tabular}
	\qquad
	\begin{tabular}{ccccc}
		& & $\,\cdot\,$ & & \\
		& $\,\cdot\,$ & & $\,\cdot\,$ & \\
		$\,\cdot\,$ & & $0$ & & $\,\cdot\,$
	\end{tabular}
	\qquad
	\begin{tabular}{ccccc}
		& & $\,\cdot\,$ & & \\
		& $\,\cdot\,$ & & $\,\cdot\,$ & \\
		$\,\cdot\,$ & & $\,\cdot\,$ & & $0$
	\end{tabular}
	\qquad
	\begin{tabular}{ccccc}
		& & $\,\cdot\,$ & & \\
		& $1$ & & $\,\cdot\,$ & \\
		$\,\cdot\,$ & & $\,\cdot\,$ & & $\,\cdot\,$
	\end{tabular}
	\qquad
	\begin{tabular}{ccccc}
		& & $\,\cdot\,$ & & \\
		& $\!\!-1\!\!$ & & $\,\cdot\,$ & \\
		$\,\cdot\,$ & & $\,\cdot\,$ & & $\,\cdot\,$
	\end{tabular}
	}
	\medskip
	\centerline{
	\begin{tabular}{ccccc}
		& & $\,\cdot\,$ & & \\
		& $\,\cdot\,$ & & $1$ & \\
		$\,\cdot\,$ & & $\,\cdot\,$ & & $\,\cdot\,$
	\end{tabular}
	\qquad
	\begin{tabular}{ccccc}
		& & $\,\cdot\,$ & & \\
		& $\,\cdot\,$ & & $\!\!-1\!\!$ & \\
		$\,\cdot\,$ & & $\,\cdot\,$ & & $\,\cdot\,$
	\end{tabular}
	\qquad
	\begin{tabular}{ccccc}
		& & $2$ & & \\
		& $\,\cdot\,$ & & $\,\cdot\,$ & \\
		$\,\cdot\,$ & & $\,\cdot\,$ & & $\,\cdot\,$
	\end{tabular}
	\qquad
	\begin{tabular}{ccccc}
		& & $0$ & & \\
		& $\,\cdot\,$ & & $\,\cdot\,$ & \\
		$\,\cdot\,$ & & $\,\cdot\,$ & & $\,\cdot\,$
	\end{tabular}
	\qquad
	\begin{tabular}{ccccc}
		& & $\!\!-2\!\!$ & & \\
		& $\,\cdot\,$ & & $\,\cdot\,$ & \\
		$\,\cdot\,$ & & $\,\cdot\,$ & & $\,\cdot\,$
	\end{tabular}
	\qquad
	\begin{tabular}{ccccc}
		& & $\,\cdot\,$ & & \\
		& $\,\cdot\,$ & & $\,\cdot\,$ & \\
		$0$ & & $0$ & & $\,\cdot\,$
	\end{tabular}
	}
	\medskip
	\centerline{
	\begin{tabular}{ccccc}
		& & $\,\cdot\,$ & & \\
		& $\,\cdot\,$ & & $1$ & \\
		$0$ & & $\,\cdot\,$ & & $\,\cdot\,$
	\end{tabular}
	\qquad
	\begin{tabular}{ccccc}
		& & $\,\cdot\,$ & & \\
		& $\,\cdot\,$ & & $\!\!-1\!\!$ & \\
		$0$ & & $\,\cdot\,$ & & $\,\cdot\,$
	\end{tabular}
	\qquad
	\begin{tabular}{ccccc}
		& & $\,\cdot\,$ & & \\
		& $\,\cdot\,$ & & $\,\cdot\,$ & \\
		$0$ & & $\,\cdot\,$ & & $0$
	\end{tabular}
	\qquad
	\begin{tabular}{ccccc}
		& & $\,\cdot\,$ & & \\
		& $\,\cdot\,$ & & $\,\cdot\,$ & \\
		$0$ & & $0$ & & $0$
	\end{tabular}
	\qquad
	\begin{tabular}{ccccc}
		& & $\,\cdot\,$ & & \\
		& $\,\cdot\,$ & & $\,\cdot\,$ & \\
		$\,\cdot\,$ & & $0$ & & $0$
	\end{tabular}
	\qquad
	\begin{tabular}{ccccc}
		& & $\,\cdot\,$ & & \\
		& $1$ & & $\,\cdot\,$ & \\
		$\,\cdot\,$ & & $\,\cdot\,$ & & $0$
	\end{tabular}
	}
	\medskip
	\centerline{
	\begin{tabular}{ccccc}
		& & $\,\cdot\,$ & & \\
		& $-1$ & & $\,\cdot\,$ & \\
		$\,\cdot\,$ & & $\,\cdot\,$ & & $0$
	\end{tabular}
	\qquad
	\begin{tabular}{ccccc}
		& & $\,\cdot\,$ & & \\
		& $1$ & & $1$ & \\
		$\,\cdot\,$ & & $\,\cdot\,$ & & $\,\cdot\,$
	\end{tabular}
	\qquad
	\begin{tabular}{ccccc}
		& & $\,\cdot\,$ & & \\
		& $1$ & & $\!\!-1\!\!$ & \\
		$\,\cdot\,$ & & $\,\cdot\,$ & & $\,\cdot\,$
	\end{tabular}
	\qquad
	\begin{tabular}{ccccc}
		& & $\,\cdot\,$ & & \\
		& $\!\!-1\!\!$ & & $1$ & \\
		$\,\cdot\,$ & & $\,\cdot\,$ & & $\,\cdot\,$
	\end{tabular}
	\qquad
	\begin{tabular}{ccccc}
		& & $\,\cdot\,$ & & \\
		& $\!\!-1\!\!$ & & $\!\!-1\!\!$ & \\
		$\,\cdot\,$ & & $\,\cdot\,$ & & $\,\cdot\,$
	\end{tabular}
	}
	\caption{The $23$ depth antichains of ~$\DPoset_4$.}
	\label{fig:depthAntichains}
\end{figure}
\end{definition}

It follows from \cref{prop:depthTriangleInequalities} that the restriction of any depth triangle to any antichain of $\DPoset_n$ is a depth antichain.
Conversely, we prove that any depth antichain can be extended to a depth triangle.

\begin{proposition}
\label{prop:extendDepthAntichain}
Any depth antichain $\varepsilon : A \to \N$ can be completed into a depth triangle, \ie there exists a depth triangle $\delta : \DPoset_n \to \N$ which restricts to~$\varepsilon$ on~$A$.
\end{proposition}

To prove \cref{prop:extendDepthAntichain}, we start with depth antichains of size one.

\begin{lemma}
\label{lem:extendDepthSingleton}
For any~$\b{z} \eqdef (z_1, z_2, z_3) \in \DPoset_n$ and~$-z_2 \le v \le z_2$ with~$v \equiv_{\mathrm{mod} \; 2} z_2$, the map ${\delta_{\b{z}, v} : \DPoset_n \to \N}$ defined by~$\delta_{\b{z}, v}(y_1, y_2, y_3) \eqdef \min(y_2, v + |y_1-z_1| + |y_3-z_3|)$ is a depth triangle with~$\delta_{\b{z}, v}(\b{z}) = v$.
\end{lemma}

\begin{proof}
Since~$y_1+y_2+y_3 = z_1+z_2+z_3$, we have
\[
v + |y_1-z_1| + |y_3-z_3| \ge v + y_1 - z_1 + y_3 - z_3 \ge -z_2 + y_1 - z_1 + y_3 - z_3 = -y_2.
\]
Therefore, we have~$-y_2 \le \delta_{\b{z}, v}(y_1, y_2, y_3) \le y_2$.
Hence, the bottom row is indeed~${0 0 \cdots 0}$.
Moreover, has~$v \equiv_{\mathrm{mod } 2} z_2$, we have~$y_2 \equiv_{\mathrm{mod} \; 2} v + |y_1-z_1| + |y_3-z_3|$, so that we get that
\[
|\delta_{\b{z}, v}(y_1, y_2, y_3) - \delta_{\b{z}, v}(y_1, y_2-1, y_3+1)| = 1 = |\delta_{\b{z}, v}(y_1, y_2, y_3) - \delta_{\b{z}, v}(y_1+1, y_2-1, y_3)|.
\qedhere
\]
\end{proof}

\begin{lemma}
\label{lem:extendDepthAntichain}
For any depth antichain~$\varepsilon : A \to \N$, the depth triangle~$\delta_\varepsilon : \DPoset_n \to \N$ obtained as the meet (\ie coordinate-wise minimum) of the depth triangles~$\delta_{\b{z}, \varepsilon(\b{z})}$ of \cref{lem:extendDepthSingleton} for~$\b{z} \in A$ restricts to~$\varepsilon$ on~$A$.
\end{lemma}

\begin{proof}
We have already seen in \cref{rem:latticeDepthTriangles} that the coordinate-wise minimum of depth triangles is a depth triangle.
Assume by means of contradiction that there is~$\b{y} \in A$ such that~$\delta_\varepsilon(\b{y}) \ne \varepsilon(\b{y})$.
By definition, we have~$\delta_{\b{y}, \varepsilon(\b{y})}(\b{y}) = \varepsilon(\b{y})$.
Hence, there is~$\b{z} \in A$ such that~$\delta_{\b{z}, \varepsilon(\b{z})}(\b{y}) < \varepsilon(\b{y}) \le y_2$.
We then get~$\varepsilon(\b{y}) - \varepsilon(\b{z}) > \delta_{\b{z}, \varepsilon(\b{z})}(\b{y}) - \varepsilon(\b{z}) = |y_1-z_1| + |y_3-z_3|$, contradicting the second condition of \cref{def:depthAntichain}.
\end{proof}

\begin{proof}[Proof of \cref{prop:extendDepthAntichain}]
The depth triangle~$\delta_\varepsilon$ of \cref{lem:extendDepthAntichain} restricts to~$\varepsilon$ on~$A$.
\end{proof}


\subsection{Covexillary permutations}

Recall that a permutation is vexillary if it avoids the pattern~$2143$.
Here, we are interested in the following variation.

\begin{definition}
A permutation is \defn{covexillary} if it avoids the pattern~$3412$.
\end{definition}

Recall that~$(i,j)$ is an ascent of an ASM~$A$ if there is another ASM~$A'$ such that the corresponding HFMs~$H,H'$ satisfy~$H'_{i,j} = H_{i,j} + 2$ and~$H'_{k,\ell} = H_{k,\ell}$ for all~${0 \le k, \ell \le n}$ with~$(i,j) \ne (k,\ell)$.
The following statements are standard.
\cref{lem:ascentsDescentsASMs} should be intuitively clear from the fact that a $-1$ in the ASM corresponds to a saddle point in the corresponding surface.

\begin{lemma}
\label{lem:covexillary}
A permutation is covexillary if and only if it does not have two ascents $(i,j)$ and $(k,\ell)$ such that $i<k$ and $j>\ell$. 
\end{lemma}

\begin{lemma}
\label{lem:ascentsDescentsASMs}
If $A_{u,v} = -1$, then 
\begin{itemize}
\item there are $i < u \le k$ and $j < v \le \ell$ such that $(i,j)$ and $(k,\ell)$ are descents of~$A$,
\item there are $i < u \le k$ and $j \ge v > \ell$ such that $(i,j)$ and $(k,\ell)$ are ascents of~$A$.
\end{itemize}
\end{lemma}

For~$A \in \ASM_n$, we denote by~$\mAnti(A)$ the antichain of~$\TPoset_n$ generating the upper set of~$\TPoset_n$ corresponding to~$A$, and by~$\bar\mAnti(A)$ the image of~$\mAnti(A)$ under the projection~$(x_1,x_2,x_3,x_4) \mapsto (x_1,x_2+x_3,x_4)$, and by~$\varepsilon_A : \bar\mAnti(A) \to \N$ the map defined by~$\varepsilon_A(x_1, x_2+x_3, x_4) = x_3-x_2$ for all~$(x_1,x_2,x_3,x_4) \in \mAnti(A)$.
Note that $\varepsilon_A$ may not be well-defined if the the projection~$(x_1,x_2,x_3,x_4) \mapsto (x_1,x_2+x_3,x_4)$ is not injective from~$\mAnti(A)$ to~$\bar\mAnti(A)$.
We are now ready to characterize the maxima of the classes of catalan congruences of~$\ASM_n$.

\begin{theorem}
\label{thm:covexillary}
The following assertions are equivalent for an ASM~$A$ of order~$n$:
\begin{enumerate}[(i)]
\item $A$ is a covexillary permutation matrix,
\item $A$ does not have two ascents $(i,j)$ and $(k,\ell)$ such that $i<k$ and $j>\ell$,
\item $\varepsilon_A$ is well-defined and forms a depth antichain of~$\DPoset_n$,
\item $\varepsilon_A$ is well-defined and extends to a depth triangle of order~$n$,
\item there exists a wall of order~$n$ containing~$\mAnti(A)$,
\item there exists a catalan congruence of~$\ASM_n$ such that~$A$ is maximal in its class.
\end{enumerate}
\end{theorem}

\pagebreak
\begin{proof}
The assertions~(i) and~(ii) are equivalent by \cref{lem:covexillary,lem:ascentsDescentsASMs}.
The assertions~(iii) and~(iv) are equivalent by \cref{prop:extendDepthAntichain}.
The assertions~(iv) and~(v) are equivalent by \cref{thm:depthTriangles}.
The assertions~(v) and~(vi) are equivalent by (the meet version of) \cref{thm:congruencesDistributiveLattices1} and the definition of walls in connection with \cref{thm:congruencesDistributiveLattices2}.
We thus just need to prove that the assertions~(ii) and~(iii) are equivalent.

It follows from the considerations in \cref{subsec:ASMs,subsec:depthTriangles} that, when~$\varepsilon_A$ is well-defined, the following assertions are equivalent for an ASM~$A$ and~$0 \le i , j \le n$:
\begin{itemize}
\item $A$ has an ascent at~$(i,j)$,
\item there is~$(x_1, x_2, x_3, x_4) \in \mAnti(A)$ with~$i = x_1+x_2$ and~$j = x_1+x_3$.
\item there is~$(y_1, y_2, y_3) \in \bar\mAnti(A)$ with~$2i = y_1 - y_3 + n - 2 - v$ and~$2j = y_1 - y_3 + n - 2 + v$ where~$v = \varepsilon_A(y_1, y_2, y_3)$.
\end{itemize}

Therefore, $A$ has two ascents $(i,j)$ and $(k,\ell)$ such that $i<k$ and $j>\ell$, if and only if there are~$\b{y}, \b{z} \in \bar\mAnti(A)$ with~$y_1 - y_3 - \varepsilon_A(\b{y}) < z_1 - z_3 - \varepsilon_A(\b{z})$ and~$y_1 - y_3 + \varepsilon_A(\b{y}) > z_1 - z_3 + \varepsilon_A(\b{z})$.
These two inequalities are equivalent to~$\varepsilon_A(\b{y})-\varepsilon_A(\b{z}) > |y_1-z_1+z_3-y_3| = |y_1-z_1| + |y_3-z_3|$, where the last equality holds since~$y_1 < z_1 \iff y_3 > z_3$ as otherwise~$\b{y}$ and~$\b{z}$ would be comparable in~$\DPoset$.
We conclude that the assertions~(ii) and~(iii) are equivalent.
\end{proof}

Note that conversely, each depth antichain~$\varepsilon : A \to \N$ of~$\DPoset_n$ defines an ASM~$A_\varepsilon$ of order~$[n]$ obtained as the meet of the meet-irreducible ASMs~$m(y_1, (y_2 - v)/2, (y_2 + v)/2, y_3)$ for each~$(y_1, y_2, y_3) \in A$ where~$v = \varepsilon(y_1,y_2,y_3)$.
We thus obtain the following statement.

\begin{corollary}
The following families are in bijection:
\begin{itemize}
\item the covexillary permutations of~$[n]$,
\item the depth antichains of~$\DPoset_n$,
\item the ASMs which appear as a maximum of a class in a catalan congruence of~$\ASM_n$.
\end{itemize}
\end{corollary}

\begin{example}
For instance, the $23$ depth antichains of \cref{fig:depthAntichains} correspond to the $23$ covexillary permutations of~$[4]$ (all except~$3412$).
\end{example}

\begin{example}
\label{exm:dominant}
Following on \cref{exm:minMaxCatalanTriangle}, consider the congruence corresponding to the minimal (resp.~maximum) catalan triangle.
Then the maxima of the congruence classes are precisely the permutations avoiding the pattern~$231$ (resp.~$312$).
Their mirror permutations are usually called \defn{dominant} (resp.~\defn{codominant}).
\end{example}

\begin{remark}
We conclude by a brief combinatorial description of meet representations of covexillary permutations as these representations are at the heart of this section.
Consider a permutation~$\sigma \in \Ss_n$.
Draw the permutation matrix~$P_\sigma$ (where~$(P_\sigma)_{i,j} = \delta_{\sigma_i,j}$ for all~${i,j \in [n]}$) and color in red all entries above or to the right of a~$1$, and in green the remaining entries.
The green entrees immediately below and to the left of red entries are called the \defn{essential points} of~$\sigma$.
The permutation~$\sigma$ is covexillary if and only if no essential point is above and to the right of another essential point (said differently, the green entries form a (broken) Ferrer's shape).
Moreover, the canonical meet representation of a covexillary permutation~$\sigma$ is given by~$\sigma = \bigwedge_{e} m(\b{x}^e)$ where $m(\b{x})$ denotes the anti-bigrassmannian permutation associated to~$\b{x} \in \TPoset_n$, the sum ranges over all essential points~$e$ of~$\sigma$, and~$\b{x}^e \eqdef (x_1^e, x_2^e, x_3^e, x_4^e)$ is the point where
\begin{itemize}
\item $x_1^e$ denotes the number of green entries below~$e$,
\item $x_2^e$ denotes the number of red entries upper right of ~$e$,
\item $x_3^e$ denotes the number of red entries to the left of (or equivalently below)~$e$,
\item $x_4^e$ denotes the number of green entries to the left of~$e$.
\end{itemize}
For instance, \cref{fig:covexillary} illustrates that
\[
3\mathrm{X}26479851 = m(6,0,2,0) \wedge m(5,1,1,1) \wedge m(3,2,1,2) \wedge m(2,1,2,3) \wedge m(1,1,2,4).
\]
\begin{figure}[t]
	\centerline{
		\includegraphics[scale=1.4]{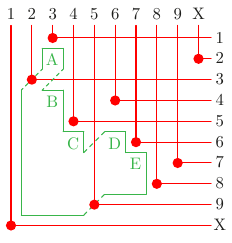}
		\qquad
		\raisebox{2.5cm}{\begin{tabular}{c|cccccc}
			e & A & B & C & D & E \\
			\hline
			$x_1^e$ & 6 & 5 & 3 & 2 & 1 \\
			$x_2^e$ & 0 & 1 & 2 & 1 & 1 \\
			$x_3^e$ & 2 & 1 & 1 & 2 & 2 \\
			$x_4^e$ & 0 & 1 & 2 & 3 & 4
		\end{tabular}}
	}
	\caption{The matrix of the covexillary permutation~$\sigma = 3\mathrm{X}26479851$ (left) and the parameters of the anti-bigrassmannian permutations in its canonical meet decomposition (right).}
	\label{fig:covexillary}
\end{figure}
As another illustration, note that the permutations avoiding the pattern~$231$ from \cref{exm:dominant} are precisely those for which the green entries really form a Ferrer's shape in the bottom-left corner of the matrix, and thus whose canonical meetands are of the form~$m(\b{x})$ with~$x_3 = 0$.
A similar characterization of course holds for the permutations avoiding~$312$ by symmetry through the main diagonal, and we obtain  canonical meetands are of the form~$m(\b{x})$ with~$x_2 = 0$.
\end{remark}


\section{Minimal permutations of congruence classes}
\label{sec:minimalPermutations}

This section is devoted to the minimal permutations in classes of catalan congruences of~$\ASM_n$.
In contrast to~\cref{sec:maxima}, the minimum of the congruence classes are not necessarily permutations (see \cref{sec:minimalNotPermutations}).
However, it turns out that each congruence class contains a minimal permutation.


\subsection{Nappes}

Recall that for an ASM~$A$, we have denoted by~$\JIdeal(A)$ the corresponding lower set of the tetrahedron poset~$\TPoset_n$.

\begin{definition}
The \defn{nappe} of an ASM~$A$ of order~$n$ is the upper enveloppe of~$\JIdeal(A)$, that is, the set~$\Nappe(A)$ containing for each $1\le i,j\le n$ the biggest $(x_1, x_2, x_3, x_4)\in \JIdeal(A)$ such that $i = x_1+x_2$ and~$j = x_1+x_3$ (if it exists).
\end{definition}

The following statement is the counterpart of \cref{lem:wall}.
Visually, it states that a nappe cannot enter the interior of the black hourglass illustrated in \cref{fig:butterflyHourglass}\,(middle) at each of its points.

\begin{lemma}
\label{lem:nappe}
If $(x_1,x_2,x_3,x_4)$ and $(x'_1,x'_2,x'_3,x'_4)$ are in a nappe, then $x_4-x'_4$, $x_1-x'_1$, $x'_3-x_3$ and $x'_2-x_2$ cannot be all strictly positive or all strictly negative.
\end{lemma}

\begin{proof}
Similar to that of \cref{lem:wall}.
\end{proof}

The following statement asserts that if two permutations $\tau\le \sigma$ differ by a transposition, the lower set~$\JIdeal(P_\tau)$ can be obtained from the lower set~$\JIdeal(P_\sigma)$ by removing a rectangular portion of its nappe.

\begin{proposition}
\label{prop:nappe}
Let $\sigma,\tau \in \Ss_n$ such that $\sigma = \tau \circ (i,j)$ and $\tau < \sigma$.
Then
\[
\JIdeal(P_\sigma) \ssm \JIdeal(P_\tau) = \set{(x_1, x_2, x_3, x_4)\in \Nappe(P_\sigma)}{i \le x_1 + x_2 < j \text{ and } \sigma(j) \le x_1 + x_3 < \sigma(i)}.
\]
We denote this set by~$\Nappe_{\sigma,i,j}$.
\end{proposition}

\begin{proof}
We have $\sigma = \tau \circ (i,j)$ with $\tau< \sigma$ if and only if the corresponding corner sum matrices~$C_\sigma$ and~$C_\tau$ satisfy~$(C_\sigma)_{k,\ell} = (C_\tau)_{k,\ell}-1$ if~$i \le k < j$ and $\tau(i) \le \ell < \tau(j)$ and~$(C_\sigma)_{k,\ell} = (C_\tau)_{k,\ell}$ otherwise. It follows that $\JIdeal(P_\tau)$ can be obtained from $\JIdeal(P_\sigma)$ by removing only one join irreducible~$(x_1, x_2, x_3, x_4)$ for each value $(x_1+x_2, x_1+x_3)$, which must be the biggest.
\end{proof}

\begin{definition}
Let $\sigma\in \Ss_n$ and $i<j$ such that $\sigma\circ (i,j)<\sigma$. We define $P_{\sigma,i,j}$ as the matrix obtained from $P_\sigma$ by summing rows $1$ to $i-1$, rows $j+1$ to $n$, columns $1$ to $\sigma(j)-1$ and columns $\sigma(i)+1$ to $n$.
See \cref{fig:nappe}.
\begin{figure}[ht]
	\centering
	\includegraphics[scale=.7]{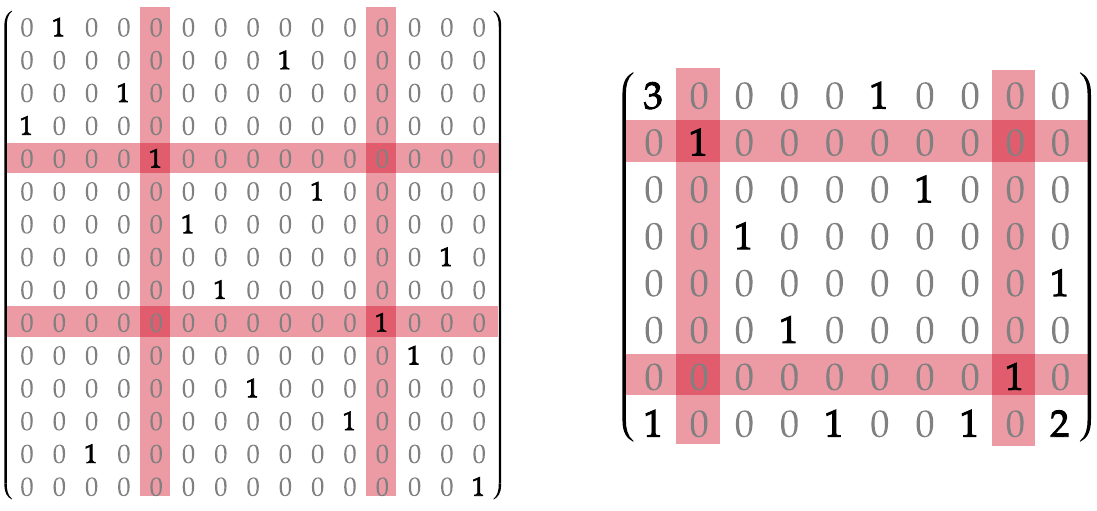}
	\caption{\centering Matrices $P_\sigma$ and $P_{\sigma,i,j}$ with $\sigma = 2~9~4~1~5~10~6~14~7~12~13~8~11~3~15$, $i=5$ and $j=10$.}
	\label{fig:nappe}
\end{figure}
\end{definition}

\begin{remark}
When seeing~$\Ss_n$ as a Coxeter group $W$ generated by the transpositions $s_1=(1,2), ..., s_{n-1}=(n-1,n)$, the matrices $P_{\sigma,i,j}$ are the contingency matrices corresponding to double cosets in the parabolic quotient~$W_I\backslash W/ W_J$ where~$I = \{s_1,...,s_{i-2}, s_{j+1},...,s_{n-1}\}$ and $J = \{s_1,...,s_{\sigma(j)-2},s_{\sigma(i)+1},...,s_{n-1}\}$.
\end{remark}

We now show that the matrices~$P_{\sigma,i,j}$ where $\sigma\circ (i,j)<\sigma$ are in bijection with the sets of join-irreducibles that are removed from some ideal $\JIdeal(P_{\sigma})$ when applying a transposition~$(i,j)$ to $P_\sigma$.

\begin{proposition}
\label{prop:bijectionPsij}
Let $\sigma_1 \in \Ss_n$ and $i_1<j_1$ such that $\sigma_1\circ(i_1,j_1) < \sigma_1$, and~$\sigma_2 \in \Ss_n$ and~$i_2<j_2$ such that $\sigma_2\circ(i_2,j_2) <\sigma_2$.
Then $\Nappe_{\sigma_1,i_1,j_1} = \Nappe_{\sigma_2,i_2,j_2}$ if and only if $i_1 = i_2$, $j_1 = j_2$ and~$P_{\sigma_1,i_1,j_1}=P_{\sigma_2,i_2,j_2}$. 
\end{proposition}

\begin{proof}
By \cref{prop:nappe}, $\Nappe_{\sigma,i,j}$ is the set of join-irreducible elements in $\JIdeal(P_\sigma) \ssm \JIdeal(P_\tau)$ where $\sigma=\tau\circ (i,j)$.
The elements of the nappe of an ASM correspond to the values and positions of the entries of its height function.
Hence $\Nappe_{\sigma_1,i_1,j_1} = \Nappe_{\sigma_2,i_2,j_2}$ if and only if the height functions of $P_{\sigma_1}$ and $P_{\sigma_2}$ have the same values in the rectangle between rows $i_1=i_2$ to $j_1=j_2$ and columns $\sigma(j_1)$ to $\sigma(i_1)$, which is equivalent to $P_{\sigma_1,i_1,j_1}=P_{\sigma_2,i_2,j_2}$.
\end{proof}

Finally, we exploit \cref{prop:bijectionPsij} to count the distinct sets of the form~$N_{\sigma,i,j}$, and those corresponding to cover relations in the Bruhat order.
The first few values of these numbers are gathered in \cref{table:NnOn}.

\begin{table}[h]
	\centering
	\begin{tabular}{l|ccccccccccc|c}
		$n$ & $1$ & $2$ & $3$ & $4$ & $5$ & $6$ & $7$ & $8$ & $9$ & $10$ & $\cdots$ & OEIS \\ \hline
		$N_n$ & $0$ & $1$ & $9$ & $52$ & $260$ & $1291$ & $6915$ & $41814$ & $289758$ & $2291381$ & $\cdots$ & \href{https://oeis.org/A391692}{A391692} \\
		$O_n$ & $0$ & $1$ & $8$ & $38$ & $140$ & $443$ & $1268$ & $3384$ & $8584$ & $20965$ & $\cdots$ & \href{https://oeis.org/A034009}{A034009}
	\end{tabular}
	\caption{First values of~$N_n$ and~$O_n$ from~\cref{prop:NnOn}.}
	\label{table:NnOn}
\end{table}

\begin{proposition}
\label{prop:NnOn}
The ordinary generating function of the number~$N_n$ of distinct sets of the form~$N_{\sigma,i,j}$ with $\sigma \in \Ss_n$ is given by
\[
\sum_{n\ge 1} N_n x^n = \frac{x^2}{(1-x)^4}\sum_{a,b,c,d,e\ge 0}\frac{(a+b+c)!~(a+d+e)!}{a!~b!~c!~d!~e!}x^{a+b+c+d+e}.
\]
The ordinary generating function of the number~$O_n$ of these sets which correspond to cover relations in the Bruhat order is given by
\[
\sum_{n\ge 1} O_n x^n = \frac{x^2}{(1-x)^4(1-2x)^2}.
\]
\end{proposition}

\begin{proof}
Matrix of the form $P_{\sigma,i,j}$ are made of a permutation matrix shuffled with entries equal to $1$ on its four sides, the two entries exchanged by the transposition $(i,j)$, and four integers in the corners (see Figure \ref{fig:nappe} for an illustration). These matrices can be enumerated by distinguishing them according to the size of the permutation matrix in the middle and the number of $1$ entries on each side. The matrices corresponding to cover relations are those with an empty permutation matrix in the middle.
\end{proof}


\subsection{Minimal permutations}

We are now ready to describe the minimal permutations in the classes of a catalan congruence of~$\ASM_n$.
The following statement follows from \cref{lem:wall,lem:nappe}.
Visually, it states that the intersection of a wall and a nappe must remain in the complement of the butterfly and of the hourglass at each of its points, which is illustrated in green in \cref{fig:butterflyHourglass}\,(right).

\begin{lemma}
\label{lem:wallnappe} 
If $\Nappe$ is the nappe of an ASM, and $\P$ is a wall of a congruence, and if~$(x_1,x_2,x_3,x_4)$ and~$(x'_1,x'_2,x'_3,x'_4)$ are both in ${\Nappe} \cap {\P}$, then $x_1+x_2-x'_1-x'_2$, $x_1+x_3-x'_1-x'_3$, $x'_4-x_4$ and $x_1-x'_1$ are all positive or all negative.
\end{lemma}

\begin{figure}[ht]
	\centerline{\includegraphics[scale=1]{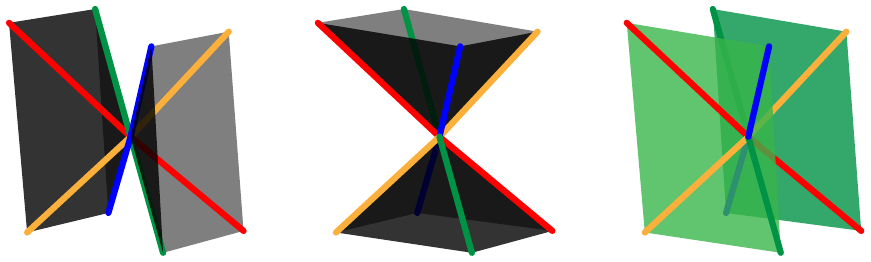}}
	\caption{The butterfly (left), the hourglass (middle), and the complement of their union (right).}
	\label{fig:butterflyHourglass}
\end{figure}

\begin{proposition}
\label{prop:minimalPermutation}
Let~$\equiv$ be a catalan congruence of~$\ASM_n$.
Let $\sigma\in \Ss_n$ containing the pattern $321$, and let~$i<j<k$ be such that~${\sigma(i) > \sigma(j) > \sigma(k)}$.
Then at least one of the permutations $\sigma \circ (i,j)$ and $\sigma \circ (j,k)$ is congruent to $\sigma$.
Moreover, if both are, then $\sigma\circ (i,k)$ is also congruent to $\sigma$.
\end{proposition}

\begin{proof}
We show that at least one of the permutations $\sigma \circ (i,j)$ and $\sigma \circ (j,k)$ is congruent to $\sigma$.
Let $\P$ be the wall inducing $\equiv$, \ie the corresponding subposet of $\TPoset_n$ isomorphic to $\DPoset_n$. Suppose $\sigma \circ (i,j)$ is not congruent to $\sigma$. By Proposition \ref{prop:nappe}, this is equivalent to $\P$ intersecting $\Nappe_{\sigma,i,j}$.
By \cref{lem:wallnappe}, if $(x_1,x_2,x_3,x_4)$ and $(y_1,y_2,y_3,y_4)$ are in $\Nappe_{\sigma,i,j}\cap\P$, then $x_1+x_2-y_1-y_2$, $x_1+x_3-y_1-y_3$ are both positive or both negative. This implies that if $(x_1,x_2,x_3,x_4)\in \Nappe_{\sigma,i,j}$, $\P$ cannot intersect $\Nappe_{\sigma,j,k}$ (since $(y_1,y_2,y_3,y_4)$ must be in the shaded area of \cref{fig:minfigure}\,(a)), hence $\sigma \circ (j,k)$ is congruent to $\sigma$.

We now prove that if both permutations $\sigma \circ (i,j)$ and $\sigma \circ (j,k)$ are congruent to $\sigma$, then $\sigma\circ (i,k)$ is also congruent to $\sigma$. For this, suppose that $(x_1,x_2,x_3,x_4)\in \Nappe_{\sigma,i,k}\cap\P$. We can assume without loss of generality that $(x_1,x_2,x_3,x_4)$ is in the bottom left corner, \emph{i.e.} $x_1+x_3+1\ge j$ and $x_1+x_2+1<\sigma(j)$. It is possible to find elements of $\Nappe_{\sigma,i,k}\cap\P$ closer to $\Nappe_{\sigma,i,j}$ or $\Nappe_{\sigma,j,k}$, until we eventually reach it: if $(x_1,x_2,x_3+1,x_4-1)$, $(x_1,x_2+1,x_3,x_4-1)$, $(x_1+1,x_2,x_3-1,x_4)$ and $(x_1+1,x_2-1,x_3,x_4)$ are not in $\Nappe_{\sigma,i,k}\cap\P$, then $(x_1+1,x_2,x_3,x_4-1)$ is in $\Nappe_{\sigma,i,k}\cap\P$, and this can only be the case if the two elements of the nappe to the right and above $(x_1,x_2,x_3,x_4)$ have different heights (value of $x_2+x_3$). Hence we cannot go between $\Nappe_{\sigma,i,j}$ and $\Nappe_{\sigma,j,k}$ without intersecting one of them, since there is a entry equal to $1$ there which corresponds to a saddle point in the nappe (cf. \cref{fig:minfigure}\,(b)).
\end{proof}

\begin{figure}[ht]
	\centerline{
			\begin{tabular}{cc}
			\includegraphics[width=0.4\linewidth]{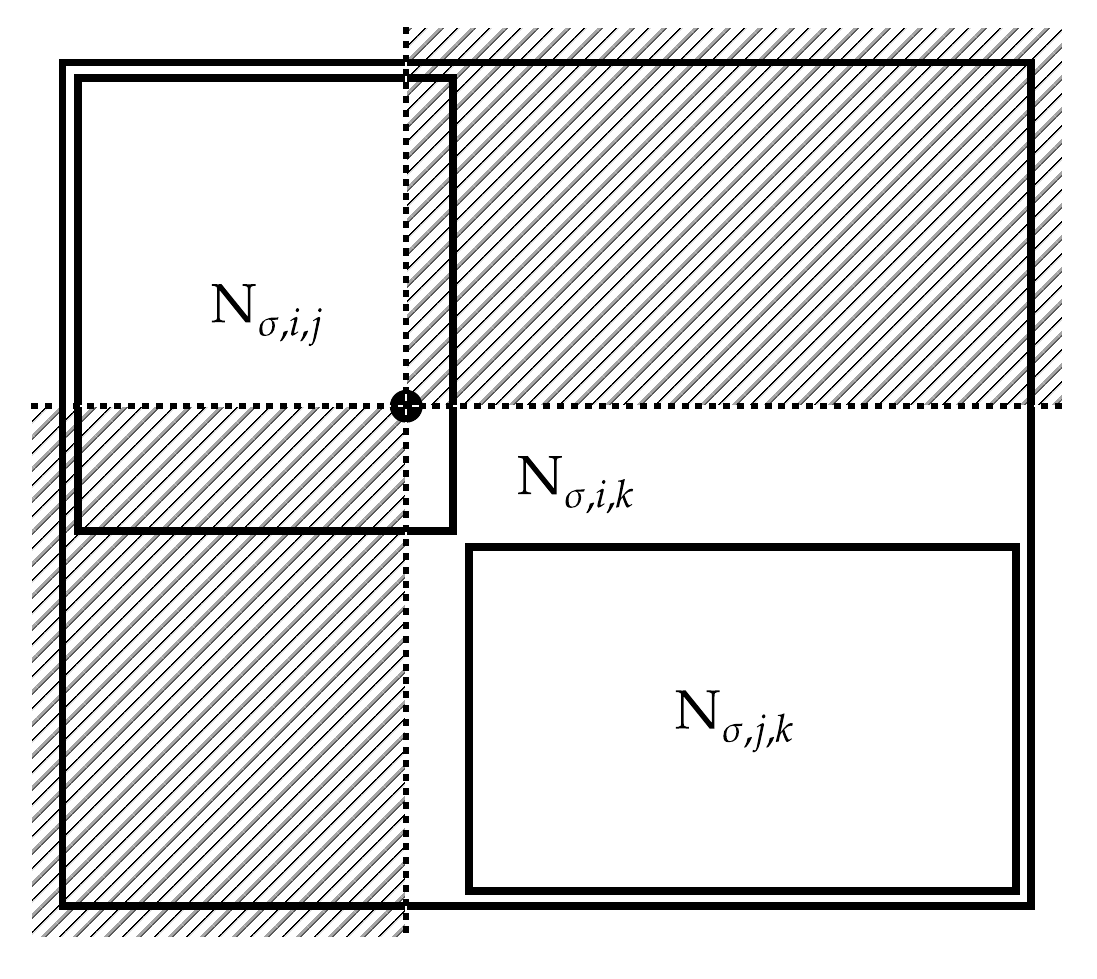} &
			\includegraphics[width=0.4\linewidth]{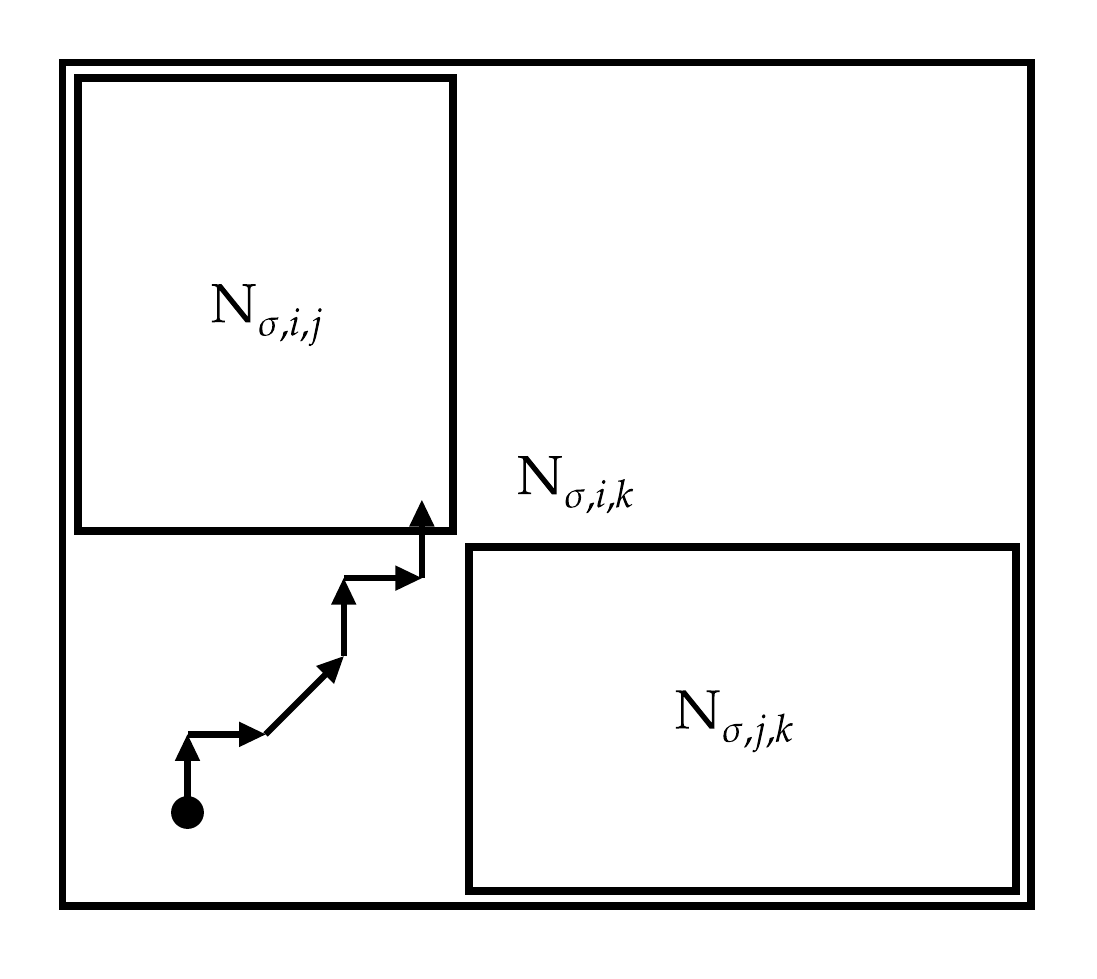} \\
			(a) & (b) 
		\end{tabular}
  }
    \caption{The nappes $\Nappe_{\sigma,i,j}$, $\Nappe_{\sigma,j,k}$ and $\Nappe_{\sigma,i,k}$ of \cref{prop:minimalPermutation}.}
	\label{fig:minfigure}
\end{figure}

\begin{theorem}
\label{thm:minimalPermutation}
For a catalan congruence of~$\ASM_n$,
\begin{itemize}
\item each congruence class contains a unique minimal permutation,
\item the minimal permutations of the congruence classes are precisely the $321$-avoiding permutations.
\end{itemize}
\end{theorem}

\begin{proof}
By \cref{prop:minimalPermutation}, a permutation containing the pattern $321$ cannot be minimal in its congruence class. By \cref{thm:covexillary}, each class contains a permutation, hence at least one minimal permutation, which must avoid~$321$ by \cref{prop:minimalPermutation}.
The result follows since the number of $321$-avoiding permutations is the same as the number of elements of~$\Dyck_n$.
\end{proof}


\subsection{Minimal elements versus minimal permutations}
\label{sec:minimalNotPermutations}

We say that an ASM is \defn{permutational} if it is the matrix~$P_\sigma$ of a permutation~$\sigma$, and \defn{non-permutational} otherwise.
As already mentioned, the minimal (and maximal) ASMs of the classes of a congruence of~$\ASM_n$ are not necessarily permutational.
Surprisingly, we have seen in \cref{thm:covexillary} that the maxima are permutational for the specific congruences of~$\ASM_n$ considered in this paper.
We will now observe that the minima for these specific congruences are not all permutational as soon as~$n \ge 3$ (as illustrated in \cref{fig:ASMQuotients}).

Recall that for an ASM~$A$, we denote by~$\JIdeal(A)$ the corresponding lower set of~$\TPoset_n$, and by~$\jAnti(A)$ the antichain of~$\TPoset_n$ generating~$\JIdeal(A)$.
Recall also from \cref{thm:congruencesDistributiveLattices1} that for a congruence~$\equiv$ of~$\ASM_n$ with non-contracted join irreducibles~$\JIrr(\equiv) \subseteq \TPoset_n$, the minima of the congruence classes are precisely the ASMs~$A$ such that~$\jAnti(A) \subseteq \JIrr(\equiv)$.
This directly yields the following statement.

\pagebreak
\begin{corollary}
\label{coro:minimaPermutahedralLong}
The following conditions are equivalent for a lattice congruence~$\equiv$ of~$\ASM_n$:
\begin{itemize}
\item the minima of the congruence classes of~$\equiv$ are all permutational,
\item $\jAnti(A) \not\subseteq \JIrr(\equiv)$ for each non-permutational ASM~$A$,
\item $\JIrr(\ASM_n) \ssm \JIrr(\equiv)$ is a transversal of the set~$\set{\jAnti(A)}{A \text{ non-permutational ASM}}$.
\end{itemize}
\end{corollary}

This immediately yields the following.

\begin{corollary}
For~$n \ge 3$ and any catalan congruence~$\equiv$ of~$\ASM_n$, there is at least one congruence class of~$\equiv$ whose minima is non-permutational.
\end{corollary}

\begin{proof}
As~$\equiv$ is a catalan congruence, both~$j(0,0,0,n-2)$ and~$j(1,0,0,n-3)$ belong to~$\JIrr(\equiv)$ by \cref{prop:walls}, and~$j(0,0,0,n-2) \vee j(1,0,0,n-3)$ has a~$-1$ in position~$(2,2)$.
\end{proof}

We now show that the condition of \cref{coro:minimaPermutahedralLong} can be largely simplified using the following lemma.

\begin{lemma}
\label{prop:subCanonicalJoinRepresentation}
For any non-permutational ASM~$A$, there exists a non-permutational ASM~$A'$ such that $|\jAnti(A')| = 2$ and $\jAnti(A') \subseteq \jAnti(A)$.
\end{lemma}

\begin{proof}
We have seen in~\cref{subsec:ASMs} that the~$1$ and $-1$ entries in an ASM corresponds to the saddle points in the corresponding surface.
More precisely, for any~$x_1, x_2, x_3, x_4 \in \N$, the ASM~$A$ has a $-1$ in position~$(x_1+x_2+1, x_1+x_3+1)$ if and only if~$\JIdeal(A)$ contains both~$j_1 \eqdef j(x_1+1, x_2, x_3, x_4)$ and~$j_4 \eqdef j(x_1, x_2, x_3, x_4+1)$ but neither~$j_2 \eqdef j(x_1, x_2+1, x_3, x_4)$ nor~$j_3 \eqdef j(x_1, x_2, x_3+1, x_4)$.
The latter happens if and only if~$\jAnti(A)$ contains two join irreducibles~$J, J'$ such that~$j_1 \le J$ and~$j_4 \le J'$ but neither~$J$ nor~$J'$ is larger than~$j_2$ nor~$j_3$.
We conclude that~$A' = J \vee J'$ also contains a~$-1$ at the same position as~$A$.
\end{proof}

\begin{corollary}
\label{coro:minimaPermutahedral}
The following conditions are equivalent for a lattice congruence~$\equiv$ of~$\ASM_n$:
\begin{itemize}
\item the minima of the congruence classes of~$\equiv$ are all permutational,
\item $\jAnti(A) \not\subseteq \JIrr(\equiv)$ for each non-permutational ASM~$A$ with~$|\jAnti(A)| = 2$,
\item $\JIrr(\ASM_n) \ssm \JIrr(\equiv)$ is a vertex cover of the graph on $\JIrr(\ASM_n)$ with an edge~$\{J,J'\}$ if and only if~$J \vee J'$ is non-permutational.
\end{itemize}
\end{corollary}

Of course, there is a dual statement for all maxima to be permutational.
For completeness, we note that the same arguments enable to prove the following version, considering simultaneously minima and maxima.

\begin{corollary}
\label{coro:minimaMaximaPermutahedral}
The following conditions are equivalent for a lattice congruence~$\equiv$ of~$\ASM_n$:
\begin{itemize}
\item the minima and maxima of the congruence classes of~$\equiv$ are all permutational,
\item $\JIrr(\ASM_n) \ssm \JIrr(\equiv)$ is a vertex cover of the graph on $\JIrr(\ASM_n)$ with an edge~$\{J,J'\}$ if and only if~$J \vee J'$ or~$\kappa(J) \wedge \kappa(J')$ is non-permutational.
\end{itemize}
\end{corollary}

\cref{coro:minimaPermutahedral,coro:minimaMaximaPermutahedral} turns the enumeration of lattice congruences of $\ASM_n$ where minima and/or maxima are required to be permutational into an enumeration of vertex covers.
\cref{table:numberCongruences} gathers the first few values.

\begin{table}
	\centerline{
	\begin{tabular}{c|ccccccc|c}
		$n$ & $2$ & $3$ & $4$ & $5$ & $6$ & $7$ & $\cdots$ & OEIS \\ \hline
	 	all congruences & $2$ & $16$ & $1024$ & $1048576$ & $34359738368$ & $72057594037927936$ & $\cdots$ & \href{https://oeis.org/A125791}{A125791} \\ 
		min permutational & $2$ & $12$ & $216$ & $10480$ & $1344096$ & $465473984$ & $\cdots$ & \href{https://oeis.org/A391690}{A391690} \\
		min \& max permutat. & $2$ & $9$ & $69$ & $716$ & $8986$ & $128065$ & $\cdots$ & \href{https://oeis.org/A391691}{A391691}
	\end{tabular}
	}
	\caption{The numbers of congruences of~$\ASM_n$ depending on whether we require all minima and/or maxima to be permutational.}
	\label{table:numberCongruences}
\end{table}


\pagebreak
\section{Bases of the Temperley--Lieb algebra}
\label{sec:TL}

In this section, we construct various bases of the Temperley--Lieb algebra~$\TL_n(2)$ using the quotients studied in \cref{sec:quotientsASMs,sec:maxima,sec:minimalPermutations}, following the work of~\cite{BergeronGagnon}.


\subsection{The Temperley--Lieb algebra}

Recall that, for~$q \in \C$, the \defn{Temperley--Lieb algebra}~$\TL_n(q)$ is the $\C$-algebra generated by~$e_1, \dots, e_{n-1}$ subject to the \defn{Jones relations}
\begin{itemize}
\item $e_i^2 = q e_i$ for~$i \in [n-1]$,
\item $e_i e_{i+1} e_i = e_i$ for~$i \in [n-2]$,
\item $e_{i+1} e_i e_{i+1} = e_{i+1}$ for~$i \in [n-2]$,
\item $e_i e_j = e_j e_i$ for~$i,j \in [n-1]$ with~$|i-j| \ne 1$.
\end{itemize}
Although we will not use it in this paper, let us recall that there is also a classical diagrammatic description of~$\TL_n(q)$ as the vector space generated by noncrossing matchings of two vertical lines of~$n$ points, where the multiplication of two diagrams is obtained by concatenating them and replacing each closed loop by a factor~$q$, and where the generator~$e_i$ corresponds to the diagram connecting the $i$th and~$(i+1)$th points of each line, and the $j$th points of both lines for all~$j \notin \{i,i+1\}$.
See \cref{fig:TL}.

\begin{figure}[ht]
	\centering
	\includegraphics[scale=.5]{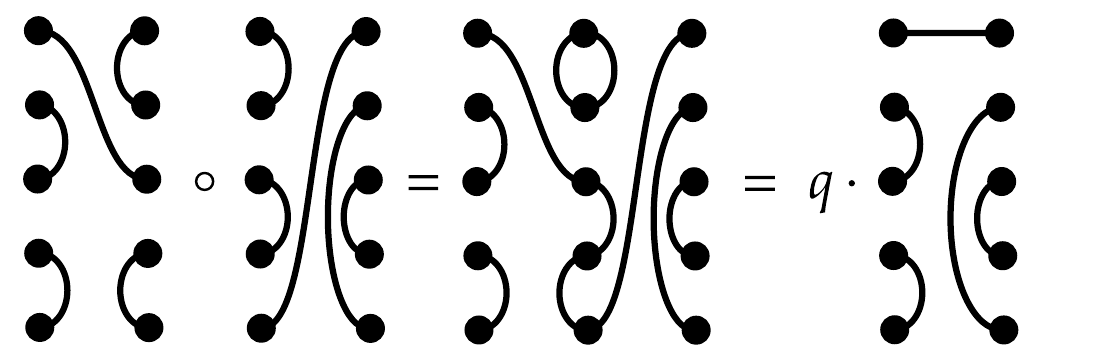}
	\caption{The classical diagrammatic description of the Temperley--Lieb algebra.}
	\label{fig:TL}
\end{figure}


\subsection{Bases}

In this section, we produce various bases of the Temperley--Lieb algebra~$\TL_n(2)$ following~\cite{BergeronGagnon}.
It is well known that~
\[
\TL_n(2) \simeq \C\Ss_n / \left\langle () - (ij) - (jk) - (ik) + (ijk) + (ikj) \mid 1 \leq i < j < k \leq n \right\rangle,
\]
where permutations are written in cycle notation. Hence any $321$ pattern can be rewritten as $321 \to 123 - 213 - 132 + 231 + 312$.
As a consequence, the set~$T_n$ of $321$-avoiding permutations of~$[n]$ forms a basis of $\TL_n(2)$.
This was extended in~\cite{BergeronGagnon} to the following.

\begin{theorem}[\cite{BergeronGagnon}]
Any choice of representative permutations of the excedance classes on~$\Ss_n$ forms a basis of the Temperley--Lieb algebra~$\TL_n(2)$.
\end{theorem}

We extend this result to all catalan congruences of~$\ASM_n$.

\begin{theorem}
Any choice of representative permutations in the classes of any catalan congruence~$\equiv$ of $\ASM_n$ forms a basis of~$\TL_n(2)$.
\end{theorem}

\begin{proof}
Using the relation $(ik) = () - (ij) - (jk) + (ijk) + (ikj)$, we obtain that any permutation~$\sigma \in \Ss_n$ can be written as a linear combination of permutations avoiding $321$ that are below $\sigma$ in the Bruhat order
\[
\sigma = \sum_{\tau\in T_n,\tau\le \sigma} c_{\tau,\sigma} \ \tau.
\]
By \cref{thm:minimalPermutation}, the congruence class of $\sigma$ contains a unique minimal permutation~$\tau$, which moreover avoids the pattern $321$.
We prove by induction that $c_{\tau,\sigma}=1$.
This is obvious if~$\sigma = \tau$, \ie if~$\sigma$ avoids the pattern~$321$.
Assume now that~$\tau < \sigma$ and suppose by induction that~$c_{\tau,\sigma'} = 1$ for any permutations~$\sigma'$ such that~$\tau \le \sigma' < \sigma$. 
Note that for~$\sigma' < \sigma$, we have~$\tau \le \sigma'$ if and only if~$\sigma' \equiv \sigma$.
Since~$\sigma \ne \tau$, there is~$i < j < k$ such that~$\sigma(i) > \sigma(j) > \sigma(k)$, so that~$\sigma = \sigma\circ (i,j)+\sigma\circ (j,k)-\sigma\circ (i,j,k)-\sigma\circ (i,k,j)+\sigma\circ (i,k)$.
\cref{prop:minimalPermutation} states that at least one of the permutations $\sigma\circ (i,j)$ and $\sigma\circ (j,k)$ is congruent to $\sigma$.
If only $\sigma\circ (i,j)$ is congruent to $\sigma$, then we get~$c_{\tau,\sigma}=c_{\tau,\sigma \circ (i,j)}=1$ since all the other permutations in the relation are not congruent to~$\sigma$, hence not larger than~$\tau$.
If both~$\sigma\circ (i,j)$ and~$\sigma\circ (j,k)$ are congruent to $\sigma$, then $\sigma\circ (i,j,k)$,  $\sigma\circ (i,k,j)$, and~$\sigma\circ (i,k)$ are also congruent to them and we get~$c_{\tau,\sigma}=c_{\tau,\sigma \circ (i,j)}+c_{\tau,\sigma \circ (j,k)}-c_{\tau,\sigma \circ (i,j,k)}-c_{\tau,\sigma \circ (i,k,j)}+c_{\tau,\sigma \circ (i,k)}=1$.
By triangularity, any set of representatives of the congruence classes of~$\equiv$ can be expressed as linearly independent combinations of the $321$-avoiding permutations, hence forms a basis of $\TL_n(2)$.
\end{proof}


\section{Posets $\P^n/\Ss_n$}
\label{sec:Pn/Sn}

To conclude this paper, we observe that the join irreducible posets of all the lattices considered along the paper (and many others) are all instances of a general symmetrization operation on posets, which will certainly deserve further study.

Let $\P$ be a finite poset and $n\ge 0$.
The set~$\P^n$ of $n$-tuples of elements of $\P$ is ordered by coordinatewise comparison.
For all $\b{x} = (x_1,...,x_n) \in \P^n$ and~$\sigma\in \Ss_n$, we write~$\b{x}^\sigma \eqdef (x_{\sigma(1)},...,x_{\sigma(n)})$.
We consider the set $\P^n/\Ss_n$ of tuples of $\P^n$ up to permutation of their coordinates. 
We write~$\bar{\b{x}} \le \bar{\b{y}}$ if there exists~$\b{x} \in \bar{\b{x}}$ and~$\b{y} \in \bar{\b{y}}$ such that~$\b{x} \le \b{y}$.

\begin{proposition}
\label{prop:symmetrizationPoset}
The relation~$\le$ is a partial order on~$\P^n/\Ss_n$.
\end{proposition}

\begin{proof}
The relation~$\equiv$ is
\begin{itemize}
\item reflexive since~$\b{x} \le \b{x}$ for any~$\b{x} \in \bar{\b{x}}$,
\item transitive since~$\b{x} \le \b{y}^\sigma$ and~$\b{y} \le \b{z}^\tau$ implies~$\b{x} \le \b{z}^{\tau\sigma}$,
\item antisymmetric since~$\b{x} \le \b{x}^\sigma$ implies that~$\b{x}$ is constant on every cycle of~$\sigma$ (since ${x_{i_1} \le x_{\sigma(i_1)} = x_{i_2} \le \dots \le x_{i_k} \le x_{\sigma(i_k)} = x_{i_1}}$ for any cycle~$(i_1,...,i_k)$ of~$\sigma$), hence \linebreak that~$\b{x} = \b{x}^\sigma$.
\qedhere
\end{itemize}
\end{proof}

\pagebreak
\begin{example}
For the $2$-element chain~$\Cc_2$, the poset~$\Cc_2^n$ is the boolean lattice, and the quotient~$\Cc_2^n / \Ss_n$ is a $(n+1)$-element chain.
For the $k$-element chain~$\Cc_k$, the poset~$\Cc_k^n$ is a poset on the integer points in the $(k-1)$th dilate of the cube, and~$\Cc_k^n / \Ss_n$ is a poset on the integer points in one of the fundamental simplices of this cube.
\end{example}

\begin{example}
As illustrated in \cref{fig:smallposets}, the join irreducible posets of the lattices $\ASM_n$, $\Dyck_n$ and $\Cat_n$ considered in this paper are all instances of \cref{prop:symmetrizationPoset}.
Other examples include the join irreducible posets of the lattices of totally symmetric partitions, Magog triangles, gapless triangles, Gelfand-Tsetlin triangles with top line $12 \cdots n$, \emph{etc.}
\begin{figure}[ht]
	\centering
	\includegraphics[scale=.8]{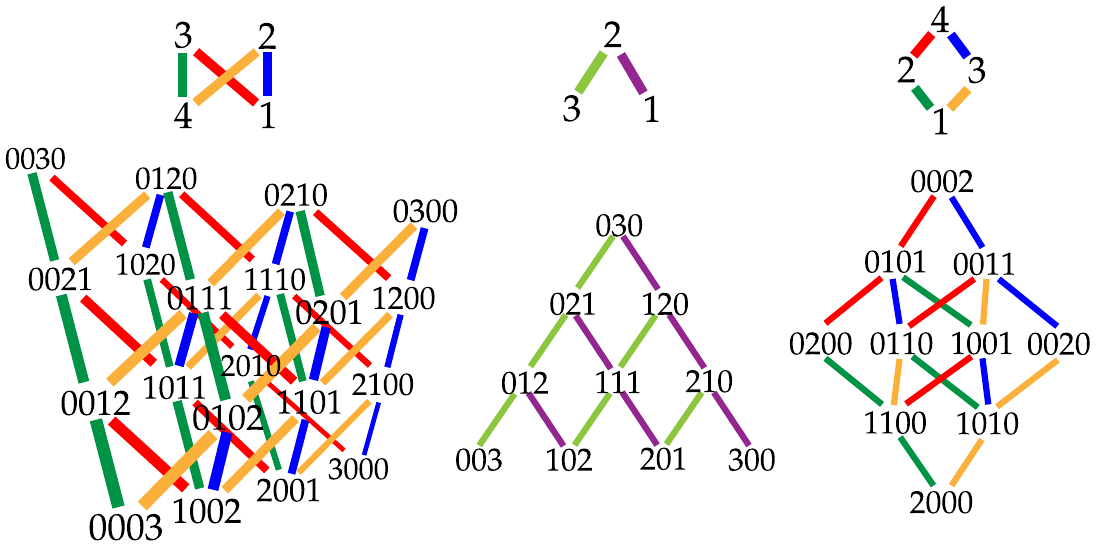}
	\caption{The posets $\P$ (top) such that $\P^{n-2}/\Ss_{n-2}$ (bottom) is isomorphic to the join irreducible posets of~$\Dyck_n$ (left), $\ASM_n$ (middle) and $\Cat_n$ (right). Here,~$n = 3$ for the first two columns, and $n = 2$ for the last one.}
	\label{fig:smallposets}
\end{figure}
\end{example}

An element $\bar{\b{x}} \in \P^n/\Ss_n$ can be encoded by the vector~$(|\b{x}|_p)_{p \in \P}$ where $|\b{x}|_p$ denotes the number of occurences of~$p$ in any representative~$\b{x} \in \bar{\b{x}}$.
These vectors generalize the barycentric coordinates described earlier to encode the join irreducible ASMs, Dyck paths, and Catalan triangles.

\begin{proposition}
The following assertions are equivalent for any~$\b{x}, \b{y} \in \P^n$:
\begin{enumerate}[(i)]
\item $\bar{\b{x}} \le \bar{\b{y}}$, \ie there exists~$\sigma \in \Ss_n$ such that~$\b{x} \le \b{y}^\sigma$ componentwise,
\item $\sum_{p \in U} |\b{x}|_p \le \sum_{p \in U} |\mathbf y|_p$ for all upper sets~$U$~of~$\P$,
\item there exists a nonnegative flow on the Hasse diagram of $\P$ with excess function ${p \mapsto |\b{x}|_p - |\mathbf y|_p}$.
\end{enumerate}
\end{proposition}

\begin{proof}
\underline{(i) $\Rightarrow$ (iii):}
Let~$\sigma$ be such that~$\b{x} \le \b{y}^\sigma$ componentwise.
For each~$i \in [n]$, let~$f_i$ denote the flow on the Hasse diagram of~$\P$ with~$1$ along an arbitrary directed path joining~$x_i$ to~$y_{\sigma^{-1}(i)}$ and $0$ elsewhere.
Then~$\sum_{i \in [n]} f_i$ is a nonnegative flow on the Hasse diagram of $\P$ with excess function $p \mapsto |\b{x}|_p - |\mathbf y|_p$.

\underline{(iii) $\Rightarrow$ (ii):}
Consider a nonnegative flow~$f$ on the Hasse diagram of $\P$ with excess function $p \mapsto |\b{x}|_p - |\mathbf y|_p$ and let~$U$ be an upper set of~$\P$.
Then~$\sum_{p \in U} |\b{x}|_p - |\mathbf y|_p$ equals the incoming flow of~$f$ to~$U$, hence it is nonnegative.

\underline{(ii) $\Rightarrow$ (i):}
Assume that $\sum_{p \in U} |\b{x}|_p \le \sum_{p \in U} |\mathbf y|_p$ for all upper sets~$U$~of~$\P$.
Consider the the bipartite graph $G \eqdef (X \sqcup Y, E)$ where $X \eqdef \set{(p,k)}{p \in \P \text{ and } 1 \le k \le |\b{x}|_p}$ while $Y \eqdef \set{(p,k)}{p \in \P \text{ and } 1 \le k \le |\b{y}|_p}$, and $E$ has an edge between $(p_1,k_1) \in X$ and $(p_2,k_2) \in Y$ if and only if~$p_1\le p_2$.
Consider an arbitrary subset~$W \subseteq X$, denote by~$N_G(W) \eqdef \bigcup_{x \in W} \set{y \in Y}{(x,y) \in E}$ the neighborhood of~$W$ in~$G$, and by~$U$ the upper set of $\P$ generated by elements $p \in \P$ such that there exists some $(p,i)\in W$.
Then
\[
|W| \le \sum_{p \in U} |\b{x}|_p  \le \sum_{p \in U} |\b{y}|_p = |N_G(W)|,
\]
where the second inequality holds by~(ii).
Hence, by application of Hall's perfect matching theorem, there is a matching of~$G$ covering~$X$.
This matching clearly defines a permutation~$\sigma \in \Ss_n$ with~$\b{x} \le \b{y}^\sigma$.
\end{proof}


\section*{Acknowledgments}

We are grateful to Lucas Gagnon, Jean-Christophe Novelli, Wenjie Fang, and
Jessica Striker for their insightful discussions and suggestions. We also
thank Romain Di Vozzo and Selena Pere for their helpful discussions and assistance with the 3D models.
We thank anonymous referees of the FPSAC'26 conference for comments on the presentation.

\bibliographystyle{alpha}
\bibliography{Quotients_ASMs}

\end{document}